\documentclass[11pt,oneside]{article}

\usepackage[round,colon,authoryear]{natbib} 

\usepackage{graphicx}
\usepackage{amsmath}
\usepackage{amsfonts}
\usepackage{amssymb}
\usepackage{amsthm}
\usepackage{enumerate}
\usepackage{pdfsync}
\usepackage{times}

\usepackage{bm}

    \numberwithin{equation}{section}

\newcommand{\e}{ \mathbb{E}}            
\newcommand{\p}{ \mathbb{P}}            
\newcommand{\R}{\mathbb{R}}            
\newcommand{\N}{\mathbb{N}}            

\newcommand{\F}{{\cal F}}         

\newcommand{\la}{\lambda}               

\newcommand{\Om}{\Omega}

\newcommand{\ud}{\mathrm{d}}

\DeclareMathOperator*{\grad}{grad}
\bmdefine\xib{\mathbf{\xi}}
\bmdefine\etab{\mathbf{\eta}}
\theoremstyle{plain}
\newtheorem{thm}{Theorem}[section]
\newtheorem{lem}[thm]{Lemma}
\newtheorem{prop}[thm]{Proposition}

\theoremstyle{definition}
\newtheorem{defn}{Definition}[section]

\theoremstyle{remark}
\newtheorem{rem}{Remark}[section]

\usepackage[dvipsname,usenames]{xcolor}
\definecolor{jan}{rgb}{0.0,0.3,0.8}
\definecolor{mat}{rgb}{0.0,0.5,0.3}

\newcommand{\ip}[1]{\langle {#1}\rangle}

\newcommand{\bip}[1]{\big\langle {#1}\big\rangle}
\newcommand{\Bip}[1]{\Big\langle {#1}\Big\rangle}

\newcommand{\bH}{\mathbb{H}}
\newcommand{\bL}{\mathbb{L}}

\newcommand{\cK}{\mathcal{K}}
\newcommand{\cL}{\mathcal{L}}
\newcommand{\cM}{\mathcal{M}}
\newcommand{\cP}{\mathcal{P}}

\newcommand{\cS}{\mathcal{S}}

\newcommand{\cX}{\mathcal{X}}

\newcommand{\cZ}{\mathcal{Z}}

\title{   \textsc{Trajectorial Dissipation  and Gradient Flow for the  Relative Entropy   in  Markov Chains}
\thanks{\,We are    indebted to Ioannis  Kontoyiannis,   Peter Michor,       Abishek  Tilva  and Lane  Yeung   for sharing their expertise on the subject,   for bringing relevant literature to our attention, and for their many helpful comments. 
 I.K.\,acknowledges support from the U.S.\,National Science Foundation under Grant NSF-DMS-20-04997. 
 J.M.\,acknowledges support from the European Research Council (ERC) under the European Union's Horizon 2020 research and innovation programme (grant agreement No 716117) and from the Austrian Science Fund (FWF) through project F65.  W.S.\,acknowledges support from the Austrian Science Fund (FWF) under grant P28861 and by the Vienna Science and Technology Fund (WWTF) through projects MA14-008 and MA16-021.  }
}

\author{  
\textsc{Ioannis Karatzas}     \thanks{\, 
Department of Mathematics,  Columbia University, 2990 Broadway, New York, NY\,10027, USA ({\it ik1@columbia.edu}),    and       \textsc{Intech} Investment Management,  One Palmer Square, Suite 441, Princeton, NJ 08542, USA    ({\it ikaratzas@intechjanus.com}). 
} 
\and
\textsc{Jan Maas}                 \thanks{\, Institute of Science and Technology 
(IST) Austria, Am Campus 1, 3400 Klosterneuburg, Austria (email: {\it jan.maas@ist.ac.at}).
}
  \and
\textsc{Walter Schachermayer}          
\thanks{\, 
Faculty of Mathematics, University of Vienna, Oskar-Morgenstern-Platz 1, 1090 Vienna, Austria \newline $~~~~~~~~ $ (email: {\it walter.schachermayer@univie.ac.at}).
               }
                                      }

\begin{document}

\maketitle

\begin{abstract}

\noindent 
\small We study the temporal dissipation of    variance and     relative entropy    for   ergodic  
 Markov  Chains in continuous time, and compute  explicitly the corresponding  dissipation rates. These are   identified, as is well known,  in the case of the variance in terms of an appropriate Hilbertian norm; and in the case of the relative entropy,  in terms of a   Dirichlet form   which morphs into a version of   the familiar  Fisher  information     under conditions of  detailed  balance. Here we obtain trajectorial versions of these  results,     valid along almost every path of the random motion and  most transparent in the backwards direction of time. Martingale arguments and time reversal play crucial roles, as   in the recent work of Karatzas, Schachermayer and Tschiderer  for conservative diffusions. Extensions   are    developed to general ``convex divergences"  and   to countable state-spaces. The  steepest descent and gradient flow properties for the variance, the relative entropy, and appropriate generalizations, are studied   along with their respective geometries 
 under conditions of   detailed balance,  leading to a very direct proof for the HWI inequality of Otto and Villani in the present context.    
\end{abstract}

\noindent{\it Keywords and Phrases:}  Markov Chain;  Relative Entropy; Time Reversal; Steepest Descent; Gradient Flow.

\noindent{\it AMS 2020 Subject Classifications:}    60J27; 60H10; 60G44; 46C05.

 \input amssym.def
\input amssym

 \medskip



 \section{Introduction and Summary}
 \label{sec0}

We present a trajectorial approach to the temporal dissipation of   variance and   relative entropy,  in the context of    ergodic  \textsc{Markov} Chains in continuous time. We follow the methodology   of the recent work by \textsc{Karatzas, Schachermayer \& Tschiderer} (2020),  which is based on stochastic calculus and uses  time-reversal in a critical fashion. By  aggregating   the trajectorial results, i.e., by averaging them  with respect to the invariant measure, we obtain a very crisp, geometric picture      of  the steepest descent property  for the  curve of   time-marginals, relative to local perturbations.  This holds for an appropriate, locally flat  metric on configuration space,  defined   in terms of a suitable discrete \textsc{Sobolev} norm.

We adopt then  a more  global  approach,  and establish also the  gradient flow  property---to the effect that the temporal  evolution    for the curve of the Chain's time-marginals is prescribed by  an appropriate  Riemannian metric  on the manifold of probability measures on configuration space, and by the differential of the relative entropy functional along this curve; cf.\,\textsc{Maas} (2011), \textsc{Mielke} (2011), \textsc{Erbar \& Maas} (2012, 2014). Both  steepest descent and gradient flow are manifestations of the seminal \textsc{Jordan,   Kinderlehrer  \& Otto}\,(1998) results  and of their  outgrowth, the so-called ``\textsc{Otto}   Calculus" initiated in \textsc{Otto} (2001).

 \smallskip
\noindent
{\it Preview:} For a    finite state-space, we set up the probabilistic framework in Section \ref{sec1} and the functional-analytic one in Section \ref{sec3}. The appropriate stochastic-analytic machinery and results appear in Sections \ref{sec2} and \ref{sec4}. Temporal dissipation and steepest descent are   developed in increasing generality:  First in Section \ref{sec5} for the variance and its associated, globally determined and flat, metric; then in Section \ref{sec6} for the \textsc{Boltzmann-Gibbs-Shannon} relative entropy; and finally in Section \ref{sec7} for general entropies induced by convex functions. Gradient flows and their associated   geometries  are taken up in Section \ref{sec8}, culminating with a very direct proof of   a discrete version of   the celebrated HWI inequality of \textsc{Otto \& Villani} (2000). Some extensions to  state-spaces with a countable infinity of elements are developed in Section \ref{sec9}.

\section{The Setting}
\label{sec1}

On a   probability space $  ( \Om,   \F, \p) , $ we start with an irreducible, positive recurrent, discrete-time \textsc{Markov} Chain $\, {\cal Z}=    (Z_n  )_{n \in \N_0}\,$ with  state-space $\,{\cal S} ,$  transition probability matrix $ \,\Pi  =  ( \pi_{xy}  )_{(x,y) \in {\cal S}^2}\, $  with entries $  \pi_{xy} = \p (Z_{n+1} =y \,|\, Z_n =x)\, $ for  $\,n \in \N_0 ,$ and  initial distribution  $\,P(0) =  ( p (0, x)  )_{x \in {\cal S}}\,$ which is a column vector with components $\,p (0,x) : = \p (Z_0 =x) >0\,$ for all $\,x \in {\cal S}.$ 
 Throughout Sections \ref{sec1}--\ref{sec8}, the state-space $\,{\cal S}\,$ is     assumed to be   {\it finite;}  extensions to  countable state-spaces  are taken up in  Section \ref{sec9}. 

\smallskip
It is  straightforward to check that the  sequence of random variables $\,\big(M^f_n \big)_{n \in \N_0}\,$ with  $\, M^f_0 :=f(Z_0)\,,$  
\begin{equation} 
\label{A1}
M^f_n \,:=\, f(Z_n)   - \sum_{k=0}^{n-1} \big( \Pi f - f \big) (Z_k)\,, \qquad n \in \N\,,
\end{equation} 
is a martingale of the filtration generated by the \textsc{Markov} Chain $ {\cal Z}$,  for any given function $\, f : {\cal S} \to \R\,.$ \,Here and in what follows, we denote $\,   ( \Pi f   ) (z)  :=  \sum_{y \in {\cal S}} \, \pi_{zy} \, f(y) ,    ~ z \in {\cal S}$.

It is well known that such a Chain has a unique    invariant distribution:  that is, a  column    vector  $\, Q = \big( q (y) \big)_{y \in {\cal S}}\,$ of {\it positive} numbers adding up to 1  and satisfying  $\, \Pi^\prime  Q   = Q\,$ or, more explicitly, 
\begin{equation} 
\label{A18}
q (y) \,=\, \sum_{z \in {\cal S}} \, q(z) \, \pi_{zy} \,,    \qquad \forall ~~y \in {\cal S}\,.
\end{equation}
 Here and throughout this paper, prime $\,^\prime\,$ denotes transposition of a matrix or vector.  A major result of discrete-time \textsc{Markov} Chain theory states that, when $ {\cal Z}$ is also aperiodic,   the $k-$step transition probabilities 
\begin{equation} 
\label{A2a}
\pi^{(0)}_{xy} := \mathbf{ 1}_{x=y}\,, \qquad \pi^{(k)}_{xy} := \p \big( Z_k =y \, \big|\, Z_0 =x\big)\,, \quad k \in \N
\end{equation}
converge   as $k$ tends to infinity to $ \,q(y) $, for every pair of states $\, (x,y) \in {\cal S}^2$. We  refer to Chapter 1 in \textsc{Norris} (1997), in particular Theorems 1.7.7 and 1.8.3, for an excellent account of the relevant theory.

\subsection{From Discrete- to Continuous-Time  \textsc{Markov} Chains, via \textsc{Poisson}}

  Consider now on the same probability space   a \textsc{Poisson} process $\, {\cal N}=  \big(N(t)  \big)_{0 \le t < \infty}\,$ with parameter $\, \la =1\,$ and {\it independent} of the discrete-time \textsc{Markov}  Chain $\, {\cal Z}.$  We construct via time-change  the continuous-time  process
\begin{equation} 
\label{A3}
X(t) \,:=\, Z_{N(t)}\,, \qquad 0 \le t < \infty \,,
\end{equation}
as well as the filtration $\,\mathbb{F}^X = \big\{ {\cal F}^X (t) \big\}_{0 \le t < \infty}\,$  this process generates via $\, {\cal F}^X (t) := \sigma \big( X(s), \, 0 \le s \le t \big) .$ 
Straightforward computation shows that this new, continuous-time process $\, {\cal X}=  \big(X(t)  \big)_{0 \le t < \infty}\,$ has the \textsc{Markov} property, and time-homogeneous transition probabilities 
\begin{equation} 
\label{A2}
\varrho_h (x,y) \, :=\,    \p \big( X(t+h) =y \, \big|\, X(t )=x\big) = e^{-h} \sum_{k \in \N_0} \frac{\, h^k\,}{k!}\, \pi^{(k)}_{xy}\,, \qquad t \ge 0, ~ h > 0
\end{equation}
with the notation of (\ref{A2a}); we set  $\, \varrho_0 (x,y) := \mathbf{ 1}_{x=y}\,$. The   functions $\, h \mapsto \varrho_h (x,y) \,$ in (\ref{A2}) are uniformly continuous and continuously differentiable; cf. Theorems 2.13, 2.14 in \textsc{Liggett} (2010).

\smallskip
More generally, for arbitrary  $\, n \in \N$, $\, 0 < \theta_1 < \cdots < \theta_n = \theta < t < \infty\,$, $(x, y_1, \cdots, y_n, z) \in {\cal S}^{n+2}\,$ with $\, y = y_n\,,$   the finite-dimensional distributions of this process are 
\begin{equation} 
\label{A5}
\p \big( X(0 )=x , X(\theta_1) = y_1, \cdots , X(\theta_n)= y_n, X(t) =z   \big) \,=~~~~~~~~~~~~~~~~~~~~
\end{equation}
$$ 
~~~~~~~~~~~~~~~~~~~~~~~~\,= \,p (0,x)\, \varrho_{\theta_1} (x,y_1)\, \varrho_{\theta_2 - \theta_1} (y_1,y_2)\cdots \varrho_{\theta_n - \theta_{n-1}} (y_{n-1},y_n) \cdot \varrho_{t - \theta}  (y ,z)
$$
and we deduce  the time-homogeneous \textsc{Markov} property
\begin{equation} 
\label{A5a}
\p \big(  X(t) =z \, \big|\, {\cal F}^X (\theta )   \big)=\varrho_{t - \theta}  \big(X(\theta) ,z \big)=\p \big(  X(t) =z \, \big|\, X (\theta )   \big) .
\end{equation}

Finally, from the \textsc{Chapman-Kolmogorov} equations $\,\pi^{(m+n)}_{xy}=\sum_{z \in {\cal S}} \, \pi^{(m)}_{xz} \pi^{(n)}_{zy} \,$ for the $k-$step transition probabilities of $  \,{\cal Z} \, $ in (\ref{A2a}), we deduce these same equations   for the quantities in (\ref{A2}):  
\begin{equation} 
\label{A8}
\varrho_{t+\theta} (x,y) \,  =\,     \sum_{z \in {\cal S}} \,  \varrho_{ \theta} (x,z) \,  \varrho_{t } (z,y) \, , \qquad (\theta, t) \in [0, \infty)^2, ~~ (x,y) \in {\cal S}^2
\end{equation}
 Here we think of the temporal argument $\theta$ as the ``backward  variable", and of $t$ as the ``forward  variable".

\subsection{Infinitesimal  Generators  and Martingales}

We introduce now the matrix
\begin{equation} 
\label{A10}
{\cal K} := \,\Pi -  \mathrm{I}\, = \big\{ \kappa (x,y) \big\}_{(x,y) \in {\cal S}^2} \qquad \text{with elements} \qquad \kappa (x,y) := \pi_{xy}- \mathbf{ 1}_{x=y}\,:
\end{equation}
non-negative off the diagonal, adding up to zero across each row. From (\ref{A2}) and with the help of time-homogeneity, we obtain  for $\, t \ge 0\,,$ $\, h >0\,$   the infinitesimal ``transition rates" 
\begin{equation} 
\label{A12}
\p \big(  X(t+h) =y \, \big|\, X (t)  =x \big)\,=\, h \cdot \kappa (x,y) + o (h)\,, \qquad x \neq y\,,
\end{equation}
\begin{equation} 
\label{A12'}
\p \big(  X(t+h) =x \, \big|\, X (t)  =x \big)\,=\, 1+ h \cdot \kappa (x,x) + o (h) 
\end{equation}
 \noindent
with the standard  convention $\, \lim_{h \downarrow 0 } \big( o (h) / h \big) =0,$ valid uniformly over $t \in [0, \infty).$ In particular, (\ref{A12})  and (\ref{A12'})   give  the infinitesimals $\,  \varrho_{h} (x,y) - \varrho_{0} (x,y) = h \cdot \kappa (x,y) + o (h)\,$ for all $\,(x,y) \in {\cal S}^2\,$, and thus 
\begin{equation} 
\label{A9}
 \partial   \varrho_h (x,y) \, \big|_{h=0} \,=\, \kappa (x,y)\,.
\end{equation}
{\it Here and throughout the paper,   $\, \partial   g  \,$ denotes   partial differentiation  of a function $g$ 
with respect to its temporal argument.}

  A bit more generally, for any   $  f : {\cal S} \to \R $ we have from (\ref{A12}), (\ref{A12'})  the semigroup computation
\begin{equation} 
\label{A12''}
\big( T_h f \big) (x)  \,:=\,
\e \big[ f \big( X(t+h)) \, \big|\, X(t)=x \big]    \,=\, f(x)+ h \cdot \big( {\cal K} f \big) (x) + o (h) \,.
\end{equation}
We deploy, here and in what follows, the {\it infinitesimal generator} of the Chain, i.e., the linear operator
\begin{equation} 
\label{A26}
\big( {\cal K} f \big) (x)   := \big( \Pi f \big) (x)- f(x) =  \sum_{y \in {\cal S}} \, \kappa (x,y)\, f(y)= \sum_{y \in {\cal S}} \, \kappa (x,y)\, \big[ f(y) - f (x) \big]  \,, \quad x \in {\cal S}  \,.
\end{equation}
\newpage
\noindent
Using the computation \eqref{A12''}, it is shown fairly easily that the exact analogue of the random sequence    (\ref{A1}) in our present setting, namely, the  process 
\begin{equation} 
\label{A25}
f \big( X(t) \big) - \int_0^t \big( {\cal K} f \big) \big(  X (\theta) \big) \, \ud \theta \,, \qquad 0 \le t < \infty\,,
\end{equation}
 is an $\,\mathbb{F}^X-$martingale; cf.\,Theorem 3.32 in \textsc{Liggett}  (2010).    As a slight   generalization, we obtain   also the following result (Lemma IV.20.12 in \textsc{Rogers  \& Williams}  (1987)).  
 
 \begin{prop}
 \label{Liggett}
Given any function $\, g : [0, \infty) \times {\cal S} \to \R\,$ whose  temporal derivative  $ \,  t   \mapsto   \partial g (t, x)\, $  is continuous for every state $\, x \in {\cal S} ,$ the process below   is a local $\,\mathbb{F}^X-$martingale:
 \begin{equation} 
\label{A27}
M^g  (t)  \,:= \, g \big( t,X(t) \big) - \int_0^t \big( \partial g + {\cal K} g \big) \big(\theta, X (\theta) \big)\, \ud \theta \,, \qquad 0 \le t < \infty \,.
\end{equation}
 \end{prop}
 
 \begin{rem}
 \label{1.1}
{\it The General Case:}  Instead of starting with    transition probabilities $\pi_{x y}$ and defining   $ \kappa (x,y)  = \pi_{xy}- \mathbf{ 1}_{x=y}\,$ as in (\ref{A10}), one can  work instead with {\it any  transition rates} $\kappa(x,y)$ satisfying: $(i)$ $\kappa(x,y) \geq 0\, $ for $ \,x \neq y\,$; and $(i)$ $\sum_{\,y \in {\cal S}} \kappa(x,y) = 0\,$ for every  $x \in {\cal S}$. In this   manner, arbitrary  irreducible  continuous-time \textsc{Markov} chains on finite state spaces can be constructed, and studied with little extra effort.   We have opted here for the  somewhat less general, but  very concrete and intuitive, approach of the present Section. 
\end{rem}

\section{Forward and Backward \textsc{Kolmogorov} Equations}
\label{sec2}

 Let us differentiate both sides of  the equations in  (\ref{A8}) with respect to the  backward  variable $\theta$, then set  $\theta =0$. We obtain on account of (\ref{A9}) the {\it Backward  \textsc{Kolmogorov} differential equations} 
\begin{equation} 
\label{A13}
 \partial  \varrho_t (x,y) \,=\, \sum_{z \in {\cal S}} \, \kappa (x, z) \, \varrho_t (z,y).
\end{equation}
 We can write this system of equations, for the matrix-valued function 
 $\, t \mapsto 
 {\cal P}_t = \big( \varrho_t (x,y) \big)_{(x,y) \in {\cal S}^2}\, 
 $  of the forward variable $\, t \in [0, \infty)\,,$   in the   form $\,
 \partial \, {\cal P}_t\,=\, {\cal K } \,{\cal P}_t\,, \,~{\cal P}_0 = \mathrm{I}\,.$

 In a similar manner, differentiating formally the equations (\ref{A8}) with respect to the  forward  variable $t$, then evaluating at $t=0$ and recalling the {\it transpose} 
 \begin{equation} 
\label{A23a}
 {\cal K}^\prime   \,:=\, \big(  \kappa^\prime (   y,z) \big)_{( y,z) \in {\cal S}^2}\,, \qquad  \kappa^\prime (   y,z) \,:=\,   \kappa (z,y) 
\end{equation} 
of the $\,{\cal K}-$matrix, we obtain the {\it Forward  \textsc{Kolmogorov} equations}
 \begin{equation} 
\label{A14}
 \partial \varrho_\theta (x,y) \,=\, \sum_{z \in {\cal S}} \, \varrho_\theta (x,z)\, \kappa (  z, y) \,=\, \sum_{z \in {\cal S}} \, \kappa^\prime (    y,z) \, \varrho_\theta (x,z)  \, ,\qquad  \text{or} \qquad   
 \partial \,{\cal P}_\theta\,=\, {\cal K }^\prime {\cal P}_\theta \,, \quad {\cal P}_0 = \mathrm{I} .
\end{equation}

\subsection{A Curve of Probability Vectors}
\label{sec2.1}

For every $\,t > 0$, let us consider the  column   vector $\,P(t) = \big( p (t, y) \big)_{y \in {\cal S}}\,$ of probabilities for the $\p-$distribution
 \begin{equation} 
\label{A19}
 p(t,y) := \p \big( X(t) = y\big) = e^{-t} \sum_{x \in {\cal S}} \,p(0,x) \sum_{k \in \N_0} \frac{\, t^k\,}{k!}\, \pi^{(k)}_{xy}\,>\,0 
\end{equation}
  of the random variable $X(t)$. The forward  \textsc{Kolmogorov} equations of (\ref{A14}), the law of total probability, and the \textsc{Markov} property, show that these   satisfy their own   forward  \textsc{Kolmogorov}  equations, namely 
\begin{equation} 
\label{A20}
 \partial  p (t,y) \,=\, \sum_{z \in {\cal S}}\, p (t,z) \, \kappa  (     z,y) \,=\, \sum_{z \in {\cal S}} \, \kappa^\prime (    y,z)\, p (t,z)   \,=: \, \big( {\cal K}^\prime p \big) (t,y)\,;
\end{equation}
or, more compactly and in matrix form, $\, \partial P(t)= {\cal K}^\prime P(t)  \,,~ ~0 \le t < \infty\, $ in the notation of (\ref{A23a}). {\it We shall think of  $   ( P (t)  )_{0 \le t < \infty}\,$ as a curve on the manifold $\, {\cal M  } = {\cal P}_+ ({\cal S})\,,$ of  vectors $\, P =  ( p (x)   )_{x \in {\cal S}}\,$ 
with strictly positive elements and total mass $\, \sum_{x \in {\cal S}} p (x)=1, $     viewed as  probability measures and governed by (\ref{A20}).}

\smallskip
Suppose that the initial distribution $P(0)$ of the discrete-time \textsc{Markov} Chain $\,{\cal Z}$ coincides with the column   vector  $  \,Q = \big( q (y) \big)_{y \in {\cal S}}\,$ of (\ref{A18})  satisfying  $  \,\Pi^\prime Q = Q\,,$   or equivalently $\,  {\cal K}^\prime Q=0 $ on account of (\ref{A10}). It follows  that $\, P(t) \equiv  Q\,,$ $\, \forall  \, t \in [0, \infty)\,$  provides  now the   solution of   (\ref{A20}): the distribution $Q$ is invariant also for the continuous-time \textsc{Markov} Chain $\,{\cal X}$ in (\ref{A3}). 

A bit more generally, $\,Q\,$ is the {\it equilibrium distribution} of $\,{\cal X}$, in the sense that for every    initial distribution $\,P(0)= \big( p (0, x) \big)_{x \in {\cal S}} \,  $ and  function $\, f : {\cal S} \to \R\,$  we have the limiting behavior 
\begin{equation} 
\label{A20a}
 \lim_{t \to \infty} p(t,y)\, = \,q(y)\,,\qquad \forall ~~ y \in {\cal S} ,
 \end{equation}
\begin{equation} 
\label{A20aa}
  \lim_{T \to \infty} \frac{1}{T} \int_0^T f \big( X(t) \big) \, \ud t \,=\, \sum_{y \in {\cal S}} \, q(y)\, f(y)\,, \quad \p-\text{a.e.;}
   \end{equation}
 see Sections 3.6\,--3.8 in \textsc{Norris} (1997) for an account of these results. In the present,  continuous-time context, aperiodicity plays no role.

\subsection{A Curve of Likelihood Ratios}
\label{sec2.2}

Let us compare now the components of the probability vector $P(t)$ in (\ref{A19}), with those of  the invariant probability vector $\,Q\,$ in (\ref{A18}). One way to do this, very fruitful in the present context, is  by considering the likelihood ratio column   vector 
 \begin{equation} 
\label{A21}
{\bm \ell }_t \equiv  {\bm \ell }  (t)=  \big( \ell (t, y) \big)_{y \in {\cal S}} \qquad \text{with components} \qquad  \ell(t,y)\,  := \frac{\,p(t,y)\,}{q(y)} \,.
\end{equation}
Substituting the product $\, p (t,y) = \ell (t,y) \,q (y)\,$ into the   forward  \textsc{Kolmogorov} equation (\ref{A20}), we obtain for the likelihood ratios of (\ref{A21})  the {\it Backward Equation}  
\begin{equation} 
\label{A22}
 \partial \ell (t,y) \,=\, \sum_{z \in {\cal S}}\, \widehat{\kappa} (   y,z) \, \ell (t,z)  \,=\, \sum_{z \in {\cal S}}\, \widehat{\kappa} (   y,z) \, \big[ \ell (t,z) - \ell (t,y) \big]   \,=: \, \big(  \widehat{{\cal K}} \, {\bm \ell }  \big) (t,y)\,,
\end{equation}
or equivalently $\,   \partial   {\bm \ell }  (t ) =   \widehat{{\cal K}} \, {\bm \ell } (t)\, $  in matrix form, with the new transition rates
\begin{equation} 
\label{A23}
  \widehat{{\cal K}}   \,:=\, \Big( \widehat{\kappa} (   y,z) \Big)_{( y,z) \in {\cal S}^2}\,, \qquad \widehat{\kappa} (   y,z) \,:=\, \frac{\, q(z)}{q(y)} \, \kappa (z,y)\,.
\end{equation}
The entries of this matrix $  \widehat{{\cal K}} $ are non-negative off the diagonal, and add up to zero  $\, \sum_{z \in {\cal S}} \widehat{\kappa} (   y,z)=0\,$ across every row $y \in {\cal S} ,$ on account of (\ref{A18}), (\ref{A10}).

{\it We shall think   of $\, ( {\bm \ell }  (t)  )_{0 \le t < \infty}\,$ as a curve, now  in the space ${\cal L = L_+ (S)}$ of    vectors $\, \Lambda =  ( \lambda (x)  )_{x \in {\cal S}}\,$ 
with strictly positive elements and   $\, \sum_{x \in {\cal S}} q (x)\,\lambda (x)=1.$}  These are  viewed as  likelihood ratios with respect to the invariant distribution and evolving in time via   (\ref{A22}).

 \smallskip
Presently, we shall identify $\,  \widehat{{\cal K}}\,$ of (\ref{A23}) with the infinitesimal generator of a suitable continuous-time \textsc{Markov} Chain, run {\it backwards} in time. A   special case, however, is worth mentioning already.

\begin{defn} {\bf Detailed Balance:} 
\label{DB}
The invariant distribution $\,Q\,$ in (\ref{A18}) is said to satisfy the {\it detailed-balance} conditions, if 
\begin{equation} 
\label{A24}
q(y) \, \kappa (y,z) \,=\, q(z) \, \kappa ( z,y) \, ,  \quad \forall ~ (y,z) \in {\cal S}^2.
\end{equation}
\end{defn}

 This   requirement turns out to be equivalent to the identity $\,q(y) \, \varrho_t (y,z)= q(z) \,  \varrho_t ( z,y)$   for all $ t \in (0, \infty), ~(y,z) \in {\cal S}^2 \,;$
 one leg of the equivalence is immediate, courtesy of  (\ref{A9}). When (\ref{A24}) prevails,   $\,   \widehat{{\cal K}} \equiv    {\cal K}\,$ holds in (\ref{A23}); and the backward equation (\ref{A22}) for the   likelihood ratios  $\,  ( \ell_t (x) )_{x \in {\cal S}}\,$  of (\ref{A21}), 
 is  then \\
  \newpage
 \noindent
    exactly the same as  the backward equation (\ref{A13})   for $\,  ( \varrho_t ( x, y) )_{x \in {\cal S}}\,$.    
 We stress that, whenever the   detailed-balance conditions  (\ref{A24}) are needed in the sequel,  they will be invoked explicitly.

\section{Discrete Gradient and Divergence; \textsc{Dirichlet} Form,   \textsc{Hilbert} Norms}
\label{sec3}

It is apt  at this point  to introduce some necessary  notation   and functional-analytic notions.     For a given function $\,f : {\cal S} \to \R\,$ we  consider the {\it discrete gradient} $\,\nabla f: {\cal S}^2 \to \R\,$ given  by  
 \begin{equation} 
\label{D1}
\nabla f (x,y): = f(y) - f(x)\,. 
\end{equation}
In a similar spirit, we consider   the {\it discrete divergence} 
 \begin{equation} 
\label{D2}
\big(\nabla \cdot F\big) (x ) \,: =\, \frac{\,1\,}{2} \sum_{y \in {\cal S}, \, y \neq x}   \kappa (x,y)\,\big[ F(x,y) - F(y,x)  \big]  
\end{equation}
of a function $\, F : {\cal S} \times {\cal S} \to \R\,,$ and note the   familiar {\it concatenation}   formula
  \begin{equation} 
\label{D3}
{\cal K} f = \nabla \cdot \big( \nabla f \big)  
\end{equation}
which allows us to think of the operator ${\cal K}$ in (\ref{A26}) also as a ``discrete Laplacian". We   introduce also the set 
$\, \cZ := \{(x,y) \in \cS \times \cS   \ : \   \kappa(x,y) > 0\}\,$ consisting of all edges in the incidence graph associated with the \textsc{Markov} chain, and the measure $\,C$ on $\,\cZ$ defined by the ``conductances"
 \begin{equation} 
\label{C_Meas}
	C\{(x,y)\} \equiv c(x,y) := \frac12 \,\kappa (x,y)\,q(x) \,, \qquad (x,y) \in \cZ.
\end{equation}
With these ingredients, we consider the bilinear forms  
\begin{equation} 
\label{D4}
  \bip{    f ,  g }_{\bL^2(\cS, Q)} \,:=\, \sum_{x \in {\cal S}} \, q(x)\,f(x) \,g(x)\,, 
  	\qquad  
	\bip{    F ,  G }_{\bL^2(\cZ, C)} \,:= \,
	\sum_{(x,y) \in {\cZ} } \, c(x,y) \,F(x,y)\, G(x,y) 
\end{equation} 
for real-valued functions defined on ${\cal S}$ (lowercase $f,\,g$) and on ${\cal S} \times {\cal S} $ (uppercase $F,\,G$), respectively. They induce the $\bL^2-$norms $ \big\| f \big\|_{\bL^2(\cS, Q)} $ (relative to the probability measure $Q$) and    $ \big \| F \big \|_{\bL^2(\cZ, C)} $ (relative to the unnormalized measure $C$ on $\cZ$ in (\ref{C_Meas})), given respectively via  
\begin{equation}
\begin{aligned}
\label{D4a}
	\big\| f \big\|_{\bL^2(\cS, Q)}^2\, 
	&:= \bip{ f ,  f }_{\bL^2(\cS, Q)} 
	 = \sum_{x \in {\cal S}} \, q(x)\,f^2(x)\,, \\    
	\big \| F \big \|_{\bL^2(\cZ, C)}^2\, 
	&:= \bip{   F ,  F }_{\bL^2(\cZ, C)} 
	 = \sum_{(x,y) \in {\cZ} } \, 
	 c(x,y) \,F^2(x,y)\,.
\end{aligned}
\end{equation}

 \begin{rem}
 \label{Rem_2.1}
  We note   from (\ref{A22})-(\ref{A23}) the adjoint relationship
\begin{equation} 
\label{AdjRel}
\bip{ f,  \widehat{{\cal K}}  g   }_{\bL^2(\cS, Q)} \,=\, \bip{    {\cal K}  f, g   }_{\bL^2(\cS, Q)}\,.
\end{equation}
 Thus (\ref{A24}) holds if, and only if, the operator ${\cal K}$ in (\ref{A26}) is self-adjoint on   $ \mathbb{L}^2 ({\cal S}, Q).$   
 \end{rem}

Finally,   we introduce   the bilinear {\it \textsc{Dirichlet}  form} associated with the \textsc{Markov} Chain: 
\begin{equation} 
\label{A26a}
{\cal E} (f,g) 
     \, := \, 
   - \bip{ f,  {\cal K}  g   }_{\bL^2(\cS, Q)}
    \,= \, 
   -  \sum_{y \in {\cal S}} \, q(y)\, f(y) \, \big( {\cal K} g \big) (y) 
   	\,=\, 
	-  \, \sum_{x \in {\cal S}} \sum_{y \in {\cal S}}  \,  \, q(y) \, \kappa (y,x) \, f (  y) \, g ( x). 
\end{equation}
This form is not symmetric, in general; but satisfies $\,{\cal E} (f,f)\ge 0\,,$ as follows from  Lemma \ref{DiscreteGradient} below.

\begin{lem}
\label{DiscreteGradient}
The \textsc{Dirichlet} form (\ref{A26a}) can be cast equivalently as  
\begin{equation} 
\label{E1}
{\cal E } (f,g) = \frac{1}{\,2\,}  \,  \sum_{x \in {\cal S}} \sum_{y \in {\cal S}}  \, \kappa (y,x) \, q(y)  \, \big( f ( y) -g ( x) \big)^2\,.
\end{equation}
 \end{lem}
 \newpage
 \noindent
  {\it Proof:}  We have clearly 
$\,
\sum_{x \in {\cal S}} \sum_{y \in {\cal S}}  \, \kappa (y,x) \,  q(y)   f^2 ( y)=0
\,$
 on account of $\, \sum_{x \in {\cal S}} \, \kappa  (    y,x )=0\,$   for every $y \in {\cal S} \, $; as well as 
 $$
 \sum_{x \in {\cal S}} \sum_{y \in {\cal S}}  \, \kappa (y,x) \, q(y)  \,  g^2 ( x)\,=\, \sum_{x \in {\cal S}} \sum_{y \in {\cal S}}  \, \widehat{\kappa} ( x,y) \, q(x)  \,   g^2 ( x) \,=\,0\,,
 $$
from  the adjoint rates of (\ref{A23}) and   their  property $\, \sum_{y \in {\cal S}} \,\widehat{\kappa} ( x,  y )=0\,,$ $\, \forall~x \in {\cal S} .$ It follows from (\ref{A26a})  that 
$$
~~~~~~~~~~ ~~\sum_{x \in {\cal S}} \sum_{y \in {\cal S}}  \, \kappa (y,x) \, q(y)  \, \big( f ( y) -g ( x) \big)^2\,= \,   - 2\,\sum_{x \in {\cal S}} \sum_{y \in {\cal S}}  \, \kappa (y,x) \, q(y)  \,   f (  y)   \, g ( x)
\,= \,2\, {\cal E } (f,g) \,. \qquad \qed
$$

\subsection{Consequences of Detailed Balance }

  The detailed-balance  conditions (\ref{A24}) can be thought of as   positing  that ``the conductances of (\ref{C_Meas})  do   not depend on the  direction of the current's flow".   Under these conditions,     
 we have  for functions $\,f : {\cal S} \to \R\,$ and $\, F : {\cal S} \times {\cal S} \to \R\, $    the   {\it discrete integration-by-parts} formula
\begin{equation} 
\label{D6}
  \bip{ \, \nabla f ,  F }_{\bL^2(\cZ,C)}
  \,=\, 
  -   \,\bip{ f , \nabla  \cdot  F }_{\bL^2(\cS,Q)} ,
\end{equation}
  in addition to the concatenation property (\ref{D3}).   As a result, {\it the bilinear \textsc{Dirichlet}  form of (\ref{A26a}), (\ref{E1})   is now symmetric,} and induces the {\it \textsc{Hilbert}   $\, \mathbb{H}^{ 1}-$inner product and norm}
\begin{equation} 
\label{A26b}
	\bip{ f ,    g}_{\bH^1(\cS,Q)} \,:=\, {\cal E} (f,g) \,=\,   \Bip{\nabla f ,  \nabla g}_{\bL^2(\cZ,C)}\,, 
\end{equation}
\begin{equation} 
\label{D5}
 \big \|  f  \big \|^2_{\bH^1(\cS,Q)}
 	 \,:=  \,  
	  {\cal E} (f,f)
	   \,=\,
	  - \big \langle f,  {\cal K} f   \big \rangle_{\bL^2(\cS,Q)} 
	  \,= \, 
		   \sum_{(x,y) \in {\cal Z} }  \, c(x,y)
	   \, \big( f ( y) - f ( x) \big)^2
		  \,=\,    \big \| \nabla f  \big \|_{\bL^2(\cZ,C)}^2\,,
\end{equation}
respectively. We introduce also the dual of this norm, the {\it \textsc{Hilbert}   $\, \mathbb{H}^{-1}-$norm } 
\begin{equation} 
\label{D10}
 \big \|  f  \big \|_{\bH^{-1}(\cS,Q)} 
 	\,:=   \,
	    \big \| \nabla \big({\cal K}^{-1} f \big)  \big \|_{\bL^2(\cZ,C)}\,, ~~~~ \text{if} ~ f \in \text{Range} ({\cal K})\,; \qquad  \big \|  f  \big \|_{\bH^{-1}(\cS,Q)} \,:=   \, + \infty\,,~~ ~~ \text{otherwise;}
\end{equation}
and note the variational characterizations 
\begin{equation} 
\label{D11}
 \big \|  f  \big \|_{\bH^{-1}(\cS,Q)} 
 \,=\,  \sup_{g : {\cal S} \to \R} \frac{\,\big \langle f, g \big \rangle_{\bL^2(\cS,Q)}\,}{\, \big \|  g  \big \|_{\bH^1(\cS,Q)}\,} \,,
 \end{equation}
 \begin{equation} 
 \label{D12}
 \big \|  f  \big \|_{\bH^{-1}(\cS,Q)}
 	\,=\,
	 \inf_{F : {\cal Z} \to \R} \big\{ \big \| F \big \|_{\bL^2(\cZ,C)}: f = \nabla \cdot F \big\} \,=\, \inf_{g : {\cal S} \to \R} \big\{ \big \| \nabla g\big  \|_{\bL^2(\cZ,C)}: f = {\cal K} g \big\}\,.
 \end{equation}
Basic   \textsc{Hilbert} space theory shows that these two  infima are attained.

\begin{lem}
\label{DiscreteGradientToo}
Under  the     conditions of  (\ref{A24}), the  expression (\ref{E1}) for the \textsc{Dirichlet}  form becomes    
\begin{equation} 
\label{A50aaa}
{\cal E } (f,g)
   =  \frac{1}{\,2\,}   \sum_{x \in {\cal S}} \sum_{y \in {\cal S}}     \kappa (y,x)  q(y)   \big[ f ( y) -f ( x) \big]  \big[ g (  y) - g (  x) \big] =  
	\Bip{  \nabla f ,  \nabla g }_{\bL^2(\cZ,C)} 
   = \big \langle   f ,    g   \big \rangle_{\bH^1(\cS,Q)} \,. 
\end{equation}
\end{lem}

\noindent
{\it Proof:}  Let us write the double summation in the above display as
$$
\sum_{x \in {\cal S}} \sum_{y \in {\cal S}}  \, \kappa (y,x)   q(y)  \, \Big[ \, f ( y) \, g(  y)- f ( y) \,  g (  x) - f( x)  \, g (  y)+ f ( x) \,  g (  x)\, \Big]=
$$
$$
=\,- \,\sum_{x \in {\cal S}} \sum_{y \in {\cal S}}  \, \kappa (y,x)   q(y)  \, \Big[ \,  f ( y) \,  g (  x) + f ( x) \,  g (  y)\, \Big]\,=\, - \,2\,\sum_{x \in {\cal S}} \sum_{y \in {\cal S}}  \, \kappa (y,x)   q(y)  \,   f ( y) \,  g (  x) \,=\, - \, 2\, {\cal E } (f,g)\,.
$$
  \newpage
 \noindent
Here, the first equality uses   (\ref{A23}), as well as the properties $\, \sum_{x \in {\cal S}} \, \kappa  (    y,x )=0\,$   for every $y \in {\cal S}    ,$ and $\, \sum_{y \in {\cal S}} \, \widehat{\kappa}  (     x ,y )=0\,$   for every $x \in {\cal S}  ; $  whereas,  the second equality uses the   conditions (\ref{A24}), and the third equality is just (\ref{A26a}).     

This proves the first equality in (\ref{A50aaa}).  The second and third are just restatements of (\ref{A26b}).    \qed

\begin{rem} 
\label{Rem3.1}
{\it   Additional Consequences:}  
    It follows   from       (\ref{D6})--(\ref{D5})  that, under the detailed-balance  conditions (\ref{A24}),  the mapping 
  $$\, \nabla \,: ~ 
{\bH^1 (\cS, Q)} \to  
{\bL^2(\cZ,C)}
$$   
is   an isometric embedding. Whereas,  the discrete divergence mapping $\, \nabla \cdot \,$ in (\ref{D2}) is, up to a minus sign,  the adjoint of the mapping   $\, \nabla \,: ~  
\bL^2(\cS,Q) \to \bL^2(\cZ,C)$.  
\end{rem}

\begin{rem} 
\label{Rem3.2}
{\it A Counterexample.}    In the absence of detailed balance, the \textsc{Dirichlet} form $\,{\cal E } (f,g) $ is not   an inner product; indeed, Remark \ref{Rem_2.1} shows that there exist functions $f :{\cal S} \to \R\,,$  $\,g: {\cal S} \to \R $ with $\, {\cal E}(f,g)=- \big \langle f,  {\cal K}  g   \big \rangle_{\bL^2(\cS,Q)} \neq - \big \langle g,  {\cal K}  f   \big \rangle_{\bL^2(\cS,Q)}= {\cal E}(g,f)$.  An explicit example of this situation is provided by the matrix  
$$
{\cal K}\,=\,
\begin{pmatrix}
      \,-1 & 1 & 0~   \\
      \,0 & -1 & 1~  \\  
      \,1 & 0 & -1    ~ 
\end{pmatrix},
 $$
whose invariant distribution $\,Q =(1/3, 1/3, 1/3)$\, is uniform on the state space $\,{\cal S} = \{ 1, 2, 3\} \, $ and for which detailed balance fails. Whereas, with $\,f = \mathfrak{e}_1 = (1, 0, 0)\,$ and $\,g = \mathfrak{e}_2 = (0, 1,   0)\,$  the first and second unit row vectors, respectively, and noting  $\, 3\, {\cal E} (\varphi, \gamma) =  \varphi \, {\cal K}' \gamma'\,$ from (\ref{A26a}), we observe
 $$
3\, {\cal E} (f, g) = \big(1, 0, 0\big) \begin{pmatrix}
      \,  1 ~   \\
      \,  -1  ~  \\  
      \,  0      ~ 
\end{pmatrix} = \, -1\,, \qquad 
3\, {\cal E} (  g, f) = \big(  0,1, 0\big) \begin{pmatrix}
      \,  -1 ~   \\
      \,  0  ~  \\  
      \,  1      ~ 
\end{pmatrix} = \, 0\,.
 $$
 
   Nevertheless,  $\,\big \|  f \big  \|_{\bH^1(\cS,Q)}  = \sqrt{{\cal E} (f,f)\,}\,$ is {\it always} a \textsc{Hilbert} norm,   with   associated inner product  given by the \textsc{Dirichlet} form $\, {\cal E}_{\text{sym}} (f,g)\,$ of the reversible \textsc{Markov} Chain,  with symmetrized rates $\, \kappa_{\text{sym}} (x,y) := ( \kappa (x,y) + \widehat{\kappa} (x,y) ) /2\,$ in the manner of (\ref{A26a}), (\ref{A23}); namely, $\, {\cal E}_{\text{sym}} (f,f) \equiv {\cal E}  (f,f)\, $ and 
 $$
  \big \langle f,g    \big \rangle  _{\bH^1(\cS,Q)}   =   -\frac{1}{2} \sum_{x \in {\cal S}} \sum_{y \in {\cal S}} \, \big[ q(y) \kappa (y,x) + q(x) \kappa (x,y) \big]  f(x) g (y) =\,-      \sum_{x \in {\cal S}} \sum_{y \in {\cal S}}  \,  \, q(y) \, \kappa_{\text{sym}} (y,x) \, f (  y) \, g ( x) .
 $$ 
  \end{rem}

\section{Time Reversal  and   Associated Martingales}
\label{sec4}

 It is well known that the \textsc{Markov} property is invariant under reversal of time (interchanging the roles of ``past" and ``future", keeping the ``present" as is). This means, in particular, that the time-reversed  process 
 \begin{equation} 
\label{A36}
 \widehat{X} (s) : = X (T-s)\,, ~~ ~~ 0 \le s \le T\,
 \end{equation} 
  is a \textsc{Markov} Chain, for any given $T \in (0, \infty)$. {\it But   how about the transition probabilities of this time-reversed  process?} These are fairly easy to compute, namely,   
\begin{equation} 
\label{A38}
\p \big(  \widehat{X}(s_2) =z \, \big|\, \widehat{{\cal G}}  (s_1 )   \big)=\,\p \big(  \widehat{X}(s_2) =z \, \big|\, \widehat{X}(s_1)   \big) = \,\rho^*  \big(s_1, \widehat{X}(s_1) ; s_2, z \big)  
 \end{equation}
 for $\, 0 \le s_1 \le s_2 \le T $, $\,  z \in {\cal S} $, where
\begin{equation} 
\label{A38'}
\rho^*  \big(s_1, y ; s_2, z \big) \,:=\, \frac{\, p (T-s_2, z)\,}{p (T-s_1, y)} \,\, \varrho_{s_2 - s_1} \big( z, y\big)\,;
 \end{equation}
but need not be time-homogeneous in general.

 \newpage
 \noindent

However: {\it Let us compute these same transition probabilities when   the Chain starts at its invariant distribution $Q$.}   We introduce at this point another probability measure $\,\mathbb{Q}\,$ on the underlying measurable space $\, ( \Omega, \F),$ under which the \textsc{Markov} Chain $\, {\cal X}\,$ has exactly the same dynamics as before, but its initial   distribution is   the  invariant probability   vector  $   Q = \big( q (y) \big)_{y \in {\cal S}}\,$ in (\ref{A18}).  Then, in lieu of (\ref{A5}), the finite-dimensional distributions of the Chain are 
 \[ 
\mathbb{Q} \big( X(0 )=x , X(\theta_1) = y_1, \cdots , X(\theta_n)= y_n, X(t) =z   ) \,=~~~~~~~~~~~~~~~~~~~~~~~~~~~~~~~~~~~~
\] 
$$ 
~~~~~~~~~~~~~~~~~~~~~~~~~~~~~~~~~~~~~~~~= \,q ( x)\, \varrho_{\theta_1} (x,y_1)\, \varrho_{\theta_2 - \theta_1} (y_1,y_2)\cdots \varrho_{\theta_n - \theta_{n-1}} (y_{n-1},y_n) \cdot \varrho_{t - \theta}  (y ,z)\,.
$$
On each $\sigma-$algebra $\,\F^X (t),$  $\, 0 \le t < \infty ,$ the two probability measures $\p$ and $\mathbb{Q}$ are equivalent; in fact, on the smaller $\sigma-$algebra $\, \sigma (X(t)) \,,$ we single out in the notation of (\ref{A21}) the so-called {\it likelihood   process}
\begin{equation} 
\label{A34}
L (t)\,:=\,\frac{\,\ud \p\,}{\ud \mathbb{Q}}\bigg|_{\sigma (X(t)) } \,=\,  \ell \big( t, X(t) \big) \,, \qquad 0 \le t < \infty\,.
\end{equation}

  Under this dispensation, the   transition probabilities are  
 \begin{equation} 
\label{A39}
\mathbb{Q} \big(  \widehat{X}(s_2) =z \, \big|\, \widehat{{\cal G}}  (s_1 )   \big)=\, \mathbb{Q}\big(  \widehat{X}(s_2) =z \, \big|\, \widehat{X}(s_1)   \big) = \, \widehat{\varrho}_{s_2-s_1}   \big(  \widehat{X}(s_1)  , z \big)  ,
 \end{equation}
 i.e., {\it  time-homogeneous,}  with 
 \begin{equation} 
\label{A39'}
\widehat{\varrho}_{h} (y,z)   \,:=\, \frac{\, q (  z)\,}{q (y)} \,\, \varrho_{h} \big( z, y\big)\,.
 \end{equation}
Invoking (\ref{A39'}) and (\ref{A9}), we see   that the $\, \mathbb{Q}-$infinitesimal-generator of this time-reversed \textsc{Markov} Chain $   \widehat{X} (s)   = X (T-s),  \,~ 0 \le s \le T$  \,in (\ref{A36}), is given {\it precisely by     $\, \widehat{{\cal K}}    =   ( \widehat{\kappa} (   y,z )  )_{(y,z) \in {\cal S}^2}\,$ as in (\ref{A23}).}  (We note parenthetically that, when the    detailed-balance  conditions (\ref{A24}) hold, the initial distributions and transition probabilities of the continuous-time \textsc{Markov} Chain $  X(t),\,  0 \le t \le T  ,$ and   of its time-reversal (\ref{A36}), are {\it exactly the same}   under  the probability measure 
$\mathbb{Q}$.) 

\begin{rem}
\label{rem5.1}
  The standing  assumption $P(0) \in {\cal M},$ i.e., that all entries of the initial distribution are strictly positive, is made     for economy of exposition.  For even when the probability vector  $P(0)$ belongs to the closure $  \overline{{\cal M}} $ of $   {\cal M}  $, i.e.,  some of its entries   are allowed to vanish,  there is at least one $x \in {\cal S} $ with $p(0,x)>0$;   then (\ref{A19}) and  irreducibility  imply $p(t,y) >0$ for all  $t>0,\, y \in {\cal S}.$ Thus, even if the curve $   ( P (t)  )_{0 \le t < \infty}\,$ starts out on the boundary $ \, \overline{{\cal M}} \setminus {\cal M},$   it enters $  {\cal M} $  immediately and stays there for all times $\,t \in (0, \infty).$  
\end{rem}

  By complete analogy with Proposition \ref{Liggett}, we formulate now the following result.   

\begin{prop}
\label{Liggett_Back}
 For any given   function $\, g : [0, T] \times {\cal S} \to \R\,$ whose  temporal derivative  $ \,  s   \mapsto   \partial g (s, x) \, $   is continuous for every state $\, x \in {\cal S} ,$ \,the process   below is a $\,\big( \widehat{\mathbb{G}},  \mathbb{Q}\big)-\,$local martingale:     
 \begin{equation} 
\label{A41}
\widehat{M}^{\,g} (s) \, :=  \, g \big( s,\widehat{X}(s) \big) - \int_0^s \big( \partial g + \widehat{{\cal K}} g \big) \big(u, \widehat{X} (u) \big)\, \ud u \,, \qquad 0 \le s  \le T\,. 
\end{equation}
 \end{prop}

 The    following important result is due to \textsc{Pavon} (1989), \textsc{Fontbona  \&   Jourdain} (2016) in the context of diffusions. Its proof (cf.\,Theorem 4.2 in \textsc{Karatzas, Schachermayer \& Tschiderer} (2019)) uses only the \textsc{Markov} property and the definition of conditional expectation, and carries over verbatim to our present context. An alternative argument, specific to the \textsc{Markov} Chain context, uses Proposition    \ref{Liggett_Back} and is given right  below.   

\begin{prop}
\label{FJ}
 {\bf Time-Reversed Likelihood Process as Martingale:} 
Fix $\,T \in (0, \infty)\,$ and consider the time-reversed Chain (\ref{A36}),  as well as the filtration $\, \widehat{\mathbb{G}} = \big\{ \widehat{{\cal G}} (s) \big\}_{0 \le s \le T}\,$ this process generates via $\, \widehat{{\cal G}} (s) := \sigma \big( \widehat{X} (u), ~ 0 \le u \le s \big).$ Then, the time-reversed likelihood   process 
\begin{equation} 
\label{A40}
L(T-s)\,=\,   \ell \big( T-s, \widehat{X} (s) \big) \,, \quad 0 \le s \le  T~~~~~~~\text{is a}~\, \big( \widehat{\mathbb{G}}, \mathbb{Q}\big)- \text{martingale.}
\end{equation}
 \end{prop}

 \newpage
\noindent
 {\it Proof:} 
We consider in (\ref{A41}) the function 
$\, 
g(s, x) = \ell (T-s, x)\,, ~\, 0 \le s \le T, ~ x \in {\cal S}
\,$
 and note that $   \partial g(s, x) = - \partial \ell (T-s, x) = - \big( \widehat{{\cal K}} \, \ell \big) (T-s, x) \, $ holds on account of (\ref{A22}).

 It follows   from (\ref{A41}),  
 whose integrand now vanishes,  that the time-reversed   likelihood ratio process    $ \ell \big( T-s, \widehat{X} (s) \big)  \,, ~ 0 \le s \le  T \,$ is a $\, \mathbb{Q}-$local-martingale of the time-reversed   filtration $\,\widehat{\mathbb{G}}\,$.  But this process is positive, thus also a $  \mathbb{Q}-$supermartingale, and   its   expectation
$$
 \mathbb{E^Q} \big[ \ell \big(T-s, X(T-s) \big) \big]  \,=\, \sum_{y \in {\cal S}} \, q(y)\, \frac{\,p (T-s, y)\,}{q(y)}\,=\,1\,, \qquad 0 \le s \le T 
$$
  is constant. Therefore $  \,\ell \big( T-s, \widehat{X} (s) \big)  \,, ~ 0 \le s \le  T \,$ is a true $\, \mathbb{Q}-$martingale, exactly as stated in (\ref{A40}).  \qed

\section{The  Variance Process}
\label{sec5}
  
 For a probability vector $\, P =  ( p(y) )_{y \in {\cal S}} \in {\cal M}\,$ with   positive  entries,  we introduce its likelihood vector ${\bm \ell } =  ( \ell (y) )_{y \in {\cal S}} \in {\cal L}\,$ with $\,  \ell  (y) = p(y) / q(y)$ as in   (\ref{A21}), relative to   the invariant distribution $Q$ of the Chain. We define then in the manner of (\ref{D4a})  the {\it Variance} of $P$ relative to $Q,$   also known as {\it $\chi^2-$divergence,}  as 
\begin{equation} 
\label{A41b}
V \big(P\,|\,Q\big) \, \equiv  \,  \text{Var}^\mathbb{Q} \big( {\bm \ell} \big)\,  :=   \,   \sum_{y \in {\cal S}} \, q (y) \,\ell  ^2 ( y) -1 \,= \,\big| \big| {\bm \ell} \big| \big|_{\bL^2(\cS,Q)}^2 -1 \,.
\end{equation} 
Let us recall now from (\ref{A19})  the curve $  \big( P(t) \big)_{0 \le t < \infty} \subset {\cal M}\,$   of time-marginal distributions for our continuous-time \textsc{Markov} Chain, and the corresponding curve of likelihoods $\,{\bm \ell }_t =  ( \ell (t,y) )_{y \in {\cal S}}\,,~ 0 \le t < \infty\,$ in the space $ {\cal L}\,,$ with $\,  \ell  (t,y) = p(t,y) / q(y)$. We will show    in Proposition \ref{DE}   that the 
variance just defined in (\ref{A41b}) plays the role of   \textsc{Lyapunov} function for  the convergence to equilibrium along this curve.

\smallskip
 
To see this, we summon   the likelihood process $\, L(t) = \ell ( t, X(t)), ~ 0 \le t < \infty\,$ from (\ref{A34}) and consider its square  $\, L^2(t)  , ~ 0 \le t < \infty\,,$ the so-called  {\it  Variance Process,  under time-reversal.}  

\begin{prop}
\label{DE_Traj}
For any given $T\in (0, \infty),$ we have the \textsc{Doob-Meyer} decomposition 
 \begin{equation} 
\label{A41a}
   \ell^2 \big( T-s, \widehat{X} (s) \big) \,=\,\widehat{M}  (s) + \int_0^s \sum_{ y \neq x }  \bigg( \widehat{\kappa} (x,y) \,   \Big( \ell (t, y) - \ell (t, x ) \Big)^2  \bigg)
  \bigg|_{t = T-u \atop x = \widehat{X} (u)}  
   \, \ud u\,, \qquad 0 \le s \le T~~~~
\end{equation} 
of the   time-reversed  
variance process $\,  \ell^2 \big( T-s, \widehat{X} (s) \big) ,$ $\,0 \le s \le T\,,$ where $\widehat{M}$  is   a     $ \big( \widehat{\mathbb{G}}, \mathbb{Q}\big)-$martingale.
\end{prop}  

\noindent
{\it Proof:} The first claim follows from  Proposition \ref{FJ} and the \textsc{Jensen} inequality. For the second claim we deploy Proposition \ref{Liggett_Back} with 
$\, 
g(s, x) := \ell^2 (T-s, x)\,, ~\, 0 \le s \le T, ~~ x \in {\cal S} 
\,,$
to conclude via the   calculation  
 \begin{equation} 
\label{A41b-claim}
\big( \partial g  + \widehat{{\cal K}} g  \big) (T-s, x) \,= \,\sum_{y \in {\cal S} } \, \,\widehat{\kappa} (x,y) \Big( \ell (T-s, y) - \ell (T-s, x) \Big)^2 
\end{equation} 
that  $\widehat{M}$   is  a  local  $\, \big( \widehat{\mathbb{G}}, \mathbb{Q}\big)-$martingale.   The uniform continuity of  $ [0,T] \ni t \mapsto p_t (x,y) \in [0,1] $  and the finiteness of the state-space  imply that this  process is actually bounded, thus a true $\, \mathbb{Q}-$martingale. 

Let us now justify the claim (\ref{A41b-claim}). From the Backwards  Equation (\ref{A22}), we have 
\begin{equation} 
\label{A41c-DM}
\partial g (T-s, x) \,= - 2 \ell (T-s, x)\, \partial  \ell (T-s, x) = - 2 \ell (T-s, x)  \sum_{y \in {\cal S}} \widehat{\kappa} (x,y) \ell (T-s, y),
\end{equation} 
$$
\big(  \widehat{{\cal K}} g  \big) (T-s, x) \,= \,\sum_{y \in {\cal S} } \, \widehat{\kappa} (x,y) \, \ell^2 (T-s, y)\,= \,\sum_{y \in {\cal S} } \, \widehat{\kappa} (x,y)  \, \Big[ \ell^2 (T-s, y)  +  \ell^2 (T-s, x) \Big]
$$
on account of the property $\, \sum_{y \in {\cal S}} \widehat{\kappa} (  x, y )=0\,$   for every $x \in {\cal S} ;$ now  (\ref{A41b-claim}) follows readily. 
\qed

 \smallskip
Proposition \ref{DE_Traj} deals with the trajectorial behavior of the  
variance process; and for this, it is crucial to let time run backwards. Now, we want to adopt also an ``aggregate" point of view, and take $\mathbb{Q}-$expectations in (\ref{A41a}). When doing this, it does not matter any more  whether time runs forwards or backwards, so we state the following result ``forwards in time".  Recalling (\ref{A23}), we obtain thus  the   dissipation of   the variance. 


\begin{prop}
\label{DE}
Along the curve $  \big( P(t) \big)_{0 \le t < \infty} $ of time-marginal  distributions in (\ref{A19}),  the   
variance   
\begin{equation} 
\label{Variance}  
V \big(P(t)\,|\,Q\big)= {\rm Var}^\mathbb{Q} \big( {\bm \ell }_t \big)= \sum_{y \in {\cal S}} \, q (y) \,\ell  ^2 (t, y) -1 \,= \, \big \| {\bm \ell}_t \big \|_{\bL^2(\cS,Q)}^2 -1\,, \quad   0 \le t < \infty\,\,
 \end{equation}
is   decreasing  with $\, \lim_{t \to \infty } \downarrow    
V \big(P(t)\,|\,Q\big)=0 ,$ and the rate of its decrease is given by 
\begin{equation} 
\label{Variance_Decay}
 \partial \, \big \|  {\bm \ell }_t  \big \|_{\bL^2(\cS,Q)}^2 \,=\,   \partial \, %
 V  \big(P(t)\,|\,Q\big) \,=\, -  2 \, {\cal E} \big( {\bm \ell }_t\,  ,  {\bm \ell }_t  \,\big) 
 \end{equation}
(thus by $\, - 2\, \big \|
  {\bm \ell }_t \big \|^2_{\bH^1(\cS,Q)}  \,$ under the detailed-balance conditions (\ref{A24})). More precisely, 
 \begin{equation} 
\label{A41d}
V \big(P(T)\,|\,Q\big) \,=   \,   
V \big(P(0)\,|\,Q\big) -    \int_0^T \sum_{(x,y) \in {\cal Z} }  \, q (y) \, \kappa (y,x)\,  \Big( \ell (t,y) - \ell (t,x) \Big)^2 \, \ud t 
\end{equation} 
 \begin{equation} 
\label{A41e}
 = \int_T^\infty \sum_{(x,y) \in {\cal Z} }  \, q (y) \, \kappa (y,x) \,  \Big( \ell (t,y) - \ell (t,x) \Big)^2 \, \ud t \,.
\end{equation}
\end{prop}

The decomposition (\ref{A41a})   is a trajectorial version of this   variance   dissipation, at the level of the individual particle    viewed under   the probability measure $\mathbb{Q}$ and   under time-reversal. 
  As a consequence of (\ref{A41a}) and of the \textsc{Bayes} rule, we deduce from (\ref{A41a})  the \textsc{Doob-Meyer} decomposition 
 \begin{equation} 
\label{A41c}
   \ell  \big( T-s, \widehat{X} (s) \big) \,=\,\widehat{N}  (s) + \int_0^s \sum_{ y \neq x } \bigg( \frac{\,  \widehat{\kappa} (x,y)\,}{\,  \ell  \big( t, x \big)\,}   \,  \Big( \ell (t, y) - \ell (t, x ) \Big)^2   \bigg)\bigg|_{t = T-u \atop x = \widehat{X} (u)}  \, \ud u\,, \qquad 0 \le s \le T~~~~
\end{equation}
of the time-reversed likelihood process, where $\widehat{N}$  is a     $\, \big( \widehat{\mathbb{G}}, \mathbb{P}\big)-$martingale.

\subsection{Steepest Descent of the Variance, under Detailed Balance }
\label{sec6.1}

 We state now and establish the following result,  Theorem \ref{Steep_Desc_Var}.  As   pointed out in \textsc{Jordan,   Kinderlehrer  \& Otto}\,(1998),  results of this type go  as far back as the paper by \textsc{Courant, Friedrichs \& Lewy} (1928)  in the   Brownian motion context. We deploy the notation of (\ref{A21}) for the   likelihood ratios   relative to the invariant distribution, as well as the following notion.

\begin{defn}
\label{SD}
We say that a   smooth  curve of probability measures $  (   P (t)  )_{t_0 \le t < \infty} \subset \cM = \cP_+(\cS)$ is   of {\it steepest descent} locally at $\,t = t_0\,,$ for a given   smooth  functional $F : \cM \to \R$ and relative to a given metric $\varrho$ on $\cM,$  if it minimizes, among all curves $ \,   ( \widetilde P (t)   )_{t_0 \le t < \infty} \subset \cM\,$ satisfying $\widetilde P (t_0) =  P (t_0),$  the infinitesimal rate of change of $F$ as measured on $\cM$ in terms of $\varrho\,,$  namely,
\begin{align*}
\lim_{h \downarrow 0} \,\frac{\, F \big(\widetilde P (t_0 +h) \big) - F  \big(P   (t_0) \big)}{\,\varrho  \big( \widetilde P (t_0+h), P (t_0) \big)\, } \,.
\end{align*}
\end{defn}

\begin{thm}
{\bf Steepest  Descent for the Variance:}
\label{Steep_Desc_Var}
Under  the conditions (\ref{A24}) of detailed balance,  {\it the curve $\,  ( P (t)  )_{0 \le t < \infty}\,$ of time-marginal distributions   in (\ref{A19}) has the property of}   steepest decent for the variance  of (\ref{Variance})   with respect to the metric distance bequeathed by the norm     of (\ref{D10}), i.e., 
\begin{equation} 
\label{Metric_Dist}
\varrho  \big( P_1, P_2 \big) 
	\,:= \, \big \| \, {\bm \ell}_1 - {\bm \ell}_2 \,\big 
		  \|_{\bH^{-1}(\cS,   Q)}  \qquad \text{ for  \, $P_1 = {\bm \ell}_1 Q$ \,  and $~P_2 = {\bm \ell}_2\, Q\,.$}
\end{equation}
\end{thm}

  \smallskip

The  proof of this result  needs    Proposition \ref{Proj} below. We pave the way towards it by formulating first a variational version  of Propositions \ref{DE_Traj},  \ref{DE}. For this purpose, we fix an arbitrary time-point $t_0 \in (0, \infty)$   and let $  \psi (\cdot)  = ( \psi (t) )_{t_0 \le t < t_0 + \varepsilon}$ be a continuous curve of real-valued functions on the state-space ${\cal S}$. 

 
With these ingredients, we define a new curve $\, \ell^\psi (\cdot)  = ( \ell^\psi (t) )_{t_0 \le t < t_0 + \varepsilon}$ of such functions, for a suitable $\varepsilon >0,$   by specifying  in the space $\, {\cal L = L_+ (S)}\,$ of subsection \ref{sec2.2}  the initial condition $\ell^\psi (t_0) = {\bm \ell} (t_0) \in {\cal L}$ and the dynamics $\, \partial \ell^\psi (t) = ( \widehat{{\cal K}} \psi ) (t)$ for $\, t \in [t_0, t_0 + \varepsilon);$    in the manner of (\ref{A22}) and a bit more explicitly,  
\begin{equation} 
\label{A22_Psi}
 \partial \ell^\psi (t,x) \,=\, \sum_{y \in {\cal S}}\, \widehat{\kappa} (   x,y) \, \psi (t,y)    \,,\qquad x \in {\cal S}.
\end{equation}
 The  curve $\, \ell^\psi (\cdot)  = ( \ell^\psi (t) )_{t_0 \le t < t_0 + \varepsilon}\,,$ the ``output" of the system (\ref{A22_Psi}) corresponding to the ``input" $  \psi (\cdot)$, is only defined on an interval $\,   [ t_0 , t_0 + \varepsilon)\,$ and lives in the space $\, {\cal L}   ,$ since
$$
\partial \, \sum_{x \in {\cal S}} \,q(x) \,\ell^\psi (t,x)\,=\, \sum_{x \in {\cal S}} \,q(x) \sum_{y \in {\cal S}} \, \widehat{\kappa} (x,y)\, \psi (t,y)\,=\, \sum_{y \in {\cal S}} \sum_{x \in {\cal S}}   \,q(y) \,  \kappa  ( y,x)\,   \psi  (t,y)\,=\,0 
$$
implies $ \sum_{x \in {\cal S}} \,q(x) \,\ell^\psi (t,x) =  \sum_{x \in {\cal S}} \,q(x) \, {\bm \ell} (t_0,x) =1$ for all $ t \in [t_0 + \varepsilon)$. Thus, the recipe 
\begin{equation} 
\label{A22_Prob}
p^\psi (t,x):= q(x) \,\ell^\psi (t,x), \qquad (t,x)  \in [t_0, t_0 + \varepsilon) \times {\cal S}
\end{equation}
procures a curve $\, ( P^\psi (t))_{0 \le t < t_0 + \varepsilon}\,, $ on the manifold $\,{\cal M = P_+ (S)}\,$ in subsection \ref{sec2.1}  consisting of vectors $\, P = \big( p (x) \big)_{x \in {\cal S}}\,$ 
with strictly positive elements and total mass $\, \sum_{x \in {\cal S}} p (x)=1\,.$ 

\smallskip
Conversely: By   irreducibility, the ``input curve" $\, \psi (\cdot)\,$ is   determined by the ``output curve" $\, \ell^\psi (\cdot)\,$ up to an additive time-dependent  constant. In particular, every smooth curve $\, \ell^* (\cdot)  = ( \ell^*  (t) )_{t_0 \le t < t_0 + \varepsilon}$  in $\, {\cal L}\,$ with  $\, \ell^* (t_0) = \ell (t_0)\,$ is representable as $\, \ell^\psi (\cdot)\,$ for a suitable continuous $\, \psi (\cdot)\,$ as above. For instance,     $ \, {\bm \ell} (\cdot) \in {\cal L}\,$ of (\ref{A21}) is the ``output" that corresponds in this manner to the ``input" $ \,   \psi  (\cdot) \equiv {\bm \ell} (\cdot)   $ in (\ref{A22_Psi}), via (\ref{A22}).

We have the following generalization of Proposition \ref{DE}, to which it reduces when $ \,   \psi  (\cdot) \equiv {\bm \ell} (\cdot)  .$

\begin{prop}
\label{Prel_Proj}
In the above context, we have for $\,t \in [t_0 + \varepsilon)\,$ the properties 
$$
\partial\,V \big(P^\psi (t)\,|\,Q\big)= \partial\, \mathbb{E^Q} \Big[ \big(  \ell^\psi \big)^2 \big(t,   X  (t) \big) \Big]= 2 \,\big \langle \psi_t, {\cal K} \ell^\psi_t \big \rangle_{\bL^2(\cS,Q)} = -2 \, {\cal E} \big( \psi_t, \ell^\psi_t \big)  .
$$
Whereas, under the detailed balance conditions (\ref{A24}), this expression becomes 
$$
\partial\,V \big(P^\psi(t)\,|\,Q\big)\,=\,
-2 \, {\cal E} \big( \ell^\psi_t , \psi_t \big)\,=\, -2\, \Bip{  \,  \nabla \ell^\psi_t ,  \nabla \psi_t }_{\bL^2(\cZ,C)}  \,=\, -2\, \Big \langle         \ell^\psi_t ,    \psi_t   \Big \rangle_{\bH^{1}(\cS, Q)}  \,.
$$
\end{prop}
\noindent
{\it Proof:}  A reasoning   similar to that   in Propositions \ref{DE_Traj} and \ref{DE},   and carried out once again in the backwards direction of time, can be deployed by applying  Proposition \ref{Liggett_Back} to  
$\, 
g(s, x) := \big( \ell^\psi \big)^2 (T-s, x)\,, ~~\, 0 \le s \le T,$ $  x \in {\cal S} 
\, $ for arbitrary but fixed $\, T \in (0, t_0 + \varepsilon).$ 
But here is a simpler argument: 
$$
\partial\,V \big(P^\psi (t)\,|\,Q\big)= \,\partial \big\| \ell^\psi_t \big\|^2_{\bL^2(\cS,Q)} = \, 2 \, \big \langle \ell^\psi_t , \widehat{{\cal K}} \psi_t  \big \rangle_{\bL^2(\cS,Q)} = \,2 \,\big \langle  \psi_t ,  {\cal K}  \ell^\psi_t  \big \rangle_{\bL^2(\cS,Q)} =  -2 \, {\cal E} \big( \psi_t, \ell^\psi_t \big),
$$
on account of (\ref{AdjRel}), (\ref{A26a}) and (\ref{A23}). This reasoning proves Proposition \ref{DE} as well. 
 \qed 

\medskip
 
  We compute now    the ``infinitesimal cost  of moving the curve" $\, \big( \ell^\psi (t) \big)_{t_0 \le t < t_0 + \varepsilon}\,. $

\begin{prop}
\label{Proj}
Under the  conditions (\ref{A24}) of detailed balance, we have 
 \begin{equation} 
\label{A49}
\lim_{h \downarrow 0}\, \frac{1}{h}\,   \big \| \,{\bm \ell }_{t+h} -   {\bm \ell }_{t} \,\big \|_{\bH^{-1}(\cS,Q)}
\, = \,\big \|   \,  {\cal K}  \,{\bm \ell }_{t} \, \big \|_{\bH^{-1}(\cS,Q)} \, = \,\big \|   \,  {\bm \ell }_{t} \, \big \|_{\bH^{1}(\cS, Q)} 
\end{equation} 
for every $t \in [t_0, t_0 + \varepsilon); $ and a bit more generally, in the notation just developed, 
 \begin{equation} 
 \label{A49a} 
\lim_{h \downarrow 0}\, \frac{1}{h}\,   \big \| \,  \ell_{t+h}^\psi -    \ell_{t}^{ \psi}  \,\big \|_{\bH^{-1}(\cS,Q)}
\, = \,\big \|   \,  {\cal K}  \,  \psi_{t} \, \big \|_{\bH^{-1}(\cS,Q)} \, = \,\big \|   \,    \psi_{t} \, \big \|_{\bH^{1}(\cS, Q)} \,.
\end{equation} 
\end{prop}
\noindent
{\it Proof:} From (\ref{A22_Psi}), (\ref{A24})  it follows that  for every  $\, x \in {\cal S}\,  $ we have 
 $$
 \lim_{h \downarrow 0}\,\frac{1}{\,h\,} \Big[ \ell^{ \psi}_{t+h}  (x) - \ell^{ \psi}_{t }  ( x)  \Big] \,=\, \big(  {\cal K}  \,  \psi_t \big)  ( x)  ,  
 $$
 \noindent
so the first equality in (\ref{A49a}) is evident.   For the second equality in (\ref{A49a})  it suffices to recall (\ref{D10})--(\ref{D12}),   observe   that $\, \nabla \psi_t\,$ is the unique element $\, F \in \bL^2(\cZ,C)
\,$ with the property $\, \nabla \cdot F =  {\cal K}  \,\psi_t\,,$ and   note from Remark \ref{Rem3.1} the isometry 
$\,
\big \| F \big \|_{\bL^2(\cZ,C)} =  \big \| \psi_t \big \|_{\bH^1(\cS,Q)}
\,$
from the space 
$\bL^2(\cZ,C)$ to $\bH^1(\cS, Q)$.

Now, (\ref{A49}) is just a special case of (\ref{A49a}) with $ \,   \psi  (\cdot) \equiv {\bm \ell} (\cdot)   ,$  as discussed above. \qed

\subsection{The Proof of Theorem \ref{Steep_Desc_Var} }
\label{justitia}

We are ready  to tackle the proof of Theorem \ref{Steep_Desc_Var}. Along {\it any} smooth  curve of the form $\, ( P^\psi (t))_{t_0 \le t < t_0 + \varepsilon}\,$ created as in (\ref{A22_Psi}), (\ref{A22_Prob})   on the manifold of probability vectors $\,{\cal M = P_+(S)}  $ and with $\ell^\psi (t_0) = {\bm \ell} (t_0) \in {\cal L}$, we have from Propositions  \ref{Prel_Proj},  \ref{Proj}  the respective rates for the variance and the metric distance,  under detailed balance: 
$$
\lim_{h \downarrow 0} \frac{\, V \big(P^\psi (t_0 +h)\,|\,Q\big) - V \big(P (t_0)\,|\,Q\big)}{h} \,=\, -2\, \Big \langle    \,    {\bm \ell }_{t_0} ,    \psi_{t_0}   \Big \rangle_{\bH^{1}(\cS, Q)}  \,,
$$
$$
\lim_{h \downarrow 0} \, \frac{\,\varrho \big( P^\psi (t_0+h), P (t_0) \big)\,}{h} \,=\, \big \|       \psi_{t_0}   \big \|_{\bH^{1}(\cS, Q)} \,.
$$
Thus, the rate of change for the variance  along  the  {\it perturbed curve} $\,  ( P^\psi (t)  )_{t_0 \le t < t_0 + \varepsilon}\,,$ when measured on the manifold ${\cal M}$ by the metric distance in (\ref{Metric_Dist}), is 
$$
\lim_{h \downarrow 0} \frac{\, V \big(P^\psi (t_0 +h)\,|\,Q\big) - V \big(P   (t_0)\,|\,Q\big)}{\,\varrho \big( P^\psi (t_0+h), P (t_0) \big)\, } \,=\, -2\, \,\bigg \langle    \,    {\bm \ell }_{t_0} ,  \frac{  \psi_{t_0}  }{\, \big \|       \psi_{t_0}   \big \|_{\bH^{1}(\cS, Q)} \,}   \bigg \rangle_{\bH^{1}(\cS, Q)} \,.
$$

On the other hand, along the {\it original curve} $\,  ( P (t)  )_{0 \le t < \infty}\,$ of time-marginal distributions for the Chain   (that is, with $    \psi  (\cdot) \equiv {\bm \ell} (\cdot)   $ modulo an  affine  transformation, as noted above), the rate of variance dissipation  measured in terms of the metric distance traveled   on  the manifold ${\cal M}\,$ is 
$$
\lim_{h \downarrow 0} \frac{\, V \big(P  (t_0 +h)\,|\,Q\big) - V \big(P  (t_0)\,|\,Q\big)}{\,\varrho \big( P  (t_0+h), P  (t_0) \big)\, } \,=\, -2\, \,\big \|    \,   {\bm \ell }_{t_0}   \big \|_{\bH^{1}(\cS, Q)} \,<\,0\,.
$$
A simple comparison of the last two displays, via \textsc{Cauchy-Schwarz}, gives   the {\it steepest descent  property of the variance with respect  to the metric distance in   
(\ref{Metric_Dist}),} i.e., 
$$
\lim_{h \downarrow 0} \frac{\, V \big(P^\psi (t_0 +h)\,|\,Q\big) - V \big(P   (t_0)\,|\,Q\big)}{\,\varrho \big( P^\psi (t_0+h), P (t_0) \big)\, } \,- \,\lim_{h \downarrow 0} \frac{\, V \big(P  (t_0 +h)\,|\,Q\big) - V \big(P  (t_0)\,|\,Q\big)}{\,\varrho \big( P  (t_0+h), P  (t_0) \big)\, }  
$$
$$
= \, 2 \left( \,\big \|       {\bm \ell }_{t_0}   \big \|_{\bH^{1}(\cS, Q)} - \bigg \langle    \,    {\bm \ell }_{t_0} ,  \frac{  \psi_{t_0}  }{\, \big \|       \psi_{t_0}   \big \|_{\bH^{1}(\cS, Q)} \,}   \bigg \rangle_{\bH^{1}(\cS, Q)} \right) \,\ge \, 0\,,
$$
{\it along the original curve $\, ( P(t))_{0 \le t < \infty}\,$ of time-marginals for the \textsc{Markov} Chain.} Equality holds here  if, and only if, $\,c+\psi_{t_0}$ is a   positive   constant multiple of $\,{\bm \ell }_{t_0}$ for some    $\,c \in \R\,.$ \qed

\medskip
We  will revisit  this theme in Sections  \ref{sec7} and \ref{sec8}.

\section{The Relative Entropy Process}
\label{sec6}

For an arbitrary probability vector $\, P =  ( p (x)  )_{x \in {\cal S}}\,$ with strictly positive elements,  let us recall the definition of  its {\it relative entropy}, or {\sc Kullback--Leibler} \emph{divergence},  
\begin{equation} 
\label{H1}
 H (P\,|\,Q) \,:= \,\sum_{x \in {\cal S}}\, p(x) \log \Big( \frac{\,p (x) \,}{ q (x)} \Big) 
 \end{equation}
with respect to the invariant distribution $\, Q =  ( q (x)  )_{x \in {\cal S}}\,$ of (\ref{A18}). In terms of the likelihood function in (\ref{A21}),   the  relative entropy of the probability vector $P(t)$ in (\ref{A19}) with respect  to $\,Q,$    is 
\begin{equation} 
\label{A29}
H \big( P(t) \, \big| \, Q \big) \, =\,   \mathbb{E^P} \Big[  \log \ell \big( t, X(t) \big) \Big] , \qquad 0 \le t < \infty\,,
\end{equation}
the $\p-$expectation of the log-likelihood at time $t$. We shall see presently that this function
\begin{equation} 
\label{A30}
t \, \longmapsto \,H \big( P(t) \, \big| \, Q \big) ~~\text{is non-negative, and satisfies} ~\, \lim_{t \to \infty} \downarrow H \big( P(t) \, \big| \, Q \big)  =0\,.
\end{equation}
In other words, the relative entropy functional of (\ref{H1}) is   a \textsc{Lyapunov} function for the   curve $   ( P(t)  )_{0 \le t < \infty} $ of time-marginal distributions for our continuous-time \textsc{Markov} Chain. We shall compute in  subsection \ref{de Bruijn} the   rate of   temporal  decrease for the function in (\ref{A30}). Of course, all this    is in accordance with general thermodynamic principles  governing the approach to equilibrium in  physical systems (e.g., Chapter 2 in \textsc{Cover \& Thomas} (1991)  in the discrete-time \textsc{Markov} Chain context of our Section \ref{sec1}).

Let us note also, that  the relative entropy in (\ref{A29}) can be cast equivalently  as the $\mathbb{Q}-$expectation 
\begin{equation} 
\label{A35}
H \big( P(t) \, \big| \, Q \big) \,=\, \sum_{y \in {\cal S}} \, q ( y) \,\frac{\,p(t,y)\,}{q(y)}\, \log \left( \frac{\,p(t,y)\,}{q(y)}  \right)\,=\, \mathbb{E^Q} \Big[   \ell \big( t, X(t) \big)
   \log  \ell \big( t, X(t) \big)   \Big] 
\end{equation}
of the {\it relative entropy process}
$\,
\ell \big( t, X(t) \big) \cdot  \log  \ell \big( t, X(t) \big)\,, ~\, 0 \le t < \infty\,.
$
This allows us to justify the first claim in (\ref{A30}), regarding  non-negativity. Indeed, the   convexity of the function $\, (0, \infty) \ni \ell \mapsto \Phi (\ell) := \ell \, \log \ell\,\,$ gives, on the strength of the \textsc{Jensen} inequality, 
\begin{equation} 
\label{A35*}
H \big( P(t) \, \big| \, Q \big) \,=\, \mathbb{E^Q} \big[  \Phi \big( \ell \big( t, X(t) \big) \big)  \big]\, \ge \, \Phi \Big( \mathbb{E^Q} \big[ \ell \big( t, X(t)   \big)  \big] \Big)\,=\, f (1) \,=\, 0\,.
\end{equation}
 Alternatively, this follows from $\,H \big( P(t) \, \big| \, Q \big) = \mathbb{E^Q}  [  \Psi \big( \ell   ( t, X(t)  )  )   ] ,$ with $\Psi \ge 0$ as   in (\ref{A52d}) below.

 \begin{prop}
 \label{FJ_cor}
In the context of Proposition \ref{FJ}, the time-reversed relative entropy process 
\begin{equation} 
\label{A40a}
  \ell \big( T-s,  \widehat{X}  (s) \big) \cdot \log \ell \big( T-s,  \widehat{X}  (s) \big) , ~~~~0 \le s \le  T\,~~~~~~~\text{is a}~\, \big( \widehat{\mathbb{G}}, \mathbb{Q}\big)- \text{submartingale;}
\end{equation}
the properties in (\ref{A30}) hold; and     the time-reversed  log-likelihood   process
\begin{equation} 
\label{A40b}
  \log \ell \big( T-s,  \widehat{X}  (s) \big) , ~~~~0 \le s \le  T  ~~~~~~~\text{is a}~\, \big( \widehat{\mathbb{G}}, \mathbb{P}\big)- \text{submartingale.}
\end{equation}
 \end{prop}

 \noindent
 {\it Proof:} 
The first claim  follows from (\ref{A40}) and the convexity of the function $\,  \Phi (\ell)  = \ell \, \log \ell\,\,$ appearing inside the expectation in (\ref{A35}), along with the \textsc{Jensen} inequality. The $\,\mathbb{Q}-$expectation 
\begin{equation} 
\label{A40d}
H \big( P (T-s) \,| \,Q)\,=\, \mathbb{E^Q} \big[ \Phi  \big( \ell \big( T-s, \widehat{X} (s) \big) \big) \big]\,,~~~~~~ ~ 0 \le s \le T 
\end{equation}
 of the process in (\ref{A40a})   is thus increasing. This is precisely the monotonicity  in (\ref{A30}); the remaining claim 
\begin{equation} 
\label{A30'}
\, \lim_{t \to \infty} \downarrow H \big( P(t) \, \big| \, Q \big)  =0\,
\end{equation}
 there,  follows now from (\ref{A20a}),  (\ref{A35}), and the finiteness of $\, {\cal S}.$ The   claim of (\ref{A40b})  is a consequence of  (\ref{A40a}), (\ref{A34}), and the familiar \textsc{Bayes} rule  (Lemma 3.5.3 in \textsc{Karatzas \& Shreve} (1988)).  \qed

\subsection{Trajectorial  Relative Entropy Dissipation}
\label{Traject}

We read now Proposition \ref{Liggett_Back} with $\,\Phi (\ell)  = \ell \, \log \ell\,$ and the function
\begin{equation} 
\label{A43}
h (s,x) \,=\, \Phi \big( \ell (T-s, x) \big)\,, \qquad 0 \le s \le T, ~ ~x \in {\cal S}\,.
\end{equation} 
\noindent
As   argued in   Propositions \ref{FJ}  and   \ref{FJ_cor},  the ``time-reversed  relative entropy" 
$ H \big( P (T-s) \,| \,Q) = \mathbb{E^Q} \big[ h \big( s, \widehat{X}(s) \big) \big] ,$ $ 0 \le s \le T \,$ is   increasing; and 
 \begin{equation} 
\label{A45}
\widehat{M}^{ h} (s) \, :=  \, h \big( s,\widehat{X}(s) \big) - \int_0^s \big( \partial h + \widehat{{\cal K}} h \big) \big(u, \widehat{X} (u) \big)\, \ud u \,, \qquad 0 \le s  \le T 
\end{equation}
 is a $\mathbb{Q}-$local-martingale of the time-reversed   filtration $\,\widehat{\mathbb{G}}\,.$ The integrand in (\ref{A45}) is straightforward to compute: from (\ref{A22}), (\ref{A23}), and with $t = T-s$ for notational convenience,  we get 
$$
 \partial h (s,x)  =   - \big( 1 + \log \ell (t, x) \big) \, \big( \widehat{{\cal K}} \, \ell \big) (t,x) = -\big( 1 + \log \ell (t, x) \big) \,
\sum_{y \in {\cal S}}  \, \widehat{\kappa} ( x, y)  \, \ell (t, y),\quad \text{thus}
$$ 
\begin{equation} 
\label{A47}
\big( \partial h + \widehat{{\cal K}} h \big) (s,x) \,=\,  \sum_{y \in {\cal S}}  \, \widehat{\kappa} ( x, y)\, \ell (t, y) \Big[ \log \frac{\, \ell (t, y)\,}{\ell (t, x)} -   1  \Big] = \,\ell (t,x) \sum_{y \in {\cal S} \atop y \neq x }  \,   \widehat{\kappa} ( x, y)   \,   \Psi \Big( \frac{\, \ell (t, y)\,}{\ell (t, x)} \Big) \ge 0\,. 
\end{equation}
Here the  function 
\begin{equation} 
\label{A52d}
  \Psi (r)\, :=\, r \log r -r +1\,,~~~~~~~r>0  
 \end{equation}
  \noindent
 is nonnegative, convex, and  attains its minimum $\Psi (1)=0$ at $r=1$. We have used in the last equality  of (\ref{A47}) the property $\, \sum_{y \in {\cal S}} \,\widehat{\kappa} ( x,  y )=0\,$   for every $x \in {\cal S} .$

 \begin{prop}
 \label{DM}
 The submartingales of (\ref{A40a}), (\ref{A40b}) admit   the respective \textsc{Doob-Meyer} decompositions 
 \begin{equation} 
\label{A50a}
 \ell \big( T-s,  \widehat{X}  ( s) \big) \log \big( \ell \big( T-s,  \widehat{X}  ( s) \big) \big)\,=\, \,\widehat{M}^{ \,h} (s) + \int_0^s \lambda^{\mathbb{Q}} (u) \, \ud u\,, \qquad 0 \le s \le T \,,
\end{equation}
\begin{equation} 
\label{A50b}
  \log \big( \ell \big( T-s,  \widehat{X}  ( s) \big) \big)\,=\, \,\widehat{N}^{\, h} (s) +   \int_0^s \lambda^{\mathbb{P}} (u) \, \ud u\,, \qquad 0 \le s \le T\,,
\end{equation}
   in the notation of (\ref{A47}), (\ref{A52d}), with  $\,\lambda^{\mathbb{Q}} (s) =\Lambda^{\mathbb{Q}} \big( T-s, \widehat{X}(s)\big) \,,$ $\,\lambda^{\mathbb{P}} (s) =\Lambda^{\mathbb{P}} \big( T-s, \widehat{X}(s)\big) $ and 
\begin{equation} 
\label{A52a}
\Lambda^{\mathbb{P}} (t,x) \,:=\, \sum_{y \in {\cal S}, \, y \neq x }   \, \widehat{\kappa} ( x, y)\, \Psi  \Big( \frac{\, \ell (t, y)\,}{\ell (t, x)} \Big) \ge 0
\,,  \qquad 
\Lambda^{\mathbb{Q}} (t,x) \,:=\,  \ell (t,x) \,\Lambda^{\mathbb{P}} (t,x) \ge 0
\,.
\end{equation}
Here  $\,\widehat{M}^{ \,h}   \,$  is the process of (\ref{A45}) and a  $\big( \widehat{\mathbb{G}}, \mathbb{Q}\big)-$martingale, whereas $\,\widehat{N}^{ \,h}   \,$ is a $\big( \widehat{\mathbb{G}}, \mathbb{P}\big)-$martingale. 
 \end{prop}
 \noindent
 {\it Proof:} Let us take   a look at the expressions of (\ref{A43})--(\ref{A47}). We have already noted that  each function $ [0,T] \ni t \mapsto p (t,x) \in (0,1)\,$ is uniformly continuous. This fact, along with the finiteness of the state space $  {\cal S} ,$ implies that the quantities in (\ref{A43}), (\ref{A47}) are actually uniformly bounded. This  implies a similar boundedness for the $\, \big( \widehat{\mathbb{G}}, \mathbb{Q}\big)-$local martingale in (\ref{A45}), which is thus seen to be a true $\, \big( \widehat{\mathbb{G}}, \mathbb{Q}\big)-$martingale. The remaining claims follow  from  the \textsc{Bayes} rule. \qed

\medskip
 The decomposition (\ref{A50a}) is a trajectorial version of   relative entropy dissipation. This manifests itself at the level of the {\it individual particles} that undergo  the \textsc{Markov} Chain motion  viewed under the lens of the probability measure $\mathbb{Q}$ and under time-reversal, rather than only at the level of their ensembles. 

 \smallskip
 We note   that the first quantity of  (\ref{A52a}) provides 
the  exact  rate of relative entropy dissipation, in the sense that for every $\,0 < t < T < \infty \,$ we have the following convergence, a.e.\,and  in $\, \mathbb{L}^1 (\mathbb{P})$:
\begin{equation} 
\label{A52c}
\lim_{s \uparrow T-t}  \frac{1}{\, T - t -s\,} \, \bigg( \mathbb{E^P}\Big[ \log \ell \big( t , X(t) \big) \, \Big| \,\widehat{\mathcal{G}} (s) \Big]  -  \log \big( \ell \big( T-s,  \widehat{X}  ( s) \big) \big) \bigg)  =\,\Lambda^{\mathbb{P}} \big( t,  X  ( t) \big) .
\end{equation}
  The decomposition (\ref{A50b})  and the trajectorial rate (\ref{A52c}) are exact analogues of those   in Theorem 4.1 and Proposition 4.5 of \textsc{Karatzas, Schachermayer \& Tschiderer} (2020). They constitute  trajectorial versions of   relative entropy dissipation,   viewed now under   the original  probability measure $\mathbb{P}$  --- and again  under time-reversal.

\subsection{\textsc{de Bruijn}-Type Identities}
\label{de Bruijn}

 \smallskip
 With this preparation,  we  are now in a position to recover   the precise rate of decay for the relative entropy function in (\ref{A29}); cf.\,\textsc{Diaconis \& Saloff-Coste} (1996), Lemma 2.5.  All this takes, is to ``aggregate"  in (\ref{A50a}) by taking $\mathbb{Q}-$expectations.  This leads to an analogue of equation (4.14) in \textsc{Karatzas, Schachermayer \& Tschiderer} (2020),    as we describe now.  \footnote{~ The seminal paper  \textsc{Stam}  (1959), from the early days of Information Theory, establishes the first    identity of this type, and  in a context where   the invariant measure $Q$ is standard Gaussian. \textsc{A.J.\,Stam}  gives   credit for   this result to his teacher, the analyst, number theorist, combinatorialist and logician Nicolaas   \textsc{de\,Bruijn}.}

\begin{thm}
\label{deBruijn}
{\bf  \textsc{de Bruijn}-type identity for the Dissipation of   Relative Entropy:} The relative entropy of (\ref{A29})   is a decreasing function of time, and satisfies 
\begin{equation} 
\label{A46}
   H \big( P (T ) \,| \,Q)= H \big( P (0) \,| \,Q) - \int_0^T I ( t) \, \ud t  = \int_T^\infty I ( t) \, \ud t\,, \qquad  I(t)  \,=\,  {\cal E} \big( {\bm \ell}_t, \log {\bm \ell}_t \big) \ge 0
\end{equation}
for all $ \, T \in [0, \infty)$, in the notation of (\ref{A26a}), (\ref{A21}).
\end{thm}

  \noindent
  {\it Proof:} The first claim is simply a restatement of (\ref{A30}); and by taking $\mathbb{Q}-$expectations in (\ref{A50a}) we obtain in conjunction with (\ref{A43}) the first equality of  (\ref{A46}), with 
$$ 
I(t) \, := \, \mathbb{E^Q} \big[ \big( \partial h + \widehat{{\cal K}} h \big) \big(T-t,  X (t) \big) \big] .
 $$
From (\ref{A47}) and (\ref{A26a}),   this quantity coincides with   the last expression in (\ref{A46}): to wit, 
\begin{equation} 
\label{A48}
I(t) \, =\, \sum_{x \in {\cal S}}  \, q(x) \, \big( \partial h + \widehat{{\cal K}} h \big) (T-t,x)  \,=\, \sum_{(x,y) \in {\cal Z} }  \,q(x)\, \widehat{\kappa} ( x, y)\, \ell (t, y) \Big[ \log \frac{\, \ell (t, y)\,}{\ell (t, x)} -   1  \Big] 
\end{equation}
$$
 =\,\sum_{x \in {\cal S}}   \,  q(x)\,  \ell (t,x) \sum_{y \in {\cal S} \atop y \neq x }   \,\widehat{\kappa} ( x, y)    \, \Psi \Big( \frac{\, \ell (t, y)\,}{\ell (t, x)} \Big) =- \sum_{x \in {\cal S}} \sum_{y \in {\cal S}}  \, \kappa (y,x) \, q(y)  \, \ell (t, y) \, \log \ell (t,x)\,=\,  {\cal E} \big( {\bm \ell}_t, \log {\bm \ell}_t \big).
$$
It is non-negative on account of the non-negativity of the last expression in (\ref{A47}), and uniformly continuous as a function of time. In the display (\ref{A48}), the second equality follows from the first equality in (\ref{A47}); the third from the last equality in   (\ref{A47}); the fourth from (\ref{A23}) and the property $\, \sum_{y \in {\cal S}} \, \kappa  (    y,x )=0\,$   for every $y \in {\cal S} $; and the fifth    from the definition (\ref{A26a}). We deduce $\,H \big( P (0) \,| \,Q)  = \int_0^\infty I ( t) \, \ud t \,$  by letting $\, T \to \infty\,$ in (\ref{A46}) and recalling (\ref{A30'}); then the second identity in (\ref{A46}) follows.
  \qed

    \begin{rem}
 Whenever there exists a positive  real constant  $\, \alpha\,$ (respectively, $\, \beta\,$) such that the \textsc{Poincar\'e} (resp., the modified log-\textsc{Sobolev}) inequality 
\begin{equation} 
\label{A60*}
 \alpha \, \le \, \frac{\, 2\, {\cal E} (f,  f) \,}{\,\sum_{y \in {\cal S}} \, q(y) f^2 (y) -1\,} \qquad \bigg(\text{resp.,} \quad  
\beta\, \le \, \frac{\,  {\cal E} (f,  \log f)\,}{\,\sum_{y \in {\cal S}} \, q(y)   f  (y)   \log f(y)\,}\bigg)
 \end{equation}
 holds for every function   $  f : {\cal S} \to (0, \infty)   $   with $\,\sum_{y \in {\cal S}} \, q(y) f  (y) =1,$  it is clear from (\ref{Variance}), (\ref{Variance_Decay}) and (\ref{A35}), (\ref{A46}) that    the variance (resp., the relative entropy) decays exponentially:
\begin{equation} 
\label{A60**}
  \text{Var}^\mathbb{Q} \big( L (t) \big) \,\le \,    \text{Var}^\mathbb{Q} 
  \big( L (0) \big)    \,e^{\, - \alpha \, t}  \qquad  \Big(\text{resp.,} ~~~H \big( P (t) \,| \,Q) \,\le \,   H \big( P (0) \,| \,Q) \,e^{\, - \beta \, t} \Big)\,.
 \end{equation}
\end{rem}
 
\begin{rem}
Expressions for entropy dissipation in a \textsc{Markov} Chain context appear in, e.g., \textsc{Miclo} (1992),  Lemma 2.5 of \textsc{Diaconis \& Saloff-Coste} (1996), \textsc{Bobkov  \& Tetali}  (2006),      \textsc{Montenegro  \&   Tetali} (2006), \textsc{Caputo et al.}\,(2009)    and \textsc{Conforti}\,(2020).   These authors   use   slightly different arguments, based on semigroups. One advantage of the more probabilistic approach we follow here, is that it provides a very sharp picture for the dissipation of relative entropy {\it along trajectories,}    as exemplified in subsection \ref{Traject}.
\end{rem}

\subsection{\textsc{Fisher} Information Under Detailed Balance}

The following  is now a direct consequence of Lemma \ref{DiscreteGradientToo}.

\begin{prop}
 Under  the   detailed-balance condition  (\ref{A24}), the rate of relative entropy dissipation in  (\ref{A46}) can be cast as  
 $$
 I (t)  \,=\,  {\cal E} \big( {\bm \ell}_t, \log {\bm \ell}_t \big) \,=\, \frac{1}{\,2\,} \sum_{(x,y) \in {\cal Z} }   \, \Big(      \log \ell  \big( t,y \big) - \log \ell  \big( t,x \big) \Big)^2  \,\, \Theta \big( \ell (t,y),  \ell (t,x)\big)    
\, \kappa (y,x) \, q(y)  
 $$
\begin{equation} 
\label{A50}
~~~~~~~~~~~~~~~~~~~~~~~~\,=\, \frac{1}{\,2\,}  \sum_{(x,y) \in {\cal Z} }   \, \frac{\,\big(       \ell (t,y) - \ell (t,x) \big)^2 \,}{\Theta \big( \ell (t,y),  \ell (t,x)\big)} 
  \, \kappa (y,x) \, q(y)  ~~~~~~~~~~~~~~~~~~~~~~~~~~~~~~~
\end{equation}
in terms  of the   ``logarithmic mean" function 
\begin{equation} 
\label{A80}
\Theta (  q, p) \,:=\, \frac{ q-p}{\,  \log q- \log p \, } \,=\, \int_0^1 q^r \,p^{1-r}\,\ud r\,, \qquad ( q,p) \in (0, \infty)^2.
\end{equation}
\end{prop}

   \smallskip

\begin{rem}
The  expression in (\ref{A50})  is reminiscent of the familiar   \textsc{Fisher}  Information  in Statistics and Information Theory; cf.\,\textsc{Bobkov  \& Tetali}  (2006). Always under  the   detailed-balance condition  (\ref{A24}), the  expression of (\ref{A50}) can be expressed in terms of a ``score function", the  discrete logarithmic-gradient of the likelihood ratio, as $\, \bip{ \nabla {\bm \ell}_t , \nabla \log {\bm \ell}_t }_{\bL^2(\cZ,C)}  \,$ in the notation of (\ref{D1})-(\ref{D4}).

As shown in \textsc{Bobkov  \& Tetali}  (2006), the  inequality $\, 2 (a-b)^2 \le (a^2 - b^2) \, \log (a / b)\,$ for $0 < a, b < \infty\,$ leads under      detailed-balance    (\ref{A24}) to the  \textsc{Diaconis and Saloff-Coste} (1996) estimate 
$$\,
{\cal E} \big( e^g, g \big)   \ge   \,4 \, \,{\cal E} \big( e^{g/2}, \,e^{g/2} \big)\,,
~~~~~~~~\text{and thus to}~~~~~~~~\, I(t) = {\cal E} \big( {\bm \ell}_t, \log {\bm \ell}_t \big) \ge \,4 \, {\cal E} \big( \sqrt{ {\bm \ell}_t\,} , \sqrt{{\bm \ell}_t\,} \,\big).
$$
\end{rem}

\section{The $\Phi-$Relative Entropy Process}
\label{Gen}
\label{sec7}

In order to reveal the common thread running through the examples of the last two  Sections, let us consider now a     continuously differentiable and convex function $\, \Phi : (0, \infty) \to \R $ with $\Phi (1)=0$   with continuous, strictly positive second derivative. We denote by $\, \varphi : (0, \infty) \to \R\, $ its derivative $\, \Phi' = \varphi\,.$ For each $\eta >0$, $\xi >0$ we define the \textsc{Bregman} $\, \Phi-${\it divergence} 
\begin{equation} 
\label{A51}
\text{div}^\Phi \big( \eta \, | \, \xi \big) \,:=\, \Phi (\eta) - \Phi (\xi )- (\eta - \xi) \, \varphi (\xi) \,,
\end{equation}
a quantity which is non-negative on account of the convexity of $\Phi$ (and has nothing to do with the ``discrete divergence" we introduced in (\ref{D2})). For instance, $\,\text{div}^\Phi \big( \eta \, | \, \xi \big)= (\xi - \eta)^2\,$ for $\, \Phi (\xi) = \xi^2-1\,;$ whereas, for $\, \Phi (\xi) = \xi \, \log \xi\,$ and  in the notation of (\ref{A52d}), we have 
\begin{equation} 
\label{A51a}
\,\text{div}^\Phi \big( \eta \, | \, \xi \big)\,=\,  \text{div}^\Psi \big( \eta \, | \, \xi \big)\,=\,  \xi \, \Psi (\eta / \xi)\,.
\end{equation}

Let us consider now, for a general convex $\Phi$ as above,  the so-called $\, \Phi-${\it relative entropy}   
\begin{equation} 
\label{A51b}
H^\Phi \big( P(t)    \big|  Q \big) \,:=\, \mathbb{E^Q} \big[ \Phi \big( \ell (t, X(t) \big) \big]\,=\, \sum_{y \in {\cal S}} \, q(y) \, \Phi \Big( \frac{p (t,y)}{q(y)} \Big)\,, \qquad 0 \le t < \infty\,;
\end{equation}
see \textsc{Chafa\"i} (2004) for a general study of such functions. The convexity of $\Phi$ and the \textsc{Jensen} inequality imply that this function is nonnegative, since $\Phi (1)=0$; and from Proposition \ref{FJ},   that the {\it time-reversed $\, \Phi-$relative entropy process} $$\, \Phi \big( \ell (T-s, \widehat{X} ( s) \big), ~~~~~~~\, 0 \le s \le T\,$$ is a $\, (\mathbb{\widehat{G} , Q})-$submartingale, for every fixed $T \in (0, \infty)$.  As a consequence  the function in (\ref{A51b}) decreases, in fact satisfies $\, \lim_{t \to \infty} \downarrow H^\Phi \big( P(t) \, \big| \, Q \big)  =0\,$   by virtue of (\ref{A20a}) and the finiteness of the state space.

\subsection{Trajectorial   Dissipation of the $\, \Phi-$Relative Entropy}
\label{Phi_Traject}

 The \textsc{Doob-Meyer} decomposition of this submartingale is obtained from Proposition \ref{Liggett_Back}  as follows: Consider  the function $g(s, x) = \Phi \big( \ell (T-s, x)\big)$ and compute, in the manner of (\ref{A47}), the quantities 
$$
  \partial g (s,x)  = - \varphi  \big(  \ell (t, x) \big)    
\sum_{y \in {\cal S}}  \, \widehat{\kappa} ( x, y)  \big[ \ell (t, y)- \ell (t,x) \big],~~~~~
\big( \widehat{{\cal K}} g \big) (s,x) =   
\sum_{y \in {\cal S}}  \, \widehat{\kappa} ( x, y)  \big[ \Phi \big( \ell (t, y) \big) - \Phi \big( \ell (t,x) \big) \big]
$$
with $t=T-s$, on account of (\ref{A22}). Putting these expressions together with (\ref{A51}), we deduce 
\begin{equation} 
\label{A52}
\big(  \partial g+ \widehat{{\cal K}} g \big) (s,x) \,=   
\sum_{y \in {\cal S}, \, y \neq x}  \, \widehat{\kappa} ( x, y) \, \,\text{div}^\Phi \big( \eta \, | \, \xi \big) \bigg|_{\eta = \ell (t, y) \atop \xi = \ell (t, x)} \, =: \, \Lambda^{\Phi,   \mathbb{Q} } (t,x) \, \ge \, 0\,.
\end{equation}
 
The following result is now a direct consequence of 
 Proposition \ref{Liggett_Back}  and the \textsc{Bayes} rule. Once again,   the finiteness of the state-space and  the continuity of the functions involved, turn  local into true martingales.

 \begin{prop}
 \label{Prop_8.1}
For any given $T \in (0, \infty)$, the process below is a $\, (\mathbb{\widehat{G} , Q})-$martingale:
 \begin{equation} 
\label{A52aa}
\Phi \big( \ell (T-s, \widehat{X}(s))\big)- \int_0^s \Lambda^{\Phi,  \mathbb{Q} } \big(  T-u, \widehat{X}(u) \big)\, \ud u\,, \qquad 0 \le s \le T\,.
\end{equation}
Whereas, with  $\,\Lambda^{\Phi,   \mathbb{P} } (t,x) := \Lambda^{\Phi,   \mathbb{Q} } (t,x)/ \ell (t,x),$ the process   below  is a $\, (\mathbb{\widehat{G} , P})-$ martingale:
\begin{equation} 
\label{A52bb}
\frac{\Phi \big( \ell (T-s, \widehat{X}(s))\big)}{\ell (T-s, \widehat{X}(s))} - \int_0^s \Lambda^{\Phi,  \mathbb{P} } \big(  T-u, \widehat{X}(u) \big)\, \ud u\,, \qquad 0 \le s \le T.
\end{equation}
 \end{prop}

\subsection{Generalized \textsc{de Bruijn}  Identities}

In view of these considerations, it is  now straightforward to ``aggregate" (i.e., take $\mathbb{Q}-$expectations of)   the $  (\mathbb{\widehat{G} , Q})-$martingale of (\ref{A52aa}). We obtain in the manner of (\ref{A46}) the 
following result, stated again in the forward direction of time; cf.\,\textsc{Chafa\"i}  (2004), Proposition 1.1.

\begin{prop}
\label{Prop_Gen_de_Br}
{\bf Generalized \textsc{de Bruijn}-type identity:} The temporal dissipation of the $\, \Phi-$relative entropy of (\ref{A51b})   is given   for  $\,0 \le T < \infty\,$ as
\begin{equation} 
\label{A53}
   H^\Phi \big( P (T ) \,| \,Q)=H^\Phi \big( P (0 ) \,| \,Q)- \int_0^T I^\Phi ( t) \, \ud t=  \int^\infty_T I^\Phi ( t) \, \ud t \,, \qquad  I^\Phi(t)  \,:=\,  \mathbb{E^Q} \big[ \Lambda^{\Phi,  \mathbb{Q} }\big(  t, X(t) \big) \big] \ge 0 \,.
\end{equation}
 On the strength of (\ref{A52}), the dissipation rate in (\ref{A53}) is given by the $\, \Phi-$\textsc{Fisher} {\it Information}
$$
I^\Phi(t)  =  \sum_{(x,y) \in {\cal Z} }    
q(x)\, \widehat{\kappa} ( x, y) \, \,\mathrm{div}^\Phi \big( \ell (t, y)   \big|   \ell (t, x) \big) = \sum_{x \in {\cal S} }      
\sum_{y \in {\cal S} }  \,q(x)\, \widehat{\kappa} ( x, y) \,\mathrm{div}^\Phi \big( \ell (t, y)   \big|   \ell (t, x) \big)
$$
\begin{equation} 
\label{A54}
=  - \sum_{x \in {\cal S} }      
\sum_{y \in {\cal S}   
}    q(y)\,  \kappa  (   y,x) \,  \ell (t, y)\, \varphi \big( \ell (t,x) \big) = {\cal E} \big( {\bm \ell}_t, \varphi ({\bm \ell}_t)\big)  .
\end{equation}
\end{prop}

 \noindent
 {\it Proof:}  
The third equality in (\ref{A54})   is a consequence of the properties $\, \sum_{y \in {\cal S}} \,\widehat{\kappa} ( x,  y )=0\,$   for every $x \in {\cal S} ,$ and $\, \sum_{x \in {\cal S}} \, \kappa  (   y,x )=0\,$   for every $y \in {\cal S} ,$ as well as of (\ref{A23}). It underscores the fact that, when passing from the trajectorial to the ``aggregate" point of view (that is, when taking $\mathbb{Q}-$expectations), the term $\, \xi \varphi (\xi) -\Phi (\xi)\,$ that depends only on the variable $\xi = \ell (t,x)$, as well as the term $\,\Phi (\eta)\,$ that depends only on the variable  $\eta = \ell (t,y)$, can be ignored in (\ref{A51}); only the ``mixed term" $\,  -\eta \, \varphi (\xi)\,$  remains relevant. We note   that similar reasoning was deployed in the proof of Lemma \ref{DiscreteGradient}.     \qed

\smallskip
\begin{rem} {\it   Some Special  Cases:}  \,{\it (i)}  \,
For the convex function $\, \Phi (\xi) = \xi   \log \xi\, , $\, and recalling (\ref{A51a}), (\ref{A52d}), the quantity $I^\Phi(t)$ of (\ref{A54}) is seen to   coincide  with $I(t)$ in (\ref{A48}), (\ref{A46}).

  \smallskip  
  \noindent
{\it (ii)} \,On the other hand, when $\, \Phi (\xi) = \xi^2-1\, $ we have $\,\text{div}^\Phi \big( \eta \, | \, \xi \big) = (\eta - \xi )^2\,$ in (\ref{A51}) and 
$$ 
H^\Phi \big( P(t)    \big|  Q \big) =\, \mathbb{E^Q}   \big( \ell^2 (t, X(t) \big)  - 1\,= \,\big| {\bm \ell}_t \big|_{\bL^2(\cS,Q)}^2 -1\,=\, \text{Var}^{\mathbb{Q}} (L (t))\,=\,V  \big(P(t)\,|\,Q\big)\,, \qquad 0 \le t < \infty 
$$ 
as in   (\ref{A41b}), and the rate of temporal dissipation for  this function is precisely the integrand in (\ref{A41e}): 
\begin{equation} 
\label{A54b}
I^\Phi(t)  \, = \,-2\, \sum_{x \in {\cal S} }      
\sum_{y \in {\cal S}   
}  \,  q(y)\,  \kappa  (   y,x) \,  \ell (t, x)\,   \ell (t,y)  \,=\, 2\, {\cal E} \big( {\bm \ell}_t, {\bm \ell}_t  \big)
\,.
\end{equation}
{\it (iii)} \,
A bit more generally, the choice of convex function $\, \Phi (\xi) = (\xi^m-1) / (m-1)\, $ with $\, m >1 ,$ leads to the so-called ``\textsc{R\'enyi} relative entropy" 
\begin{equation} 
\label{Ren_Rel_Ent}
 H^\Phi \big( P(t)    \big|  Q \big) \,= \,\frac{\,\mathbb{E^Q}   \big( \ell^m (t, X(t) \big)  - 1\,}{m-1}  \,, \qquad 0 \le t < \infty 
 \end{equation}
whose rate of temporal dissipation   is a generalized version of (\ref{A54b}): 
$$
I^\Phi(t)  \, = \,-\, \frac{m}{m-1}\, \sum_{x \in {\cal S} }      
\sum_{y \in {\cal S}   
}  \,  q(y)\,  \kappa  (   y,x)\, \ell (t,y) \,  \big(   \ell  (t, x) \big)^{m-1} \,=\, \frac{m}{m-1} \, \,{\cal E} \big( {\bm \ell}_t, {\bm \ell}_t^{\,m-1} \big)\,.
$$
  The variance $\text{Var}^{\mathbb{Q}} (L (t))\,$ is thus a special case of the  \textsc{R\'enyi} relative entropy, corresponding to $\,m=2\,$;   whereas the relative entropy in (\ref{A35})  
 corresponds to the limit of (\ref{Ren_Rel_Ent})   as $\, m \downarrow 1.$

{\it We stress  that nowhere in this subsection, or in the one preceding it, did we invoke the detailed-balance     conditions of (\ref{A24}).}
\end{rem}

\subsection{Locally Steepest Descent  for the $\, \Phi-$Relative Entropy Under Detailed Balance}
\label{subsec:LocSteep}

 We formulate now a variational version of Proposition \ref{Prop_Gen_de_Br} {\it under the       conditions   (\ref{A24}) of detailed balance. These will  be in force throughout the current subsection.} 
 
 \smallskip  
  
 \begin{rem}
First, let us   take a look at the expression of (\ref{A54}).    From the consequence $\,q(x)\, \widehat{\kappa} ( x, y)= q(y)\,  \kappa  (   y,x) = q(y)\, \widehat{\kappa} (   y,x)\,$ of the detailed balance conditions (\ref{A24}), as well as from the consequence 
$$
\text{div}^\Phi \big( \eta \, | \, \xi \big) + \text{div}^\Phi \big( \xi \, | \, \eta \big) \,=\, \big( \eta - \xi\big) \big( \varphi (\eta) -  \varphi (\xi) \big)
$$
 of (\ref{A51}), we see that the $\, \Phi-$\textsc{Fisher}  Information  of (\ref{A54})  can be cast in this case as 
\begin{equation}\begin{aligned}\label{A55}
I^\Phi (t) 
 &  = \, 
  \frac{1}{2}\, \sum_{x \in {\cal S} }      
\sum_{y \in {\cal S}   
}  q(x)\, \widehat{\kappa} ( x, y) \, \Big( \text{div}^\Phi \big( \eta \, | \, \xi \big) + \text{div}^\Phi \big( \xi \, | \, \eta \big) \Big)\bigg|_{\eta = \ell (t, y) \atop \xi = \ell (t, x)}
\\ & = \,\frac{1}{2}\, \sum_{x \in {\cal S} }      
\sum_{y \in {\cal S}   
}  q(y)\,  \kappa  (   y,x) \, \Big(  \big( \eta - \xi\big) \big( \varphi (\eta) -  \varphi (\xi) \big) \Big)\bigg|_{\eta = \ell (t, y) \atop \xi = \ell (t, x)} = {\cal E} \big(  \varphi ({\bm \ell}_t), {\bm \ell}_t \big)   
 \\&  =  \,\frac{1}{2}  \sum_{(x,y) \in {\cal Z} }    
q(x)\,  \kappa  ( x, y)  \,
 \Theta^\Phi (  \xi, \eta) \big(\varphi(\xi) - \varphi(\eta)\big)^2 \, \bigg|_{\xi = \ell (t, x) \atop  \eta = \ell (t, y) } \,\, 
\end{aligned}\end{equation}
in the manner of (\ref{A50}); we recall the notation $\, \varphi  = \Phi' .$ Here, the function 
 \begin{equation} 
\label{A51ab}
\Theta^\Phi (  q, p) \,:=\, \frac{\,    q-   p \, }{\,  \varphi (q)- \varphi (p) \,}  \,, \quad    0 < q \neq p < \infty  \,, \qquad ~~\Theta^\Phi (  p, p) \,:=\,\frac{1}{\, \Phi^{''} (p)\,}\,, \quad 0 < p < \infty\,,
\end{equation}
  extends the ``logarithmic mean" of (\ref{A80}), to which it reduces when $\Phi (\xi) = \xi \log \xi.$ With $\, \Phi (\xi) = \xi^2 -1$ we get $\,\Theta^\Phi \equiv 1/2,\,$ and the last expression in (\ref{A55}) reduces to   $\,\,\sum_{(x,y) \in {\cal Z} }   \, 
q(x)\,  \kappa  ( x, y) \cdot  \big(\ell (t, x) - \ell (t, y) \big)^2  \,$ as in (\ref{A41e}).   We shall comment further on this  choice of (\ref{A51ab}), in subsection 9.2 below.  
\end{rem}
 
 \medskip
{\it  We set out now to find a  metric on the manifold $\, {\cal M = P_+ (S)}$ of probability vectors on the state-space, relative to which the   time-marginals for the \textsc{Markov} Chain $ ( P (t) )_{0 \le t < \infty} $ constitute a curve of steepest descent for the $\Phi-$relative entropy.} In other words, we look for a metric on $\, {\cal M}\,$ that can play --- in the current   general  context --- a role similar to that played by the \textsc{Hilbert} norm $\, \| \cdot \|_{\mathbb{H}^{-1} (\cS, Q)}\, $ in Section \ref{sec5}.  

This norm defines  the metric distance of (\ref{Metric_Dist}) that works for the variance $V (P(t)|Q)$, i.e., in the special case $\, \Phi (\xi) = \xi^2-1.$ But except for such very special cases, the Riemannian metric on the manifold ${\cal M}$ will {\it not be flat;}  i.e., {\it not}   induced by such a simple norm as in Proposition  \ref{Proj}. For this reason we are forced to consider the machinery of Riemannian geometry, which we take up in the next Section \ref{sec8}. 
 In this Section we avoid Riemannian terminology, and present the steepest descent property of the curve $\, ( P (t) )_{0 \le t < \infty} \, $ in terms of appropriate \textsc{Hilbert} norms that capture the {\it local}  behavior of the Riemannian metric.

\subsubsection{  Locally Weighted \textsc{Sobolev} Norms  } 

  We start this effort by recalling  from (\ref{D4a})  the norm $   \| F   \|_{\bL^2(\cZ, C)}  $ for functions $F: {\cal Z} \to \R\,.$ This is defined     on the ``off-diagonal Cartesian product" $  {\cal Z}   
$ by assigning to its elements $\,(x,y)$, where $ \, x \neq y\,$, the weights $c(x,y) = \, q(x)\, \kappa  (  x,y) / 2\,$ and taking the usual $\,\mathbb{L}^2-$norm   relative to the positive measure with these weights. For a fixed likelihood ratio $\ell$ in the space ${\cal L = L_+ (S)}\,$ of subsection \ref{sec2.2}  we  consider now, in place of $\, c(x,y) \equiv q(x)\, \kappa (  x,y) / 2\,$ and with the notation of (\ref{A51ab}), the new weights
\begin{align}
\label{eq:def-vartheta}
c(x,y) \cdot 
\vartheta_\ell(x,y)\,,
\qquad\text{ where\, } 
\quad
\vartheta_\ell(x,y) := \Theta^\Phi \big(  \ell (x), \ell (y) \big) \,=\, \frac{\nabla \ell (x,y)}{\nabla ( \varphi \circ \ell )  (x,y)}\,.
\end{align}
 The resulting  {\it weighted} inner product and norm, extensions of the respective quantities for real-valued functions on $\, {\cal S} \times {\cal S}\,$ in (\ref{D4}), (\ref{D4a}) (to which they reduce when $\Phi (\xi) = \xi^2 / 2\,$), are   respectively
 \begin{equation}
 \begin{aligned}
\label{A56}
 \bip{    F ,  G }_{\bL^2(\cZ, \vartheta_\ell C)} 
 	\,&  :=  \, 
		\sum_{(x,y) \in {\cal Z} } 
			\, c(x,y) 
		    \, \vartheta_\ell(x,y)
		 F(x,y)
		\, G(x,y)
		\,\,= \,\bip{  \, \vartheta_\ell F,
	G\, }_{\bL^2(\cZ,C)}\,, 
\\
		\big \| F \big \|_{\bL^2(\cZ, \vartheta_\ell C)}^2
	&\, := \,\bip{  F ,  F }_{\bL^2(\cZ, \vartheta_\ell C)}\,.
\end{aligned}
\end{equation}

 We define now for     $  f: {\cal S}  \to \R $ the  {\it Weighted  \textsc{Sobolev}  Norm} $    \|       \cdot     \|_{\bH^1_\Theta(\cS, \ell Q)} ,$
by replacing on the right-hand sides of (\ref{A26b})--(\ref{D10})    the norm $\,  \| \cdot   \|_{\bL^2(\cZ, C)}\,$  by the new norm $\,  \| \cdot   \|_{\bL^2(\cZ, \vartheta_\ell C)}\,$    in (\ref{A56}):
 \begin{equation} 
\label{A58}
 \big \langle f, g \big \rangle_{\bH^1_\Theta(\cS, \ell Q)}	 	\,:= 
	\,\bip{ \nabla f ,  \nabla g }_{\bL^2(\cZ, \vartheta_\ell C)}\,, 
 \quad  
 \big \| f  \big \|^2_{\bH^1_\Theta(\cS, \ell Q)} 
 \,:=  \,  \big \langle f, f 
 		   \big \rangle_{\bH^1_\Theta(\cS, \ell Q)}= \big \| \nabla f \big \|_{\bL^2(\cZ, \vartheta_\ell C)}^2\,.
\end{equation} 

\smallskip

\begin{rem} 
\label{rem8.3}
It is interesting to note at this point, and will become quite important down the road, that the $  \Phi-$\textsc{Fisher}  Information of (\ref{A54}), (\ref{A55}) can be expressed in terms of the square of this new, weighted \textsc{Sobolev} norm. Indeed, for any $\ell \in {\cal L_+ (S)}$ we have
$$
	{\cal E} \big({ \ell}, \varphi ( { \ell} )  \big) 
	= \bip{   \nabla { \ell} ,
	\nabla \varphi ( { \ell} )  }_{\bL^2(\cZ,C)}
	= \bip{    \vartheta_\ell \nabla \varphi ( { \ell} ) ,
	\nabla \varphi ( { \ell} )  }_{\bL^2(\cZ,C)}
	=    \big \|       \nabla \varphi ( { \ell})   \big \|^2_{\bL^2(\cZ, \vartheta_\ell C)}
	=      \big \|   \varphi ( { \ell} )   \big \|^2_{\bH^1_\Theta(\cS, \ell Q)}.
$$
Thus, the $\, \Phi-$\textsc{Fisher}  Information of     (\ref{A54})  takes the form 
$\,
I^\Phi (t) =   \big \|   \varphi ( {\bm \ell}_t )    \big \|^2_{\bH^1_\Theta(\cS, { \bm \ell}_t Q)}.
$
\end{rem}
\smallskip

Finally, we introduce   in the manner of (\ref{D11}), (\ref{D12})   the dual of this  weighted  \textsc{Sobolev} norm 
\begin{equation}\begin{aligned}
\label{D11_too}
 \big \|  f  \big \|_{\bH^{-1}_\Theta(\cS, \ell Q)} 
  \,:=\, 
  	\sup_{g : {\cal S} \to \R}
		 \frac{\,\bip{f, g}_{\bL^2(\cS,Q)}\,}
		 	  {\, \big \|  g  \big \|_{\bH^1_\Theta(\cS, \ell Q)}\,}  .
\end{aligned}\end{equation}
 This admits a variational characterization    analogous to  (\ref{D12}),  which will be crucial in what follows.  

\begin{prop} {\bf Variational Interpretation:} 
\label{prop:H-min-char}
For any function $\,f : \cS \to \R\,$ we have
\begin{align}
\label{eq:H-min-char}
	\big \|  f  \big \|_{\bH^{-1}_\Theta(\cS, \ell Q)} 	& = \inf_{G : \cZ \to \R } 
    \bigg\{ \| G \|_{\bL^2(\cZ, \vartheta_\ell C)} 
    	 \ : \  f + \nabla \cdot \big(\vartheta_\ell G\big)=0   \bigg \}\, .
\end{align}
\noindent
Moreover, $\big \|  f  \big \|_{\bH^{-1}_\Theta(\cS, \ell Q)}$ is finite if, and only if, $\,\sum_{x \in \mathcal{S}} q(x) f(x)  = 0\,;$ 
  in this case the infimum is attained, and uniquely, by the unique discrete gradient that is admissible.
\end{prop}

\begin{proof}
Consider a function $f : \cS \to \R$ such that $\sum_{x \in \mathcal{S}} q(x) f(x) = 0;$ if this is not the case, it is straightforward to verify that both sides in \eqref{eq:H-min-char} are infinite.   We note   that the set of admissible $G$ on the right-hand side of \eqref{eq:H-min-char} is non-empty (indeed, $G_0 := - \frac{1}{\vartheta_\ell}\nabla \cK^{-1} f$ is admissible) and that a minimizer exists. 

Let $G : \cZ \to \R$ be such a minimizer.  
We show first that $G$ is a discrete gradient, by a projection argument in the \textsc{Hilbert} space 
$\bL^2(\cZ, \vartheta_\ell C)$.

To this end, let us  denote by $\nabla h$   the orthogonal projection of $G$ onto the subspace of discrete gradients in $\bL^2(\cZ, \vartheta_\ell C)$.
We claim that $\nabla h$ is admissible on the right-hand side of \eqref{eq:H-min-char}. Indeed, $ G - \nabla h$ is orthogonal in $\bL^2(\cZ, \vartheta_\ell C)$ to $\nabla g$ for any $g : \cS \to \R$. This implies
\begin{align*}
 - \bip{\, g , \nabla \cdot\big( \vartheta_\ell (G - \nabla h)\big) \, }_{\bL^2(\cS,Q)}
 =
 \bip{\, \nabla g , \vartheta_\ell (G - \nabla h) \, }_{\bL^2(\cZ,C)}
 =  \bip{\,\nabla g , G - \nabla h}_{\bL^2(\cZ, \vartheta_\ell C)} 
 = 0  
\end{align*}
and yields $\nabla \cdot\big( \vartheta_\ell G \big) = \nabla \cdot\big( \vartheta_\ell \nabla h \big),$  proving the claim. 

 \smallskip
By orthogonality, we have $ \| G \|_{\bL^2(\cZ, \vartheta_\ell C)}^2 =  \|\nabla h\|_{\bL^2(\cZ, \vartheta_\ell C)}^2 +  \| G - \nabla h\|_{\bL^2(\cZ, \vartheta_\ell C)}^2$. 
Since $G$ is a minimizer, we infer   $\| G - \nabla h\|_{\bL^2(\cZ, \vartheta_\ell C)} = 0$, which implies that $G \equiv \nabla h$.This shows that $\nabla h$ is a minimizer, and that the right-hand side of \eqref{eq:H-min-char} is equal to $\|  h \big \|_{\bH^1_\Theta(\cS, \ell Q)}$. It is shown in \textsc{Maas} (2011) that $\nabla h$ is actually the \emph{unique} discrete gradient satisfying the constraint in \eqref{eq:H-min-char}.

 \smallskip
To prove the equality in \eqref{eq:H-min-char}, we note   for any $g : \cS \to \R$ the identities 
\begin{align*}
	\ip{f,g}_{\bL^2(\cS,Q)}
	= - \bip{\nabla \cdot\big( \vartheta_\ell \nabla h \big),g}_{\bL^2(\cS,Q)}
	= \bip{\vartheta_\ell \nabla h ,\nabla g}_{\bL^2(\cZ,C)}
	= \big \langle h , g \big \rangle_{{\bH^1_\Theta(\cS, \ell Q)}  \,\,.} 
\end{align*}
Writing the dual norm as a \textsc{Legendre} transform, we obtain
\begin{align*}
	 \big \|  f  \big \|_{\bH^{-1}_\Theta(\cS, \ell Q)}^2 
	& = \sup_{g : \cS \to \R} 
		 \bigg\{ 
		 	2\ip{f,g}_{\bL^2(\cS,Q)}
			 - 
		 \big \|  g  \big \|_{\bH^1_\Theta(\cS, \ell Q)}^2
		 \bigg\}
\\&	= \sup_{g : \cS \to \R} 
		 \bigg\{ 
		 	 2   \, \big \langle h , g \big \rangle_{{\bH^1_\Theta(\cS, \ell Q)}}		 - 
		 \big \|  g  \big \|_{\bH^1_\Theta(\cS, \ell Q)}^2
		 \bigg\}
	= \|  h \big \|_{\bH^1_\Theta(\cS, \ell Q)}^2 \,\,,
\end{align*}
which establishes the equality in \eqref{eq:H-min-char}.
\end{proof}

Let us   consider now  as   in   Section \ref{sec5}, for some $\varepsilon>0$   an arbitrary smooth curve $\, \ell^\psi (\cdot)  = ( \ell^\psi (t) )_{t_0 \le t < t_0 + \varepsilon}$ with  initial position $\ell^\psi (t_0) = {\bm \ell} \equiv {\bm \ell} (t_0)$ in ${\cal L = \cal L_+ (S)}$.
In order to compute ${\bH^{-1}_\Theta(\cS, \ell Q)}$-norms, it is natural in view of Proposition \ref{prop:H-min-char} to write the time-evolution in the manner of a ``discrete continuity equation'' $$\,\partial \ell^\psi_t + \nabla \cdot ( \vartheta_{\ell_t} \nabla \psi_t ) = 0$$ as in subsection \ref{sec6.1}, where $\psi_t : \mathcal{S}\to \R$ is unique up to an additive constant.

  We regard here  $  \psi (\cdot)$ as an input, whose gradient is the velocity vector field  that yields   the infinitesimal change $\partial \ell^\psi_t$  of the likelihood ratio flow.  In light of (\ref{A22}), (\ref{eq:def-vartheta}) and detailed balance, the original backward   equation $ \,\partial {\bm \ell}_t = \widehat{{\cal K}} {\bm \ell}_t= {\cal K} {\bm \ell}_t = \nabla \cdot ( \nabla {\bm \ell}_t ) = \nabla \cdot ( \vartheta_{\bm \ell_t} \nabla \varphi(\bm \ell_t ) $    corresponds to $\, \psi_t = - \varphi ({\bm \ell}_t)\,$ in this scheme of things.

\smallskip
We define  as in (\ref{A22_Prob})  the corresponding curve $   P^\psi (\cdot)  = \big( P^\psi (t) \big)_{t_0 \le t < t_0 + \varepsilon} $     on   the manifold $  {\cal M = P_+ (S)}$  of probability vectors on the state-space. We obtain   the  following  generalization of Proposition \ref{Prop_Gen_de_Br}.

\begin{prop}
\label{Prop8.3}
In the above context, we have
\begin{equation}\begin{aligned}
\label{A57}
	\partial H^\Phi \big(P^\psi (t)\,|\,Q\big)
 =   \bip{ \,  \varphi (\ell^\psi_t ) ,    \psi_t  \,}_{\bH^1_\Theta(\cS, \ell_t Q)}  \,.
\end{aligned}\end{equation}
\end{prop}

\begin{proof}
Using the abovementioned discrete continuity equation, a discrete integration by parts, and the definitions of the scalar products, we deduce 
\begin{align*}
\partial H^\Phi \big(P^\psi (t)\,|\,Q\big) 
& = \,\partial \,\mathbb{E^Q} \Big[ \Phi \big(  \ell^\psi    \big(t,   X  (t) \big) \big) \Big]
\,  
=  - \bip{  \,   \varphi (\ell^\psi_t ),\nabla \cdot ( \vartheta_{\ell_t} \nabla \psi_t )  \, }_{\bL^2(\cS,Q)} 
\\&=  \bip{  \,  \nabla \varphi (\ell^\psi_t ), \vartheta_{\ell_t} \nabla \psi_t \,  }_{\bL^2(\cZ,C)}
=  \bip{  \,  \nabla \varphi (\ell^\psi_t ), \nabla \psi_t \, }_{\bL^2(\cZ, \vartheta_{\ell_t} C)} 
= \bip{ \,  \varphi (\ell^\psi_t ) ,    \psi_t  \,}_{\bH^1_\Theta(\cS, \ell_t Q)}\,,
\end{align*}
as desired.
\end{proof}

 With the context and   notation just established,  and always for $\,{\bm \ell} \equiv {\bm \ell} (t_0)= ( \ell (t_0, x)  )_{x \in {\cal S}}\,,$  we can formulate the following  analogue of Proposition \ref{Proj}.  This result  uses the characterizations of the weighted $\mathbb{H}^{-1}-$norm  in \eqref{D11_too},   along with the identity $\cK {\bm \ell} = \nabla \cdot \big( \vartheta_{\bm \ell} \nabla \varphi({\bm \ell})\big)$.

\begin{prop}
\label{Proj_Gen}
Under the  conditions (\ref{A24}) of detailed balance, we have, with $\ell^\psi (t_0) = {\bm \ell} \equiv {\bm \ell} (t_0)$,
 \begin{equation} 
\label{A60}
\lim_{h \downarrow 0}\, \frac{1}{h}\,   \big \| \,{\bm \ell }_{t_0 +h} -   {\bm \ell }_{t_0} \,\big \|_{\bH^{-1}_\Theta(\cS, {\bm \ell} Q)} 
	\, = \,\big \|   \, {\cal K}\, {\bm \ell}_{t_0} \, \big \|_{\bH^{-1}_\Theta(\cS, {\bm \ell} Q)} \, 
	    = \,\big \|   \,  \varphi( {\bm \ell}_{t_0} )\, \big \|_{\bH^1_\Theta(\cS, {\bm \ell} Q)} \,; 
\end{equation} 
  and a bit more generally, 
   \begin{equation} 
 \label{A61} 
\lim_{h \downarrow 0}\, \frac{1}{h}\,   \big \| \,  \ell_{t_0+h}^\psi -    \ell_{t_0}^{ \psi}  \,\big \|_{\bH^{-1}_\Theta(\cS, {\bm \ell} Q)}
 \, = \,
 \big \|   \,\nabla \cdot ( \vartheta_{{\bm \ell }_{t_0}} \nabla \psi_{t_0})  \, \big \|_{\bH^{-1}_\Theta(\cS, {\bm \ell} Q)}
 \, = \,\big \|      \psi_{t_0}   \big \|_{\bH^1_\Theta(\cS, \ell Q)} \,.
\end{equation} 
\end{prop}

 \medskip
 
We   pass now to the principal result of the Section.  This  generalizes Theorem \ref{Steep_Desc_Var}, to which it reduces when $\Phi (\xi) = \xi^2-1.$   It is also a direct analogue of Theorem 3.2 in \textsc{Karatzas, Schachermayer \& Tschiderer}\,(2020), where a similar steepest-descent  for the   relative entropy is established for \textsc{Langevin} diffusions, and with distance on the ambient space measured by the quadratic \textsc{Wasserstein}  metric.   The role of that metric is played  now by the locally flat metric defined in (\ref{Metric_Dist_Too}) below.

\begin{thm}
{\bf Steepest Descent for the $\Phi-$Relative Entropy:} 
\label{Steep_Desc_Gen}
Under the  detailed-balance   conditions (\ref{A24}),    the curve $\,  ( P (t)  )_{t_0 \le t < \infty}\,$ of time-marginal distributions   in (\ref{A19}) has the property of   steepest descent   in   Definition \ref{SD} for the $\Phi-$Relative Entropy   of (\ref{A51b}), locally at $t=t_0\,,$  and  with respect to the distance induced by the ``flat metric"
\begin{equation} 
\label{Metric_Dist_Too}
\varrho_\star \big( P_1, P_2 \big) 
	\,:= \, \big \| \, {\bm \ell}_1 - {\bm \ell}_2 \,\big 
		  \|_{\bH^{-1}_\Theta(\cS, \ell Q)}  \qquad \text{ for  \, $P_1 = {\bm \ell}_1 Q$ \,  and $~P_2 = {\bm \ell}_2 \,Q.$}
\end{equation}
Here again we have ${\bm \ell} \equiv {\bm \ell} (t_0)$.
\end{thm}

\noindent
{\it Proof:} 
This is proved  exactly as in subsection \ref{justitia}, with the caveat that the distance-inducing flat metric is now determined ``locally", that is, depends on   $(t_0, {\bm \ell}) \equiv (t_0,{\bm \ell} (t_0))$ in the weighted norms of (\ref{A56})--(\ref{D11_too}).  We   go through the argument again, however, in order   to highlight the role that these weighted norms   play in the present,   more general  context. From (\ref{A57}), and recalling the initial position $\,\ell^\psi (t_0) =   {\bm \ell} (t_0) \in {\cal L}\,,$ we obtain
$$
\lim_{h \downarrow 0} \frac{\,H^\Phi \big(P^\psi (t_0 +h)\,|\,Q\big) - H^\Phi \big(P (t_0)\,|\,Q\big)}{h} 
\,=\, \,\Big \langle    \,    \varphi({\bm \ell }_{t_0}) ,    \psi_{t_0}   \Big \rangle_{\bH^1_\Theta(\cS, \ell Q)}  \,;
$$
whereas  (\ref{A61}) gives
$$
\lim_{h \downarrow 0} \, \frac{\,\varrho_\star \big( P^\psi (t_0+h), P (t_0) \big)\,}{h} \,=\, \big \|       \psi_{t_0}   \big \|_{\bH^1_\Theta(\cS, \ell Q)} \,,
$$
thus 
\begin{equation}
\label{inner}
\lim_{h \downarrow 0} \frac{\, H^\Phi \big(P^\psi (t_0 +h)\,|\,Q\big)
	 - H^\Phi \big(P   (t_0)\,|\,Q\big)}{\,\varrho_\star \big( P^\psi (t_0+h), P (t_0) \big)\, } \,=\,   \,\bigg \langle    \,   \varphi( {\bm \ell }_{t_0}) ,  \frac{  \psi_{t_0}  }{\, \big \|       \psi_{t_0}   \big \|_{\bH^1_\Theta(\cS, \ell Q)} \,}   \bigg \rangle_{\bH^1_\Theta(\cS, \ell Q)} \,.
\end{equation}
This is the rate of change for the $\Phi-$relative entropy    along  the  {\it perturbed curve} $\,  \big( P^\psi (t)  \big)_{t_0 \le t < t_0 + \varepsilon}\,,$  as measured on the manifold ${\cal M}$ with respect  to the   distance in (\ref{Metric_Dist_Too}). 

\smallskip
On the other hand,  we have from (\ref{A54}), (\ref{A53})  and (\ref{A55}),  the following observation: Along the {\it original curve}  of time-marginal distributions $\,  ( P (t)  )_{t_0 \le t < \infty}\,$ for the Chain, corresponding to taking   $    \psi  (\cdot) \equiv \varphi({\bm \ell}) (\cdot)   $   above, the rate of $\Phi-$relative entropy  dissipation measured in terms of the ``flat metric"   distance traveled   on  the manifold ${\cal M},  $ is given as
$$
\lim_{h \downarrow 0} \frac{\, H^\Phi \big(P  (t_0 +h)\,|\,Q\big) - H^\Phi \big(P  (t_0)\,|\,Q\big)}{\,\varrho_\star \big( P  (t_0+h), P  (t_0) \big)\, } \,=\, -  \,\big \|   \varphi(    {\bm \ell }_{t_0} )  \big \|_{\bH^1_\Theta(\cS, \ell Q)} \,<\,0\,.
$$
A simple comparison of the last two displays, via \textsc{Cauchy-Schwarz}, gives   the   steepest descent  property
$$
\lim_{h \downarrow 0} \frac{\, H^\Phi \big(P^\psi (t_0 +h)\,|\,Q\big) - H^\Phi  \big(P   (t_0)\,|\,Q\big)}{\,\varrho_\star \big( P^\psi (t_0+h), P (t_0) \big)\, } \,- \,\lim_{h \downarrow 0} \frac{\, H^\Phi \big(P  (t_0 +h)\,|\,Q\big) - H^\Phi \big(P  (t_0)\,|\,Q\big)}{\,\varrho_\star \big( P  (t_0+h), P  (t_0) \big)\, }  
$$
$$
= \,  \big \|   \varphi (   {\bm \ell }_{t_0}  ) \big \|_{\bH^1_\Theta(\cS, \ell Q)} + \bigg \langle    \, \varphi (   {\bm \ell }_{t_0} ) ,  \frac{  \psi_{t_0}  }{\, \big \|       \psi_{t_0}   \big \|_{\bH^1_\Theta(\cS, \ell Q)} \,}   \bigg \rangle_{\bH^1_\Theta(\cS, \ell Q) } \,\ge \, 0 
$$
 of the $\Phi-$relative entropy  with respect  to the  distance in   (\ref{Metric_Dist}),     along the original curve of   \textsc{Markov} Chain time-marginals.  Equality holds if, and only if, $\nabla \psi_{t_0}$ is a   negative   constant multiple of $\nabla  \varphi ({\bm \ell }_{t_0})$. \qed

\subsubsection{Non-uniqueness of the Flat Metric} 

There exist norms other than $\,\bH^{-1}_\Theta(\cS, {\bm\ell} Q)$ of (\ref{D11_too}), for which Theorem \ref{Steep_Desc_Gen} remains valid; see \textsc{Dietert} (2015) and Proposition \ref{Diet} below. Here we exhibit an explicit example.

\smallskip
Fix $\ell \in \cL_+(\cS)$ and consider the {\it ``modified weighted $\,\mathbb{H}^{-1}-$norm''}  given by 
\begin{align}
\label{eq:H-1-new}
	\big \| f \big\|_{{\widetilde \bH}^{-1}_\Theta(\cS, \ell Q)}^2
\,	:= \,\Bip{ \frac{1}{\vartheta_\ell} \nabla \big(\cK^{-1} f\big), \nabla \big( \cK^{-1} f \big) }_{\bL^2(\cZ,C)}
\end{align} 
for functions $f : \cS \to \R$ with $\,\sum_{x \in \cS} f(x) q(x) = 0\,$. This norm is never smaller than the original $\bH^{-1}_\Theta(\cS, \ell Q)-$norm as defined in (\ref{A58}); namely, 
\begin{align}
\label{eq:H-1-ineq}
\big \| f  \big \|_{{\widetilde \bH}^{-1}_\Theta(\cS, \ell Q)} \, \geq \,	\big \| f  \big \|_{\bH^{-1}_\Theta(\cS, \ell Q)}  \, .
\end{align}
And equality holds when $f = \cK \ell\,;$ to wit, 
\begin{align}
\label{eq:H-1-eq}
	\big \| \cK \ell \big \|_{{\widetilde \bH}^{-1}_\Theta(\cS, \ell Q)} \,= \,	\big \| \cK \ell \big \|_{\bH^{-1}_\Theta(\cS, \ell Q)}.
\end{align}
These two facts   imply that the curve $\,  ( P (t)  )_{t_0 \le t < \infty}$ from Theorem \ref{Steep_Desc_Gen}, which corresponds to  the  original backward   equation $ \,\partial {\bm \ell}_t = \nabla \cdot ( \nabla {\bm \ell}_t ) = {\cal K} {\bm \ell}_t $ of (\ref{A22}),   is   a curve of steepest descent also  with respect to the  modified  ${\widetilde \bH}^{-1}_\Theta(\cS, \ell Q)-$norms  in (\ref{eq:H-1-new}).

To prove the inequality \eqref{eq:H-1-ineq}, we use Proposition \ref{prop:H-min-char} and the identity $\cK f = \nabla \cdot ( \nabla f)$ to obtain 
\begin{equation}
\begin{aligned}
\label{eq:norm-ineq}
	\big \|  f  \big \|_{\bH^{-1}_\Theta(\cS, \ell Q)}^2 
		& 
		= \inf_{G : \cZ \to \R } 
    \bigg\{ \bip{ G, \vartheta_\ell G }_{\bL^2(\cZ,   
    C)} 
    	 \ : \  f + \nabla \cdot \big(\vartheta_\ell G\big)=0   \bigg \}
	 \\& \leq 
	 \Bip{\, \frac{1}{\vartheta_\ell} \nabla \big( \cK^{-1} f \big), \, \vartheta_\ell\Big( \frac{1}{\vartheta_\ell}  \nabla \big( \cK^{-1} f \big) \Big) }_{\bL^2(\cZ,   
	 C)}
 	= \big \| f \big \|_{{\widetilde \bH}^{-1}_\Theta(\cS, \ell Q)}^2 \, .
\end{aligned}
\end{equation}
 On the one hand, the equality \eqref{eq:H-1-eq} holds for $f = {\cal K} \ell ,$ since 
\begin{align*}
	\big\| \cK \ell \big\|_{{\widetilde \bH}^{-1}_\Theta(\cS, \ell Q)}^2
	 = \Bip{ \frac{1}{\vartheta_\ell} \nabla \ell , \nabla \ell }_{\bL^2(\cZ,C)}
	& = \Bip{ \nabla \varphi(\ell) , \nabla \ell }_{\bL^2(\cZ,C)}
	\\& = \Bip{ \nabla \varphi(\ell) , \vartheta_\ell \nabla \varphi(\ell) }_{\bL^2(\cZ,C)}
	 =  \big\|  \varphi(\ell) \big\|_{\bH^1_\Theta(\cS, \ell Q)}^2\,\, ;
\end{align*}
while, on the other hand, Proposition \ref{prop:H-min-char} yields
\begin{align*}
	\big\| \cK \ell \big\|_{\bH^{-1}_\Theta(\cS, \ell Q)}^2
	= 	\big\| \nabla \cdot ( \nabla \ell) \big\|_{\bH^{-1}_\Theta(\cS, \ell Q)}^2
	= 	\big\| \nabla \cdot \big(\vartheta_\ell \nabla \varphi(\ell)\big) \big\|_{\bH^{-1}_\Theta(\cS, \ell Q)}^2
	= \big\|  \varphi(\ell) \big\|_{\bH^1_\Theta(\cS, \ell Q)}^2 \,\, .
\end{align*}

 \begin{rem}
In general, the norms $\big \|  f  \big \|_{{\widetilde \bH}^{-1}_\Theta(\cS, \ell Q)}$ and $\big \|  f  \big \|_{\bH^{-1}_\Theta(\cS, \ell Q)}$ are different.
Indeed, it follows from Proposition \ref{prop:H-min-char} and \eqref{eq:norm-ineq} that  equality of norms  holds  if, and only if, $\frac{1}{\vartheta_\ell} \nabla \big( \cK^{-1} f \big)$ is a discrete gradient.  This is in general false, but   {\it is} true in the following very special cases: 
\begin{itemize}
\item At   equilibrium, i.e., with $\,\ell \equiv 1,$ we have $\,\vartheta_\ell \equiv 1$, so that $\,\frac{1}{\vartheta_\ell} \nabla \big( \cK^{-1} f \big) =  \nabla \big( \cK^{-1} f \big)$;
\item For the variance functional $\Phi (\xi) = \frac12( \xi^2 - 1)$ as   in Section \ref{sec5},     $\vartheta_\ell \equiv 1$ for every likelihood ratio $\ell\,$; 
\item If the state space $\cS$ consists of only two points, $\frac{1}{\vartheta_\ell} \nabla \big( \cK^{-1} f \big) $ is a discrete gradient, since this holds for every anti-symmetric function on $\,\cS \times \cS$.
\end{itemize}
\end{rem}

\section{Gradient Flows}
\label{sec8}

Let us reconsider now, {\it   under conditions of detailed balance,}  the results of Sections \ref{sec5}--\ref{sec8} from a different,   ``Riemannian''  point of view.  We shall see here   that,   {\it under  the conditions    (\ref{A24}), the curve $ ( P(t) )_{0 \le t < \infty}$ of time-marginal distributions     for  the Chain  evolves as a   gradient flow  of the relative $\Phi-$entropy.} This takes place in a suitable geometry on the space of probability measures,    in the spirit of the pioneering work by \textsc{Jordan, Kinderlehrer \& Otto} (1998).  We refer to  \textsc{Erbar  \& Maas} (2012, 2014), \textsc{Mielke} (2011, 2013) and to the expository paper \textsc{Maas} (2017), for an in-depth study of such issues in discrete spaces.

 \smallskip
We summon  from subsection \ref{sec2.1} the manifold $\, {\cal M = P_+ (S)}  \,$ of  probability vectors $  P = \big( p (x) \big)_{x \in {\cal S}}$ with strictly positive entries;      i.e., $ \, {\cal M }\,   $ is  the interior of the lateral face of the unit simplex in $\R^{n},$ with $\,n=|{\cal S}| $ the cardinality of the state-space.    
 We denote by $ {\cal M}_0 ({\cal S}) $ the collection of vectors $\, W = \big( w (x) \big)_{x \in {\cal S}}\,$ with total mass $\, \sum_{x \in {\cal S}} w (x)=0\,,$  viewed as ``signed measures", and observe that $\, {\cal M}\,$ is a relatively open subset of the $(n-1)-$dimensional affine space $\, P + {\cal M}_0 ({\cal S})= \{ P+W : W \in{\cal M}_0 ({\cal S}) \} ,$  for an arbitrary $\, P \in {\cal M}\,$. This observation allows us to identify the tangent space   at each $\, P \in {\cal M}\, $ with $\, {\cal M}_0 ({\cal S})\,$.

\subsection{Gradient Flow for  the Variance
}
\label{sec8a}

As a warmup, let us start as in Section \ref{sec5} with a derivation for the gradient flow property for the variance functional $\, {\cal M} \ni P \mapsto V(P|Q)\in \R_+ $ of (6.1). Following \textsc{de\,Giorgi}'s approach to  curves of maximal slope  (cf.\,\textsc{Ambrosio, Gigli \& Savar\'e}\,(2008)), we compute the dissipation of this functional along an arbitrary smooth curve $( \widetilde{P}_t)_{0 \le t < \infty} $ on ${\cal M} $; or equivalently, along the   curve $( \widetilde{\ell}_t)_{0 \le t < \infty} $ induced on the space $ {\cal L} $ by the likelihood ratios   $\,\widetilde{\ell}_t (y) = \widetilde{p}_t ( y)/ q(y), ~ y \in {\cal S}$.

  As in Section \ref{sec5},   we express the time-evolution of this likelihood ratio curve as $\,\partial \widetilde{\ell}_t = {\cal K} f_t= \nabla \cdot \big( \nabla f_t \big)\,$ in the manner of (\ref{A22}), for a suitable curve $\, ( f_t )_{0 \le t < \infty}\,$  of functions   $f_t : {\cal S} \to \R\,.$ This is uniquely determined up to an additive constant  on account of the Chain's irreducibility, and its discrete gradient provides the ``momentum vector field" of the motion. Recalling the consequences  $\widehat{{\cal K}}f = {\cal K}f = \nabla \cdot ( \nabla f )$ of detailed  balance   (\ref{A24}) and of (\ref{D3}),  as well as the fact that $\,\nabla \cdot \,$ is the adjoint of $- \nabla$ from (\ref{D6}), we obtain 
\begin{equation}
\begin{aligned}
  \partial V \big( \widetilde{P}_t \big| Q \big)
 	 & = \partial \big\| \widetilde{\ell}_t \big\|^2_{\bL^2(\cS,Q)} 
	 = 2 \bip{ \widetilde{\ell}_t , \partial \widetilde{\ell}_t}_{\bL^2(\cS,Q)} 
	 = 2 \big \langle \widetilde{\ell}_t , {\cal K} f_t \big \rangle_{\bL^2(\cS,Q)} 
	= -2 \bip{ \nabla \widetilde{\ell}_t , \nabla f_t }_{\bL^2(\cZ,C)}
\\ &	 \ge \, -2\, \big \| \nabla \widetilde{\ell}_t \big \|_{\bL^2(\cZ,C)} \, \big \| \nabla f_t \big \|_{\bL^2(\cZ,C)}
 \label{MS1}
  \,\ge \, - \, \big \| \nabla \widetilde{\ell}_t \big \|^2_{\bL^2(\cZ,C)} \,-\, \big \| \nabla f_t \big \|^2_{\bL^2(\cZ,C)}\,.
\end{aligned}\end{equation}
Equality holds in the first (resp., the second) of these inequalities if, and only if, $\nabla f_t $ and $\nabla \widetilde{\ell}_t$ are positively collinear (resp., have the same norm). In other words,   both these last two inequalities hold  as equalities if and only if $\,\nabla f_t =\nabla \widetilde{\ell}_t$, and this leads to the backwards equation (\ref{A22}) on account of detailed balance: 
$$
\partial \widetilde{\ell}_t  = {\cal K} f_t = \nabla \cdot \big( \nabla f_t \big) =\nabla \cdot \big( \nabla \widetilde{\ell}_t \big) = {\cal K} \,\widetilde{\ell}_t = \widehat{{\cal K}} \,\widetilde{\ell}_t\,.
$$
But the last two norms  in (\ref{MS1}) are then $ \,\| \nabla f_t \|_{\bL^2(\cZ,C)} = \|  \nabla ({\cal K}^{-1} (\partial \widetilde{\ell}_t)) \|_{\bL^2(\cZ,C)} = \| \partial \widetilde{\ell}_t  \|_{\mathbb{H}^{-1}(\cS, Q)}\,$ as well as  $\,  \| \nabla  \widetilde{\ell}_t   \|_{\bL^2(\cZ,C)} = \|   \widetilde{\ell}_t   \|_{\mathbb{H}^{1}(\cS,Q)} \,.$

\smallskip
In this manner we obtain from (\ref{MS1}) the following classical result. This provides another proof for Theorem \ref{Steep_Desc_Var} by identifying  the solutions of   $\, \partial P_t = {\cal K}^\prime P_t\,$   in (\ref{A20})  as  curves in the direction of steepest descent for the variance,    relative to the   distance induced by the $\bH^{-1}(\cS,Q)$ norm. But it also strengthens Theorem \ref{Steep_Desc_Var}, by identifying also the correct velocity  with which the gradient flow moves into this direction.

\begin{thm}
\label{prop9.1}
   For any given probability vector $P \in {\cal M} $  and  with  $\,{\bm \ell} \in {\cal L}\,$   the  likelihood ratio vector   corresponding to $P$,  we have  along any smooth curve $( \widetilde{P}_t)_{0 \le t < \infty} $ on ${\cal M} $ with $\widetilde{P}_0=P $ the inequality  
$$
\bigg( \partial V \big( \widetilde{P}_t \big| Q \big)+ \big \| \partial \widetilde{\ell}_t  \big \|^2_{\mathbb{H}^{-1}(\cS,Q)} \bigg)\bigg|_{t=0} \,\ge \, - \, \big \|     {\bm \ell}  \big \|^2_{\mathbb{H}^{1}(\cS,Q)} \, .
$$
 Equality holds   if, and only if, the curve $( \widetilde{P}_t)_{0 \le t < \infty} \subset {\cal M} $ satisfies the forward equation $ \partial \widetilde{P}_t = {\cal K}^\prime \widetilde{P}_t$ $($equi- valently, the  induced likelihood ratio  curve $( \widetilde{\ell}_t)_{0 \le t < \infty} \subset {\cal L}  $ satisfies the backward equation $ \partial \widetilde{\ell}_t = {\cal K}  \widetilde{\ell}_t).$
  \end{thm}

\subsection{Gradient Flow for the $\Phi-$Relative Entropy
}
\label{sec8b}
Let us examine now, how these ideas might work in the context of the generalized relative entropy functional 
   \begin{equation} 
\label{Z13}
{\cal M} \ni P \, \longmapsto \,   H^\Phi \big( P     \big|  Q \big) \, :=\,   \sum_{y \in {\cal S}} \, q(y) \, \Phi \Big( \frac{p ( y)}{q(y)} \Big)  \, \in \, [0, \infty)
\end{equation}
     corresponding to a convex function $\Phi,$  as in Section \ref{Gen}. We fix a smooth curve $( \widetilde{P}_t)_{0 \le t < \infty} $ on ${\cal M} $ emanating from a given $\widetilde{P}_0=  P \in {\cal M};$ and consider  the   induced curve $( \widetilde{\ell}_t)_{0 \le t < \infty} \subset {\cal L} $   of likelihood ratios $\widetilde{\ell}_t (y) = \widetilde{p}_t ( y)/ q(y), ~ y \in {\cal S}\,$ emanating from $\,{\bm \ell} = \ell _0$.

As in subsection \ref{subsec:LocSteep},   we cast the time-evolution of the likelihood ratio curve as a {\it continuity equation}  
    \begin{equation} 
    \label{Cont_Eq}
  \partial \widetilde{\ell}_t \,+\, \nabla \cdot \big( \widetilde \vartheta_t \nabla f_t \big) = 0
  \end{equation} 
where the ``velocity vector field'' is the discrete gradient of                                                                                      a suitable function $\,f_t : {\cal S} \to \R\,$, and   $\widetilde \vartheta_t$ is a shorthand for $\vartheta_{\widetilde{\ell}_t}$  from  \eqref{eq:def-vartheta}.   In the manner of (\ref{MS1}), this expresses the time-evolution of the $\Phi-$relative entropy  functional      $\,H^\Phi \big( \widetilde{P}_t \big| Q \big) = \sum_{y \in {\cal S}} \, q(y) \,\Phi \big(\widetilde{\ell}_t (y) \big) $   in  (\ref{Z13}) along the curve $\,\big( \widetilde{P}_t\big)_{0 \le t < \infty} \,  $  as 
\begin{equation}
\begin{aligned}
 \label{MS2}
\partial H^\Phi \big( \widetilde{P}_t \big| Q \big)
 & =   \big \langle \varphi ( \widetilde{\ell}_t ), \partial \widetilde{\ell}_t \big \rangle_{\bL^2(\cS,Q)} 
 = - \big \langle \varphi ( \widetilde{\ell}_t ), \nabla \cdot ( \widetilde\vartheta_t \nabla f_t ) \big \rangle_{\bL^2(\cS,Q)} 
 =    \bip{ \,\nabla   \varphi ( \widetilde{\ell}_t )  , \widetilde\vartheta_t\nabla f_t \, }_{\bL^2(\cZ,C)}
\\ & 
=    \bip{ \,\nabla  \varphi ( \widetilde{\ell}_t )  ,  \nabla f_t \, }_{\bL^2(\cZ, \widetilde\vartheta_t C)}
 \ge \, -  \big \| \nabla  \varphi ( \widetilde{\ell}_t )   \big \|_{\bL^2(\cZ, \widetilde\vartheta_t C)} \, \big \| \nabla f_t \big \|_{\bL^2(\cZ, \widetilde\vartheta_t C)}  
  \\&\ge   - \, \Big( \big \| \nabla \varphi ( \widetilde{\ell}_t ) \big \|^2_{\bL^2(\cZ, \widetilde\vartheta_t C)} + \big \| \nabla f_t \big \|^2_{\bL^2(\cZ, \widetilde\vartheta_t C)}\Big)\Big/2\,.
\end{aligned}
\end{equation}
Once again, equality holds if and only if $\,\nabla f_t =\nabla  \big( \varphi (\widetilde{\ell}_t) \big),$  and this leads  by detailed balance to the backwards equation 
$$
\partial \widetilde{\ell}_t 
	= - \nabla \cdot \big( \widetilde\vartheta_t \,\nabla f_t \big) 
	= \nabla \cdot \big(
		 \widetilde\vartheta_t \,\nabla (\varphi ( \widetilde{\ell}_t ))
		 			\big) 
	= \nabla \cdot \big( \nabla \widetilde{\ell}_t \big) 
	= {\cal K} \,\widetilde{\ell}_t 
	= \widehat{{\cal K}} \,\widetilde{\ell}_t
$$
of  (\ref{A22}). We have used here the elementary but crucial  consequence  
$\,
\widetilde\vartheta_t \,\nabla (\varphi ( \widetilde{\ell}_t ))= \nabla \widetilde{\ell}_t\, 
 $ 
of (\ref{eq:def-vartheta}),   a ``discrete chain rule" that sheds   light on our choice of   weight-function $\Theta^\Phi$ in (\ref{A51ab}). But the last two norms displayed in (\ref{MS2}) are $ \,\| \nabla f_t \|_{\bL^2(\cZ, \widetilde\vartheta_t C)} = \| \partial \widetilde{\ell}_t  \|_{_{\bH^{-1}_\Theta(\cS, \widetilde\ell_t Q)}}\,$ and  $ \, \| \nabla \varphi ( \widetilde{\ell}_t )  \|_{\bL^2(\cZ, \widetilde\vartheta_t C)}=\|  \varphi ( \widetilde{\ell}_t )   \|_{\bH^1_\Theta(\cS, \widetilde\ell_t Q)} \,.$  

\smallskip
We summarize the situation in Theorem \ref{prop9.2} below; this corresponds to Theorem \ref{Steep_Desc_Gen},  in the same manner as Theorem \ref{prop9.1} corresponds to Theorem \ref{Steep_Desc_Var}. Again, the \textsc{de Giorgi} argument (\ref{MS2}) gives not only the ``direction of steepest descent" into which the gradient flow travels, but also the velocity of this flow.

\begin{thm}
\label{prop9.2}
   For any given probability vector $P \in {\cal M},$  and  with  $\,{\bm \ell} \in {\cal L}\,$   the  likelihood ratio vector   corresponding to $P$,  we have  along any smooth curve $( \widetilde{P}_t)_{0 \le t < \infty} $ on ${\cal M} $ with $\widetilde{P}_0=P $ the inequality  

\begin{equation}
\label{9.4a}
\bigg( 2\, \partial H^\Phi \big( \widetilde{P}_t \big| Q \big)+ \big \| \partial \widetilde{\ell}_t  \big \|^2_{\bH^{-1}_\Theta(\cS, \bm\ell Q)} \bigg)\bigg|_{t=0} \,\ge \, - \, \big \|   \varphi (  {\bm \ell} ) \big \|^2_{\bH^{1}_\Theta(\cS, \bm\ell Q)} \, .
\end{equation}
\noindent
  Equality holds here if, and only if, the curve $( \widetilde{P}_t)_{0 \le t < \infty} \subset {\cal M} $ satisfies the forward equation $\,\partial \widetilde{P}_t = {\cal K}^\prime \widetilde{P}_t\,$ $($equivalently, the likelihood ratio  curve $\,( \widetilde{\ell}_t)_{0 \le t < \infty} \subset {\cal L} \, $ satisfies the backward equation $\,\partial \widetilde{\ell}_t = {\cal K} \, \widetilde{\ell}_t \,,$ and the corresponding ``driver" in (\ref{Cont_Eq}) is $\, f_t = -   \varphi (\widetilde{\ell}_t)\,  .)$  
  \end{thm}

\subsection{A   Riemannian    Framework}
\label{sec8c}

 Let us take up these same ideas again, but now in a  Riemannian-geometric framework  as for instance in  
 \textsc{Maas} (2011), \textsc{Mielke} (2011). For any given probability vector $\, P \in {\cal M}\,, $ we   define the  ``likelihood ratio" vector $\,   {\bm \ell}  = \big( \ell (x) \big)_{x \in {\cal S}} \in \cL$ with strictly positive elements $\, \ell (x) := p(x) / q(x)$.   We consider then the Riemannian metric $(g_{\bm \ell})_{{\bm \ell} \in \cL}$ on $\cL$ induced by the scalar products 
 $$
 \, \,g_{\bm \ell} (\partial \ell_1, \partial \ell_2 ) \,:= \,
\bip{ \nabla \psi_1, \nabla \psi_2 }_{\bL^2(\cZ, \vartheta_\ell C)} \,  ,  
$$
where $\nabla \psi_i$ is the unique discrete gradient satisfying the equation of ``continuity type"   $\partial \ell_i = \nabla \cdot ( \vartheta_\ell \nabla \psi_i ) $ for $i=1, 2$.
In particular, $g_{\bm \ell}(\partial \ell, \partial \ell) = \| \nabla \psi \|^2_{\bL^2(\cZ, \vartheta_\ell C)}= \|\partial \ell\|_{\bH^{-1}_\Theta(\cS, \bm\ell Q)}^2$ on account of (\ref{eq:H-min-char}).

\medskip
The Riemannian gradient $\grad F$ of a smooth functional $F : \cL \to \R$ is then given by 
\begin{align*}
 \grad F = - \nabla \cdot \Big(\vartheta_\ell  \,\,\nabla D_\ell F \Big) \,, \qquad \text{where} \quad D_\ell F  \, \equiv  \,  \frac{\delta F}{\delta \ell}
\end{align*}
is the $\bL^2(\cS,Q)$-derivative defined by $\lim_{\varepsilon \to 0} \varepsilon^{-1}\big( F(\ell + \varepsilon \eta ) - F(\ell) \big) =   \ip{ D_\ell F, \eta}_{\bL^2(\cS,Q)}$ for $\eta : \cS \to \R$ with $\sum_{x \in \cX} \eta(x) q(x) = 0$.   
In particular, the gradient flow equation $\,\partial \ell = - \grad F(\ell)\,$ reads   
\begin{align}\label{eq:gf-eq}
	\partial \ell = \nabla \cdot \Big(\vartheta_\ell  \,\, \nabla D_\ell F \Big)  \ .
\end{align}

The Riemannian metric $g$ on $\cL$ can   be turned into a Riemannian metric $G$ on the manifold of probability measures $\cM$, via $\,G_P(\partial P_1, \partial P_2) := g_{\bm \ell} (\partial \ell_1, \partial \ell_2)$, where $P = {\bm \ell}\, Q$ and $\,\partial P_i  =  \partial \ell_i \, Q\,$ for $i = 1,2$.

 \begin{thm} {\bf (\textsc{Maas} (2011),  \textsc{Mielke} (2011)):}  
\label{JM}
Under the detailed balance conditions (\ref{A24}), and with $  \Theta $ the function  of (\ref{A51ab}), the Forward \textsc{Kolmogorov} equation  $\,\partial P(t)= {\cal K}^\prime P(t)\,$ in  (\ref{A20}) is the gradient flow of the $\Phi-$relative entropy in (\ref{Z13}) with respect to the Riemannian metric $G$ induced on the manifold $    {\cal M}  .$  
  \end{thm}

\noindent
  {\it Proof:} 
Let $\,( P(t))_{0 \le t < \infty}$ solve the Forward \textsc{Kolmogorov} equation  $\,\partial P(t)= {\cal K}^\prime P(t)\,$. 
By detailed balance, the associated likelihood ratio curve $\,\big({\bm \ell} (t)\big)_{0 \le t < \infty} \subset {\cal L} \, $ satisfies the backward equation $\,\partial {\bm \ell} (t) = {\cal K} \,{\bm \ell} (t)$.
In view of \eqref{eq:gf-eq}, we thus need to verify the identity
\begin{align*}
 	{\cal K} \,{\ell} = \nabla \cdot \Big(\vartheta_\ell \, \nabla D_\ell  h^\Phi \Big)  \ ,   
\end{align*}
where $h^\Phi : \cL \to \R$ is defined by $h^\Phi(\ell) = H^\Phi(\ell Q|Q)$.

 \smallskip
  For $\,{\bm \ell} \in \cL\,$ and $\,\eta : \cS \to \R\,$ with $\,\sum_{x \in \cS} \eta(x) q(x) = 0\,,$ we have the directional derivative computation 
\begin{equation} 
\label{E6}
  \frac{\ud ~}{\ud \varepsilon} \, h^\Phi \big( {\bm \ell} + \varepsilon \eta \, \big) \, \bigg|_{\varepsilon = 0}\,=\,\sum_{x \in {\cal S}} \, \eta(x)   \,\varphi \big( \ell(x) \big);\quad ~~~~\text{thus}~ ~~~~~  D_\ell  h^\Phi \equiv 
  \frac{\delta h^\Phi}{\delta \ell}
  = \varphi (  {\bm \ell} ) : = \Big( \varphi ( \ell (x) ) \Big)_{x \in {\cal S}} .   
\end{equation} 
Invoking  the ``discrete chain-rule'' 
$\,
\vartheta_\ell \,\nabla (\varphi ( {{\bm \ell}} ))= \nabla{\ell}
 $
we obtain the desired identity
\begin{align*}
~~~~~~~~~~~~~~~~~~~~~~~~~~~~~~~
\nabla \cdot \Big(\vartheta_\ell \, \nabla  D_\ell  h^\Phi  \Big) 
 = \nabla \cdot \Big(\vartheta_\ell \, \nabla (\varphi({\bm \ell}))\Big) 
 = \nabla \cdot \big( \nabla \ell \big) 
 = \cK \ell  .
  \qquad  \qquad  \qquad \qquad   \quad \qed
\end{align*}

 \smallskip

 \smallskip

 Theorem \ref{JM} has a converse,  developed in \textsc{Dietert} (2015) as follows.

  \begin{prop}
\label{Diet}
Suppose that there exists a $\,{\cal C}^1$ Riemannian metric on the manifold of probability vectors  ${\cal M} ,\,$ under which the Forward \textsc{Kolmogorov} equation $\,\partial P(t)= {\cal K}^\prime P(t)\,$ of (\ref{A20}) is the gradient flow   for the relative entropy in (\ref{H1}). Then the \textsc{Markov} Chain satisfies the detailed balance conditions (\ref{A24}).
  \end{prop}

\subsection{The HWI Inequality}
\label{sec8d}

In the Riemannian framework of this Section, we present  
now a version of the celebrated HWI inequality  of \textsc{Otto  \& Villani} (2000). The basic ingredient is the notion of \textsc{Ricci} curvature in the present context, as   in Definition 1.3 of \textsc{Maas} (2011). We recast this definition    using the more general notion of $\Phi$-entropy   in Section \ref{sec7} -- rather than the classical entropy  which is, of course, a special case. 
  We recall also from subsection 3.1, Remark \ref{rem5.1} the manifold ${\cal M}$  of probability vectors on ${\cal S}$ with strictly positive entries, its closure $\overline{{\cal M}}$  of probability vectors with nonnegative entries, and the corresponding manifolds  ${\cal L}\,,$   $\overline{{\cal L}}$  of likelihood ratios.  

\begin{defn} {\bf Ricci$^\Phi$-curvature:} 
\label{RC}
   We say that our finite-state \textsc{Markov} Chain with generator $\cK$   has {\it non-local \textsc{Ricci} curvature bounded from below by $\kappa \in \R$   relative to $\Phi$} as above, and write \,Ricci$^\Phi(\cal K)\ge \kappa,\,$  if   for every constant-speed geodesic $(P_t)_{0 \le t \le 1}$ on the closed manifold  $\overline {\cal M}$ we have       the inequality   
\begin{equation}
\label{quadratic}
H^\Phi \big(P_t \big| Q\big) \le (1-t) H^\Phi \big(P_0 \big| Q\big) + tH^\Phi \big(P_1 \big| Q\big) - \frac{\,\kappa \,}{2}  \,t  (1-t)\,     \mathcal{W}^2   (P_0,P_1) ,\quad 0 \le t \le 1\,.
\end {equation}
\end{defn}

 Here  $\cal {W} (\cdot \,, \cdot) $ is   the geodesic distance with respect to the Riemannian metric of subsection 9.3. It admits   the \textsc{Benamou-Brenier}-type representation 
\begin{align}
  \mathcal{W}^2(P_0, P_1)
  \,= \,\inf
    \bigg\{ 
      \int_0^1 \big \| f_t  \big \|^2_{\bH^1_\Theta(\cS,  \widetilde\ell_t Q)} \, \ud t
      \ : \
      \partial \widetilde{\ell}_t \,+\, \nabla \cdot \big( \widetilde \vartheta_t \nabla f_t \big) = 0
    \bigg\},
\end{align}
with the infimum running over all solutions  to the continuity equation connecting $\,P_0 \equiv   \widetilde{\ell}_0 Q$ with $\,P_1 \equiv   \widetilde{\ell}_1 Q\,$; cf.\,\textsc{Maas} (2011),  \textsc{Erbar \& Maas} (2012), \textsc{Mielke} (2013). We shall apply the above inequality (\ref{quadratic}) in the form of the following  fact about functions of a real variable.

  \begin{prop}
\label{secondderivative}
Let   $\big( f(t) \big)_{0 \le t \le1}$ be a continuous, real-valued function such that 
\begin {equation}
\label{discr}
f(t+h) - 2f(t) + f(t-h) \ge \kappa h^2 
\end {equation}
\noindent 
  holds for some $\kappa \in \R$  and  every   pair $(t,h) \in  \R_+^2$  with $\, h \le t \le 1 - h . $    Suppose also    that f is right-differentiable at $t=0$ with derivative $f'(0)$.
Then 
\begin {equation}
\label{concl}
f(1) \ge f(0) + f'(0) + \frac {\kappa} {2}.
\end {equation}
 \end{prop}
  
 \noindent
{\it Proof:}    If $f$ is twice differentiable, the condition (\ref{discr}) amounts to $f^{\prime \prime}  \ge \kappa$. For general $f$ and supposing   $\kappa = 0,$ condition (\ref{discr}) is tantamount to the convexity of $f$, so the inequality (\ref{concl}) becomes obvious. The case of general $\kappa$ follows by subtracting from $f(t)$ the quadratic $ \,\kappa \,  t^2 /2$. 
  \qed
 
 \medskip
 For a  constant-speed geodesic $(P_t)_{0 \le  t \le 1}$  joining $ P_0 \in \cal M$ with $P_1 \in \overline {\cal M}$  such that $\mathcal{W}    (P_0,P_1) =1$, the function $f(t)=H^{\Phi}(P_t | Q)$ satisfies the conditions of Proposition 9.5 under the assumption Ricci$^\Phi(\cal K)\ge \kappa$.    Indeed,    $(P_u)_{t-h \le u \le t+h}$ is then a constant-speed geodesic which joins $P_{t-h} $ with   $P_{t+h}$   and satisfies $\mathcal{W}    (P_{t-h},P_{t+h}) =2h$,   so   (\ref{quadratic}) applies with $t= 1/2 $.   The existence of   constant-speed geodesics and of $f'(0)$, follows respectively from Theorem 3.2 and   Proposition 3.4 in \textsc{Erbar  \& Maas} (2012).  
  
  \smallskip
    We   formulate now a version of the HWI inequality in the present context. This  sharpens slightly Theorem 7.3 of \textsc{Erbar  \& Maas} (2012), where   $P_1$ in the following  theorem is the invariant measure $\,Q\,$; and its proof       does not rely on the ``evolution variational inequality"   (the EVI of Theorem 4.5 in \textsc{Erbar  \& Maas} (2012)), but rather on the   very elementary  estimate of   Proposition \ref{secondderivative}.

 \begin{thm} {\bf  HWI Inequality of \textsc{Otto-Villani}:}  
\label{HWI} Under the assumptions of subsection 9.2,   suppose that  $ \, \textnormal{Ricci}^\Phi(\cal K) \ge \kappa \, $  holds   for some $\kappa \in \R$.    
  With $P_0,$   $P_1$    any probability measures in   $\cal M,$        $\overline {\cal M},$ respectively,  denote by $\mathcal{W} (P_0,P_1)$ their geodesic distance    and by $I^\Phi(P_0|Q)$   the $\Phi$-\textsc{Fisher} information   of (\ref{A54}) with $t=0$. We have then   
\begin{equation} 
\label{p5}
H^\Phi (P_0|Q) - H^\Phi(P_1|Q)\,\le \, \Big( I^\Phi(P_0|Q)\Big)^{1/2}  \ \mathcal{W} (P_0,P_1) - \frac{\kappa}{2} \, \mathcal{W}^2  (P_0,P_1) .
\end{equation}
\end{thm} 

\noindent
{\it Proof:}   We follow the argument  in Theorem 4.11 of  \textsc{Karatzas,   Schachermayer  \& Tschiderer}  (2020), where the HWI inequality is established for  diffusions in $\R^n$.    We let $(P_t)_{0 \le t \le1} \subset \overline{{\cal M}}$ be a constant-speed geodesic   of probability measures joining $P_0$ with $P_1$ (which we know   exists, by Theorem 3.2 of \textsc{Erbar  \& Maas} (2012)),   denote by      $(\ell_t)_{0\le t\le 1} \subset \overline{{\cal L}}\,$   the corresponding   likelihood-ratio curve,  consider the function 
$ \,f(t):=H^{\Phi}(P_t | Q)\, ,  ~~  0 \le t \le1  ,$ and pass to the parametrization 
$$
  u=u(t)=\frac{w}{i^{1/2}}t\,   , \qquad 0 \le u \le \frac{w}{i^{1/2}},
 $$
where $\,i=I^\Phi (P_0)=\mathcal{E} (\ell_0,\phi(\ell_0))\,$ and $w= \mathcal{W} (P_0,P_1)$.
  We set $g(u) = g(u(t)) = f(t)$. Recalling  the likelihood ratio $\ell_t$ corresponding to $P_t\,,$   consider the continuous curve   of likelihood ratios
$$
  \widetilde \ell (u) = \ell_{t}\,  , \qquad 0 \le u \le \frac{w}{i^{1/2}} 
$$
so that $\widetilde \ell (0) = \ell_0$ and $\widetilde \ell (w\,i^{-1/2}) = \ell_1$, as well as the corresponding curve $\widetilde{P} (u), ~ 0 \le u \le w\, i^{- 1/2}$ of probabilities.   Since  $(P_t)_{0 \le t \le 1}$ is a   geodesic of constant speed   $w$    with $\,{\bm \ell} = \ell_0,$  we have
$$
\big \|  \partial \ell_0 \big \|_{H^{-1}_\Theta (\mathcal{S} , \ell Q)} = \mathcal{W} (P_0 , P_1 ) = w , \qquad \text{thus} \qquad \big \|   \partial \widetilde \ell (0)  \big \|^{2}_{H^{-1}_\Theta (\mathcal{S} , \ell Q)} = i \,;
$$
  this last display gives the second term in   (\ref{9.4a}).  As for the  term $  \, \big \|   \varphi (  {\bm \ell} ) \big \|^2_{\bH^{1}_\Theta(\cS, \bm\ell Q)}$ in (\ref{9.4a}),    the expression  (\ref{A54}) and Remark \ref{rem8.3}  give    $\, \big \|   \varphi (\widetilde {\bm \ell} (0))\big \|^2_{H^{1}_\Theta (\mathcal{S}, {\bm \ell} Q)} = \mathcal{E} (\ell_0, \varphi(\ell_0)) = i. $ 
In this manner, (\ref{9.4a}) leads to the  inequality 
\begin{equation}
 \label{ini1}
      g'(0)=\partial H^\Phi (\widetilde P_u | Q)\Big|_{u=0} \ge -i,     
    \end{equation} 
  where   the existence of the right-derivative $g'(0)$ is assured by Proposition 3.4 of \textsc{Erbar  \& Maas} (2012). 
  
  Going back  to the original parametrization, we  obtain    $   f'(0) \ge - w i^{1/2}.
  $   The assumption   Ricci$^\Phi ({\cal K})\ge \kappa$ implies that $f$  satisfies (\ref{discr}), with $\kappa$ replaced by $\kappa w^2$. In conclusion, (\ref{concl}) gives  
  $$
  H^\Phi (P_1 | Q) \ge H^\Phi (P_0 | Q) - i^{1/2} w + \frac{\kappa}{2} w^2,
  $$
 which is tantamount to the HWI inequality \eqref{p5}. \qed
  
 \begin{rem}
 As is well known (e.g., \textsc{Erbar  \& Maas} (2012)), the HWI inequality leads directly to the corresponding versions of the Modified Log-\textsc{Sobolev} and \textsc{Talagrand} inequalities,   by taking $  \Phi (\cdot) = \Psi (\cdot)  $ as in (7.13)     and $P_1 = Q$.    \textsc{Poincar\'{e}}-type inequalities also follow this way, by linearizing the Modified Log-\textsc{Sobolev} inequality.   
 \end{rem}

 \smallskip
 The HWI inequality \eqref{p5} can be sharpened. In the above proof, we estimated the slope of the function   $t \mapsto H^\Phi (P_t | Q)$ at $t= 0$ in terms of the ``worst case", i.e., the steepest possible descent; this  led  to the  square root of the  \textsc{Fisher} information,  by Theorem \ref {Steep_Desc_Gen}.  But Propositions  \ref{Prop8.3}, \ref{Proj_Gen} allow us     to calculate the slope of this function  with respect to the norm ${H^{-1}_\Theta (\cal S , \ell Q)},$ which induces the local Riemannian metric at $\ell = \ell_{0}$. We   obtain in this manner the following more precise result, in the spirit of \textsc{Otto  \& Villani} (2000),  \textsc{Cordero-Erausquin} (2002) or \textsc{Karatzas, Schachermayer \& Tschiderer} (2020).

 \begin{prop}
\label{HWIsharp} Under the assumptions of Theorem \ref{HWI}, suppose in addition that the curve $(P_t)_{0 \le t \le 1}$ is driven by a continuous function $(\psi_t)_{0 \le t \le 1}$ via the ``discrete continuity equation" 
\begin {equation}
\label{ciscont}
\,\partial \ell^\psi_t + \nabla \cdot ( \vartheta_{\ell_t} \nabla \psi_t ) = 0.
\end{equation}
 Then with $\,{\bm \ell} = \ell_0,$   we have the inequality 
\begin{equation} 
\label{p5sharp}
H^\Phi (P_0|Q) - H^\Phi(P_1|Q)\,\le \, \mathcal{W} (P_0,P_1)\, \bigg \langle    \,   \varphi( {\bm \ell } ) ,  \frac{  \psi_{0}  }{\, \big \|       \psi_{0}   \big \|_{\bH^1_\Theta(\cS, {\bm \ell } Q)} \,}   \bigg \rangle_{\bH^1_\Theta(\cS, {\bm \ell } Q)}    - \frac{\kappa}{2} \, \mathcal{W}^2  (P_0,P_1) .
\end{equation}
\end{prop}

\begin {proof} From (\ref{inner}), the slope of the function $H^\Phi (P_t | Q)$ with respect to  the norm ${H^{-1}_\Theta (\cal S , {\bm \ell } Q)}$ on $\mathcal{M}$,  which induces the local Riemannian metric at $(t, {\bm \ell } )=(0,  \ell_{0})$, is given by the bracket term on the right hand side of (\ref{p5sharp}). Hence, we may replace the inequality (\ref{ini1}) by the more precise equality
\begin{equation} 
g'(0) = - \bigg \langle    \,   \varphi( {\ell }_{0}) ,  \frac{  \psi_{0}  }{\, \big \|       \psi_{0}   \big \|_{\bH^1_\Theta(\cS, {\bm \ell } Q)} \,}   \bigg \rangle_{\bH^1_\Theta(\cS, {\bm \ell } Q)}.
\end{equation}
The rest of the proof of Theorem \ref{HWI} can be repeated verbatim, to  arrive as (\ref{p5sharp}) instead of  (\ref{p5}).       \end{proof}

\begin{rem}
\label{rem9.2}
 {\it What happens  when $P_0$ is on the boundary of $\mathcal M$, as in Remark \ref{rem5.1}? That is, when the set $ \,\mathcal N_0 = \{ x \in {\cal S}: P_0(x)=0\}$ is non-empty?} To be specific, let us  concentrate on the classical entropy  $\,\Phi (\ell)  = \ell \, \log \ell\,$. Then  the \textsc{Fisher}  information $ I^\Phi(P_0|Q)$ is infinite, and  the HWI inequality (\ref{p5}) holds trivially. On the other hand, the refined version (\ref{p5sharp}) may deliver some nontrivial information. 
 
 Indeed, suppose that $(P_t)_{0 \le t \le 1}$ is driven by a continuous function $(\psi_t)_{0 \le t \le 1}$via the ``discrete continuity equation" (\ref{ciscont}). If $\psi_0$ also vanishes on $\mathcal N_0\,,$   the bracket term in (\ref{p5sharp}) is finite (via the rule $ 0 \cdot \infty = 0$). As we assume that $\psi (\cdot)$ is continuous (actually, we only need this continuity at $t=0$), we can still apply the above argument and conclude that (\ref{p5sharp}) holds, yielding a nontrivial result. The geometric interpretation of   $\psi_0$ vanishing on $\mathcal N_0\,, $ is that the curve $(P_t)_{0 \le  t \le 1}$ starts ``tangentially to the boundary of $\mathcal M$", when departing from $P_0$ at this boundary.
\end{rem}

\section{Countable State-Space}
\label{sec9}

It is     well known that the results of Sections \ref{sec1} and \ref{sec2} hold also for   countably infinite state-spaces  ${\cal S}\,;$ see Chapters 2, 3 in \textsc{Norris} (1997) and \textsc{Liggett} (2010). In particular, the ergodic property (\ref{A20aa}) holds at least for bounded functions $\,f : {\cal S} \to \R\,.$  The crucial Proposition  \ref{FJ}    also remains  valid.

\smallskip
{\it Propositions \ref{DE_Traj}, \ref{DE} carry over to    countable state-spaces  under the  assumption}  $ V  \big(P(0)\,|\,Q\big) < \infty  $.  To see this, we start by observing  that we can guarantee now {\it prima facie} only the local martingale property of the processes $\,\widehat{M}\,$ in (\ref{A41a}).  Still, we can localize    $\,\widehat{M} \,   $  by an increasing sequence $  \big\{ \sigma_n \big\}_{n \in \N} \,$ of $\,\widehat{\mathbb{G}}-$stopping-times with values in $[0,T]$ and $ \, \lim_{n \to \infty} \uparrow  \sigma_n = T,$ and   create  the   bounded $(\widehat{\mathbb{G}}, \mathbb{Q})-$martingales  $\,\widehat{M}  (s \wedge \sigma_n)\,, ~ 0 \le s \le T.$  Taking expectations in (\ref{A41a}) with $\,s=\sigma_n\,$, \,then letting $n \to \infty$ and using monotone convergence, the  $ \mathbb{Q}-$submartingale property of   $\,\ell^2 \big( T-s, \widehat{X} (s) \big) \,, ~~ 0 \le s \le  T$  from Proposition \ref{FJ}, and optional sampling, we obtain  from  (\ref{E1}) the inequality 
$$
\mathbb{E^Q} \big[\ell^2 \big( T ,  X  (T) \big) \big]+ \int_0^T   
2 \, {\cal E} \big( \ell_t\,  ,  \ell_t  \,\big)\,
\ud t \,= \,\lim_{n \to \infty} \uparrow \mathbb{E^Q} \big[\ell^2 \big( T-\sigma_n ,  \widehat{X}  (\sigma_n) \big) \big] \,\le\, \mathbb{E^Q} \big[\ell^2 \big( 0 ,  X  (0) \big) \big].
$$
 But the reverse of this last inequality also holds, on account of  \textsc{Fatou}'s Lemma; thus (\ref{A41d}) follows for  countable state-spaces as well, and $\,\widehat{M} \,   $ is seen to be a true $(\widehat{\mathbb{G}}, \mathbb{Q})-$martingale. Then $\, \lim_{t \to \infty} V  \big(P(t)\,|\,Q\big)=0$, and with it \eqref{A41e}, are  proved for a countable state-space in the   manner of Proposition  \ref{to zero} below.

\subsection{Relative Entropy Dissipates   all the way down to Zero}
\label{sec101}

{\it When the  state-spaces ${\cal S}$ is countably infinite, the results of Section \ref{sec6}  pertaining to the relative entropy need    the additional assumption}   
  \begin{equation} 
  \label{Z0}
  H \big(P(0)\big| Q \big)  = \sum_{y \in {\cal S}} \, p (0,y) \, \log \left( \frac{\,p(0,y)\,}{q(y)}  \right)< \infty\,.
  \end{equation} 
Then everything   goes through as before, including  
 the non-negativity and decrease claims in (\ref{A30}) -- except for the argument establishing (\ref{A30'}), which uses the finiteness of the state-space in a crucial manner.

 Here  is a proof for this result in the countable case.

 \begin{prop}
 \label{to zero}
 The dissipation of relative entropy all the way down to zero, as in (\ref{A30'}), holds  for a countable state-space under the condition (\ref{Z0}).
 \end{prop}

\noindent
{\it Proof:} \, Let us recall  the likelihood ratio process $\, L (t) := \ell \big(t, X(t)\big),$ $ 0 \le t < \infty\,$ of (\ref{A34}), and  from (\ref{A40}) that its time-reversal $\, L (T-s)  , ~ 0 \le s \le  T\,$ is a $\, \big( \widehat{\mathbb{G}}, \mathbb{Q}\big)-$martingale.

\smallskip
Fix $\, 0 \le t_1 < t_2 < \infty\,$. For any $\, T \in (t_2 , \infty)\,$, this means $\, \mathbb{E^Q} \big[ L (T-s_1) \, \big| \, {\cal G} (T-s_2 )\big] = L (T-s_2)\,$ for $s_1 = T - t_1,$ $~s_2 = T - t_2,$ or equivalently: 
$$\,
\mathbb{E^Q} \big[ L (t_1) \, \big| \, \sigma \big( X (\theta) ,\,  \,   t_2  \le \theta \le T \big)  \big] \,= \,L (t_2)\,.\,
$$ 
But this last identity holds for any $\, T \in (t_2 , \infty)$, so it leads --- on the strength of the P.\,\textsc{L\'evy} martingale convergence Theorem 9.4.8 in \textsc{Chung} (1974) --- to 
\begin{equation} 
\label{Z6}
\mathbb{E^Q} \big[ L (t_1) \, \big| \, {\cal H } (t_2)  \big] = L (t_2)\,, \qquad ~~~{\cal H } (t) := \sigma \big( X (\theta) ,\, \,  t  \le \theta < \infty \big).
\end{equation} 
To wit, the likelihood ratio process $\, \big( L(t) \big)_{0 \le t < \infty }\,$ is a  martingale of the backwards filtration $\, \big( {\cal H}(t) \big)_{0 \le t < \infty }\,,$ whose ``tail" sigma-algebra is trivial on account of the ergodicity property (\ref{A20aa}) of the \textsc{Markov} Chain (\textsc{Blackwell  \& Freedman}  (1964)):
$$
{\cal H}(\infty) \,:=\, \bigcap_{0 \le t < \infty} {\cal H}(t) \,=\, \big\{ \emptyset, \Omega \big\}\,, \quad \text{mod.} ~ \mathbb{Q}\,.
$$

We invoke now the martingale version of the  backward submartingale convergence Theorem 9.4.7 in \textsc{Chung}  (1974). It follows from this result   that $\, \big( L(t) \big)_{0 \le t < \infty }\,$ is a $\,\mathbb{Q}-$uniformly integrable family; that the limit
$\,
L(\infty) :=  \lim_{t \to \infty} L(t) 
\,$
exists, both a.e.\,and in $\,\mathbb{L}^1$ under $\mathbb{Q}\,;$  and that the backward martingale property (\ref{Z6}) extends all the way to infinity,  namely 
\begin{equation} 
\label{Z9}
\mathbb{E^Q} \big[ L (t_1 ) \, \big| \, {\cal H } (\infty)  \big] = L (\infty)\,.
\end{equation} 
But the triviality under $\,\mathbb{Q} $ of the tail sigma-algebra, implies that $\,L(\infty) \,$ is $\,\mathbb{Q}-$a.e.\,constant. Then the extended martingale property (\ref{Z9}) identifies this constant as 
$\,
L (\infty) = \mathbb{E^Q} \big[L (\infty)  \big]  = \mathbb{E^Q} \big[L (t_1)  \big] = 1\,.
 $

We recall   the relative entropy from (\ref{A35*}). The convexity of the function $\, \Phi (\ell) = \ell \, \log \ell\,$ shows, in conjunction with (\ref{Z6}) and the \textsc{Jensen} inequality, that 
\begin{equation} 
\label{Z11'}
\Big( \Phi \big( L (t) \big), {\cal H } (t) \Big)_{0 \le t < \infty} \qquad \text{is a backward} ~~ \mathbb{Q}-\text{submartingale,} 
\end{equation} 
with decreasing expectation $\, \mathbb{E^Q} \big[  \Phi \big( L( t )  \big]= H \big( P(t) \, \big| \, Q \big) \ge 0.\,$ Because this expectation is bounded from below, we can   appeal  once again  to the backward submartingale convergence  Theorem 9.4.7 in \textsc{Chung}  (1974). We deduce that the process in (\ref{Z11'}) is a $\,\mathbb{Q}-$uniformly integrable family which converges, again both a.e.\,and in $\,\mathbb{L}^1$ under $\mathbb{Q}\,,$ to 
$\,
\lim_{t \to \infty} \Phi \big( L(t) \big)  =\Phi \big( L (\infty) \big) = \Phi(1) = 0\,.
 $

 Furthermore, the aforementioned uniform integrability 
gives
 $$
\lim_{t \to \infty} \downarrow H \big(P(t) \, \big| \,Q\big) \,=\, 
\lim_{t \to \infty}  \, \mathbb{E^Q} \big[  \Phi \big( L( t ) \big)  \big]\,=\, \mathbb{E^Q} \Big( \lim_{t \to \infty}  \Phi \big( L( t ) \big)   \Big) \,=\, 0\,;
$$
that is, (\ref{A30'}) is also valid  in this general case with countable state-space.   \qed

\subsubsection{Relative Entropy is Continuous  at the Origin}
\label{sec62}

We discuss now the validity of  the \textsc{de Bruijn}  identities of (\ref{A46})  when the state-space is countable.

\begin{prop}
The \textsc{de Bruijn}  identities of (\ref{A46})  for the dissipation of relative entropy are valid  for a countable state-space,  under the finite entropy condition (\ref{Z0}).
\end{prop}

 To justify this claim, we would like to   use the argument already deployed; but there is now no obvious, general way to turn the local martingale $\,\widehat{M}^{\, h}   $ of (\ref{A45}) into a true $\mathbb{Q}-$martingale.   Thus,  we localize   $\,\widehat{M}^{\, h}   $  by an increasing sequence $  \big\{ \sigma_n \big\}_{n \in \N} \,$ of $\,\widehat{\mathbb{G}}-$stopping-times with values in $[0,T]$ and $ \, \lim_{n \to \infty} \uparrow  \sigma_n = T.$   In  this manner we create the {\it bounded} $(\widehat{\mathbb{G}}, \mathbb{Q})-$martingales  $\,\widehat{M}^{\, h} (s \wedge \sigma_n)\,, ~ 0 \le s \le T,$  which then give 
\begin{equation} 
\label{Z01}
\mathbb{E^Q}
\int_0^{\sigma_n} \big( \partial h + \widehat{{\cal K}} h \big) \big(u, \widehat{X} (u) \big)\, \ud u\,=\, \mathbb{E^Q} \big[ h \big( \sigma_n, \widehat{X}(\sigma_n) \big) \big] - \,\mathbb{E^Q} \big[ h \big( 0, \widehat{X}(0) \big) \big]  ~~~~~~~~~~~~~~~~~~~~~~~~~~~~~~~~~~~~~~~~~~~~~
\end{equation} 
$$
~~~~~~~~~~~~=\,  H \big( P(T-\sigma_n) \, \big| \, Q \big) -  H \big( P(T) \, \big| \, Q \big) \, \le \, H \big( P(0 ) \, \big| \, Q \big) -  H \big( P(T) \, \big| \, Q \big) \, \le \, H \big( P(0 ) \, \big| \, Q \big) \, < \, \infty 
$$
 for every $\, n \in \N,$ on account of  (\ref{Z0}); see also the argument straddling (\ref{Z03}) below. In particular, the sequence of real numbers  in (\ref{Z01}) takes values in the compact   interval $\, [ - H  ( P(0 ) \,  | \, Q  ), H  ( P(0 ) \,  | \, Q ) ] .$

 We  would like  now to let $ n \to \infty$ in (\ref{Z01}), and establish  the \textsc{de Bruijn}  identity (\ref{A46}) in this case. The issue   once again is  continuity of the relative entropy   --- though now {\it at the origin} (rather than at infinity, as in (\ref{A30'})); and not along fixed times, but rather along an appropriate sequence of stopping times,  i.e.,
\begin{equation} 
\label{Z04}
\lim_{n \to \infty} \uparrow H \big( P(T- \sigma_n) \, \big| \, Q \big)\,=\, H \big( P(0) \, \big| \, Q \big)\,.
\end{equation} 
Accepting this for a moment, and letting $\, n \to \infty\,$ in (\ref{Z01}), we obtain the \textsc{de Bruijn}  identity (\ref{A46}), i.e., 
\begin{equation} 
\label{Z06}
\int_0^T I (t) \, \ud t\,=\,\mathbb{E^Q} \int_0^{T} \Big( \partial h + \widehat{{\cal K}} h \Big) \big(u, \widehat{X} (u) \big) \,  \ud u \,=\, H \big( P(0) \, \big| \, Q \big)  -H \big( P(T) \, \big| \, Q \big)  
\end{equation}
 by monotone convergence. We let now   $  T \to \infty\,$ in (\ref{Z06}) and arrive at the second identity in (\ref{A46}), thanks to the property (\ref{A30'})   already established in Proposition  \ref{to zero}.

\medskip
\noindent
{\it Proof of   (\ref{Z04}):}  By analogy with (\ref{A40d}), and invoking now additionally the optional sampling theorem for the bounded stopping times $  \big\{ \sigma_n \big\}_{n \in \N} \,$ of $\,\widehat{\mathbb{G}}\,$  with values in $[0,T]$, we deduce that the sequence of non-negative real numbers
\begin{equation} 
\label{Z03}
H \big( P (T-\sigma_n) \,| \,Q)\,=\, \mathbb{E^Q} \big[ \Phi \big( \ell \big( T-\sigma_n, \widehat{X} (\sigma_n) \big) \big) \big]\,,~~~~~~ ~n \in \N
\end{equation} 
is increasing; in particular,   
$ 
 \,\lim_{n \to \infty} H \big( P(T-\sigma_n) \, \big| \, Q \big)   \le     H \big( P(0) \, \big| \, Q \big) .  
 $ 
   On the other hand, the boundedness-from-below of the function $\,\Phi (\ell) = \ell \,\log \ell\, $ gives   
$$
  \lim_{n \to \infty}  H \big( P(T-\sigma_n) \, \big| \, Q \big) 
 \, \ge \,  \mathbb{E^Q} \Big[  \lim_{n \to \infty} \, \Phi
 \big( \ell \big( T-\sigma_n, \widehat{X} (\sigma_n) \big) \big) \Big]\,=\, \mathbb{E^Q} \big[ \Phi \big( \ell \big( 0, X (0) \big) \big) \big]\,=\,  H \big( P(0) \, \big| \, Q \big) 
$$
  with the help of \textsc{Fatou}'s Lemma,  and (\ref{Z04}) follows.  \qed

  \begin{rem} {\it The General Case:} 
  Exacly the same methods show that the results of Propositions \ref{Prop_8.1} and \ref{Prop_Gen_de_Br}, pertaining  to a general  convex function $\, \Phi : (0, \infty) \to \R \,$ with the properties imposed there, continue to hold for the generalized relative entropy functional  of (\ref{Z13}) in the case of a countable state-space ${\cal S},$  under the condition $\,   H^\Phi \big( P(0)     |  Q \big) <  \infty\,.$ 
  
  \smallskip
{\it  Once again, it is important to stress that  nowhere in the present Section   have we invoked the detailed-balance  conditions of    (\ref{A24}). }
  \end{rem}

\section*{Bibliography}

\noindent
 \textsc{Ambrosio, L., Gigli, N. \& Savar\'e, G.}  (2008)    {\it  Gradient Flows in Metric Spaces and in the Space of Probability Measures.} Second Edition.   {\sl Lectures in Mathematics, ETH Z\"urich.} Birh\"auser Verlag, Basel.

\medskip 
\noindent    \textsc{Blackwell, D. \& Freedman, D.}  (1964) The tail $\sigma-$field of a Markov chain and a theorem of Orey. {\it Annals of Mathematical Statistics} {\bf 35}, 1291-1295. 

\medskip
\noindent    
\textsc{Bobkov, S.G. \& Tetali, P.}  (2006) Modified logarithmic Sobolev inequalities in discrete settings. {\it Journal of Theoretical Probability} {\bf 19},  289-336. 

\medskip
\noindent    \textsc{Caputo, P., Dai\,Pra, P. \& Posta, G.}  (2009) Convex entropy decay via the Bochner-Bakry-\'Emery approach. {\it Annales de l'  Institut Henri Poincar\'e (S\'er. B, Probabilit\'es et  Statistiques)} {\bf 45},  734-753.

\medskip
\noindent    \textsc{Chafa\"i, D.}  (2004) Entropies, convexity and functional inequalities. {\it Journal of Mathematics  Ky\^oto University} {\bf 42},  325-363.

\medskip
\noindent    \textsc{Chung, K.L.}  (1974)\,
 {\it A Course in Probability Theory.} Second Edition, Academic Press, New York.

\medskip
\noindent
 \textsc{Conforti, G.}  (2020) A probabilistic approach to convex ($\phi$)-entropy decay for Markov chains. Preprint, available at \, {\it https://arxiv.org/abs/2004.10850}
 
 \medskip
\noindent
 \textsc{Cordero-Erausquin, D.} (2002) Some applications of mass transport to Gaussian-type inequalities. {\it Archive for Rational Mechanics and Analysis} {\bf  161}, 257-269.

  \medskip
\noindent 
 \textsc{Courant, R., Friedrichs, K. \& Lewy, H.} (1928) \"Uber die partiellen Differenzellengleichungen der mathematiscen Physik. {\it Mathematische Annalen} {\bf 100}, 32-74. 
 
 \medskip
\noindent 
 \textsc{Cover, T.M. \& Thomas, J.A.} (1991) {\it Elements of  Information Theory.} J. Wiley \& Sons, New York.

\medskip
\noindent 
 \textsc{Diaconis, P. \& Saloff-Coste, L.} (1996)  Logarithmic Sobolev inequalities for finite Markov chains. {\it  Annals of Applied Probability} {\bf 6}, 695-750.
 
 \medskip
\noindent 
 \textsc{Dietert, H.} (1996)  Characterization of gradient flows for finite-state Markov chains. {\it Electronic Communications in Probability} {\bf 20}, no. 29, 1-8.
 
 \medskip
\noindent 
 \textsc{Erbar, M. \& Maas, J.} (2012)  Ricci curvature of finite Markov Chains via convexity of entropy. {\it  Archive for Rational Mechanics and Analysis} {\bf 206}, 997-1038.

 \medskip
\noindent 
 \textsc{Erbar, M. \& Maas, J.} (2014)  Gradient flow structures for discrete porous medium equations. {\it  Discrete and Continuous Dynamical Systems} {\bf 34}, 1355-1374.

\medskip
\noindent 
\textsc{Fontbona, J. \&    Jourdain, B.} (2016) A trajectorial interpretation of the dissipations of entropy and Fisher information for stochastic differential equations. {\it  Annals of Probability} {\bf 44}, 131-170.

\medskip
\noindent 
\textsc{Jordan,  R., Kinderlehrer, D. \& Otto, F.} (1998) The variational formulation of the Fokker-Planck equation. {\it SIAM Journal of Mathematical Analysis} {\bf 29}, 1-17.

\medskip
\noindent
 \textsc{Karatzas, I., Schachermayer, W. \& Tschiderer, B.}  (2019) Trajectorial  Otto   Calculus. Preprint (63 pages), preliminary version of [KST\,20]. arxiv:1811.08686.
 
 \medskip
\noindent
 \textsc{Karatzas, I., Schachermayer, W. \& Tschiderer, B.}  (2020) 
 {A trajectorial approach to the gradient flow properties of Langevin-Smoluchowski diffusions}. Condensed version of [KST\,19] (37 pages). Submitted,  	arxiv:2008.09220.

 \medskip
 \noindent    \textsc{Karatzas, I. \& Shreve, S.E.}  (1988)  {\it Brownian Motion and Stochastic Calculus.}  Volume 113 of   series  {\sl Graduate Texts in Mathematics.}   Springer-Verlag, New York.

\medskip
\noindent
 \textsc{Liggett, T.G.}  (2010)    {\it  Continuous Time Markov Processes: An Introduction.} Volume 113 of the series  {\sl Graduate Studies in Mathematics.} American Mathematical Society, Providence, RI.
 
  \medskip
\noindent 
 \textsc{Maas, J.} (2011)  Gradient flows of the  entropy for finite Markov chains. {\it  Journal of Functional Analysis} {\bf 261}, 2250-2292.
 
   \medskip
\noindent 
 \textsc{Maas, J.} (2017)  Entropic Ricci curvature for discrete spaces. {\it  Lecture Notes in Mathematics}\, {\bf 2184}, 159-173. Springer-Verlag, Berlin.
 
 \medskip
 \noindent  
 \textsc{Miclo, L.} (1992) Recuit simul{\'e} sans potentiel sur un ensemble fini. 
 S{\'e}minaire de Probabilit{\'e}s XXVI,
 {\it  Lecture Notes in Mathematics}\, {\bf 1526}, 47–60. Springer-Verlag, Berlin

   \medskip
\noindent 
 \textsc{Mielke, A.} (2011) A gradient structure for reaction-diffusion systems and for energy-drift-diffusion systems. {\it  Nonlinearity} {\bf 24}, 1329-1346.
 
  \medskip
\noindent 
 \textsc{Mielke, A.} (2013) Geodesic convexity of the relative entropy in reversible  Markov  chains. {\it  Calculus of Variations and Partial Differential Equations} {\bf 48}, 1-31.
 
\medskip
\noindent 
 \textsc{Mielke, A.} (2016) On evolutionary $\Gamma-$convergence for gradient systems. In {\it Macroscopic and Large Scale Phenomena: Coarse Graining, Mean-Field Limits, and Ergodicity}. {\sl Lecture  Notes in Applied  Mathematics and  Mechanics} {\bf 3}, 187-249. Springer-Verlag, New York.

 \medskip
\noindent
 \textsc{Montenegro, R. \&   Tetali, P.} (2006)  Mathematical aspects of mixing times in Markov chains.  {\it  Foundations and Trends  in Theoretical Computer Science} {\bf 1}\,(3),  237-354.

\medskip
\noindent 
\textsc{Norris, J.}  (1997)\,  {\it  Markov Chains.} Cambridge University Press.

 \medskip
\noindent
 \textsc{Otto, F.} (2001)  The geometry of dissipative evolution equations: the porous medium equation.  {\it  Communications in Partial Differential Equations} {\bf 26},  101-174.
 
  \medskip
\noindent
 \textsc{Otto, F. \& Villani, C.} (2000) 
 Generalization of an inequality by Talagrand, and links with the logarithmic Sobolev inequality. 
 {\it Journal of  Functional Analysis} {\bf 173}, 361-400.

 \medskip
 \noindent
  \textsc{Pavon, M.} (1989) Stochastic control and nonequilibrium thermodynamical systems. {\it Applied Mathematics \& Optimization} {\bf 19}, 187-202. 
 
\medskip
\noindent 
\textsc{Rogers, L.C.G. \& Williams, D.}  (1987)\,  {\it  Diffusions, Markov Processes and Martingales. Vol.\,II: It\^o Calculus.} J.\,Wiley \& Sons, New York.

\medskip
\noindent 
\textsc{Stam, A.J.} (1959) Some inequalities satisfied by the quantities of information of Fisher and Shannon. {\it Information and Control} {\bf 2}, 101-112.

 \end{document}